\documentclass[11pt]{amsart}
\usepackage[
text={450pt,575pt},
headheight=9pt,
centering
]{geometry}

\usepackage{latexsym}
\usepackage[all,cmtip]{xy}
\usepackage{times}
\usepackage{braket}
\usepackage[english]{babel}
\usepackage[utf8]{inputenc}
\usepackage{paralist,url,verbatim}
\usepackage{amscd}
\usepackage[T1]{fontenc}
\usepackage{xspace}
\usepackage{comment}
\usepackage{amsmath,amsthm,amssymb,amsfonts}
\usepackage{mathtools}
\usepackage{xfrac}
\usepackage{enumitem}
\usepackage{wasysym}
\usepackage{float}
\usepackage{footmisc}
\usepackage[usenames,dvipsnames]{xcolor}
\usepackage[bookmarksopen=true]{hyperref}
\definecolor{darkblue}{RGB}{0,0,160}
\definecolor{darkgreen}{RGB}{0,80,0}
\definecolor{colsemigr}{rgb}{0.55, 0.71, 0.0}
\definecolor{colsubsgr}{rgb}{0.0, 0.5, 1.0}
\definecolor{colboundr}{rgb}{1.0, 0.01, 0.24}
\hypersetup{
colorlinks,%
citecolor=black,%
filecolor=black,%
linkcolor=darkblue,%
urlcolor=darkblue,%
}
\setlist[enumerate]{labelindent=0pt, labelwidth =0pt, itemindent=*,leftmargin=0pt, font=\upshape}

\newcommand{\ktrash}[1]{}
\newcommand{\matej}[1]{ \textcolor{Lavender}{\sf [$\diamondsuit$ #1]}}
\newcommand{\change}[1]{}
\definecolor{cKlaus}{rgb}{0.15,0.40,0.03}
\definecolor{cKLinkGBU}{rgb}{1,0,0}  % Good Bad Ugly - the NEW Macros
\definecolor{cKLink}{rgb}{0,0,0} %{0.6,0.2,0.3}
\definecolor{cALink}{rgb}{0,0.3,0}

\usepackage[textsize=small]{todonotes}
\usepackage{tikz}
\usepackage{tikz-cd}
\usetikzlibrary{decorations.markings,intersections,positioning,calc}

  \tikzset{mylabel/.style  args={at #1 #2  with #3}{
    postaction={decorate,
    decoration={
      markings,
      mark= at position #1
      with  \node [#2] {#3};
 } } } }

\theoremstyle{plain}
\newtheorem{theorem}{Theorem}[section]

\newtheorem{corollary}[theorem]{Corollary}
\newtheorem{lemma}[theorem]{Lemma}
\newtheorem{proposition}[theorem]{Proposition}

\theoremstyle{definition}
\newtheorem{question}[theorem]{Question}
\newtheorem{remark}[theorem]{Remark}
\newtheorem{definition}[theorem]{Definition}

\newtheorem{example}[theorem]{Example}

\newtheorem{convention}[theorem]{Convention}
\newtheorem{notation}[theorem]{Notation}

\numberwithin{equation}{section}

\newcommand{\hmm}[1]{\ifnum\ifhmode\spacefactor\else2000\fi>1000 \uppercase{#1}\else#1\fi}%Does something to \aextension

\newcommand{\aed}{{d}} %the proper notation for edges
\newcommand{\aextend}{\hmm{e}xtend} %maybe we want a new word for this
\newcommand{\aextension}{\hmm{e}xtension} %maybe we want a new word for this
\newcommand{\ahb}{{\color{cKLink}\mathfrak{c}}} %Hbasis element

\newcommand{\aold}[1]{}

\newcommand{\awt}[1]{\ensuremath{\widetilde{#1}}}
\newcommand{\A}{\ensuremath{\mathbb{A}}}
\newcommand{\bfm}{\ensuremath{\mathbf{m}}}

\newcommand{\bfs}{\ensuremath{\mathbf{s}}}
\newcommand{\bft}{\ensuremath{\mathbf{t}}}

\newcommand{\bound}{{\partial}} % Projektion auf Rand
\newcommand{\card}[1]{\lvert #1\rvert}
\newcommand{\CC}{\ensuremath{\mathbb{C}}}
\newcommand{\cedges}{\operatorname{Edge}^{\operatorname{c}}}

\newcommand{\ceil}[1]{\lceil #1 \rceil}
\newcommand{\CE}{\ensuremath{{\mathcal E}}}
\newcommand{\CF}{\ensuremath{{\mathcal F}}}

\newcommand{\cocartesian}{{co-Cartesian}} % co-Cartesian: maybe we should change it to flat \aextension, deformation%

\newcommand{\cone}{\operatorname{cone}}
\newcommand{\conv}{\operatorname{conv}}

\newcommand{\dep}{\operatorname{Dep}}

\newcommand{\dual}{^{\scriptscriptstyle\vee}}

\newcommand{\etaFrac}[1]{{\color{cKLink}\left\{#1\right\}}}
\newcommand{\face}{{\operatorname{face}}}

\newcommand{\geta}{\ensuremath{\diamondsuit}} % generic notation for one of  eta, etaZ, teta, tetaZ.

\newcommand{\gH}{\operatorname{H}}
\newcommand{\height}{{\lambda}} % Hoehe in T
\newcommand{\hideLatticeVert}[1]{}
\newcommand{\Hom}{\operatorname{Hom}}
\newcommand{\id}{\operatorname{id}}
\newcommand{\im}{\mbox{\rm Image}}

\newcommand{\innt}{\operatorname{int}}
\newcommand{\kbb}{{\kss \bullet}}
\newcommand{\kbox}{\emph}
\newcommand{\kbp}{{\color{cKLink}b_\partial}}
\newcommand{\kcP}{{\color{cKLink}\sigma}} % cone over \kP
\newcommand{\kc}{{\color{cKLink}c}}

\newcommand{\kellp}{{\color{cKLink}\kell_\partial}} % \ell on \partial
\newcommand{\kells}{{\color{cKLink}\kell_s}} % image of s_v
\newcommand{\kellt}{{\color{cKLink}\kell_t}} % image of t_E
\newcommand{\kellz}{{\color{cKLink}\kell^{\mbox{\rm\tiny st}}}} % stabilization of \kellp
\newcommand{\kell}{{\color{cKLink}\ell}} % initial object map
\newcommand{\kenditem}{\vspace{-1ex}\end{itemize}}

\newcommand{\ketaZ}{{\color{cKLink}\eta_{\ZZ}}}
\newcommand{\keta}{{\color{cKLink}\eta}}

\newcommand{\kF}{{F}} % common factor for trivial deformations

\newcommand{\kiota}{{\iota}} % inclusion \kT -> \kS 
\newcommand{\kitem}{\begin{itemize}\vspace{-2ex}}

\newcommand{\kL}{{L}} % the section ftom T to \til{T}
\newcommand{\kMC}{{C}} % C(P)
\newcommand{\kMtC}{{\tC}} % tilde C(P)
\newcommand{\kMTD}{{\mathcal T^*}} % dual T(P)
\newcommand{\kMTPD}{{\mathcal T_{+}\dual}} % pos in T(P)*

\newcommand{\kMTPZ}{{\color{cKLink}\mathcal T_{+\ZZ}}}
\newcommand{\kMTP}{{\mathcal T_{+}}} % positive cone in T(P)

\newcommand{\kMTtP}{{\wt{\mathcal T}_+}} % tautological cone
\newcommand{\kMTtZD}{{\color{cKLink}\wt{\mathcal T}_{\ZZ}^*}}
\newcommand{\kMTZD}{{\mathcal T_{\ZZ}^*}} % dual lattice T(P)*
\newcommand{\kMTZ}{{\mathcal T_{\ZZ}}} % lattice in T(P)
\newcommand{\kMT}{{\mathcal T}} % T(P) = V(P) + W(P)
\newcommand{\kMVZ}{{C^{\mbox{\tiny lin}}_{\ZZ}}} % lattice in V(P)
\newcommand{\kMV}{{C^{\mbox{\tiny lin}}}} % V(P) = ambient von C(P)
\newcommand{\kodSp}{{\rho}} % Kodaira-Spencer map
\newcommand{\ko}{\overline}

\newcommand{\kpiD}{{e}} % dual extension \kTD -> \ktTD
\newcommand{\kpi}{{\pi}} % extension \ktS -> \kS
\newcommand{\kPm}{{m}} % #{vertices in \kP}
\newcommand{\kpr}{{p}} % projection \kSD ->> \kTD
\newcommand{\kPr}{{r}} % #{compact edges in \kP}
\newcommand{\kp}{{\color{cKLink}R}} % ray in \kcPD
\newcommand{\kP}{{P}} % the main polyhedron (used to be Q)

\newcommand{\kqr}{{q}} % projection \kSpanS ->> \kQuot
\newcommand{\kQuotD}{{N}} % V/W-dual
\newcommand{\kQuotRD}{{N_{\RR}}}
\newcommand{\kQuotR}{{M_{\RR}}} % M or Q \otimes \RR
\newcommand{\kQuotSup}{{\color{cKLink}M_{\Z\Z}}} % M = Q = \kSpanS/\kSpanT = V/W
\newcommand{\kQuotS}{{\ko{\kS}}} % image of \kS in \kQuot
\newcommand{\kQuot}{{M}} % M = Q = \kSpanS/\kSpanT = V/W

\newcommand{\kR}{{R}} % idegree of the deformation 
\newcommand{\kSD}{{\kS\dual}} % dual of \kS

\newcommand{\kSpanSD}{{\kSpanS^*}}
\newcommand{\kSpanS}{{V}} % V ambient vector spaces of cone \kS
\newcommand{\kSpanTD}{{\kSpanT^*}}
\newcommand{\kSpanT}{{W}} % W ambient vector spaces of cone \kT

\newcommand{\kss}{\scriptscriptstyle}
\newcommand{\kst}{\,:\;}
\newcommand{\ksum}{{a}} % a: T x Rand(S) -> S
\newcommand{\ks}{\scriptstyle}
\newcommand{\kS}{{S}} % T<S straight pair
\newcommand{\kTD}{{\kT\dual}} % dual of \kT
\newcommand{\ktetaZ}{{\color{cKLink}\wt{\eta}_{\ZZ}}}
\newcommand{\kteta}{{\color{cKLink}\wt{\eta}}}
\newcommand{\ktqr}{{\wt{q}}} % projection \ktSpanS ->> \ktQuot
\newcommand{\ktQuot}{{\wt{M}}} % \tilde Q 

\newcommand{\ktsum}{{\wt{a}}} % same for \ktT and \ktS
\newcommand{\kts}{{\wt{s}}} % elements of \ktS
\newcommand{\ktS}{{\wt{S}}} % liftings of T<S 
\newcommand{\kttpiS}{{\pi'_{\kS}}} % extension \kttS -> \kS
\newcommand{\kttpiT}{{\pi'_{\kT}}} % extension \kttT -> \kT
\newcommand{\kttpi}{{\pi'}} % extension \kttS -> \kS
\newcommand{\kttS}{{\wt{S}'}} % other liftings of T<S 
\newcommand{\kttT}{{\wt{T}'}} %
\newcommand{\ktT}{{\wt{T}}} %
\newcommand{\ktt}{{\wt{t}}} % elements of \ktT
\newcommand{\kt}{{\color{cKLink}t}} % elements of \kT
\newcommand{\kT}{{T}} %
\newcommand{\ku}{\underline}
\newcommand{\lan}{\langle}
\newcommand{\largestint}[1]{\lfloor #1 \rfloor}
\newcommand{\la}{\ensuremath{\lambda}}

\newcommand{\MinkFam}{{\CF}} % category of MinkFams
\newcommand{\minks}{{k}} % number of Minkowski summands

\newcommand{\neht}{\Leftarrow}
\newcommand{\NN}{\ensuremath{\mathbb{N}}}
\newcommand{\normal}{{\mathcal N}}

\newcommand{\oneone}{{(\ku{1};\ku{1},\ku{0})}} % \ku{1} T(P)
\newcommand{\one}{{\ku{1}}} % \ku{1} in C(P)

\newcommand{\opp}{\operatorname{opp}}
\newcommand{\orient}{\varepsilon}
\newcommand{\pr}{\operatorname{pr}}
\newcommand{\QQ}{\ensuremath{\mathbb{Q}}}
\newcommand{\Q}{\ensuremath{\mathbb{Q}}}

\newcommand{\Rand}[2]{\ensuremath{\partial_{#1}{ #2} }}
\newcommand{\rank}{\operatorname{rank}}
\newcommand{\ran}{\rangle}
\newcommand{\rounddown}[1]{\largestint{#1}}
\newcommand{\roundup}[1]{\smallestint{#1}}
\newcommand{\RR}{\ensuremath{\mathbb{R}}}

\newcommand{\se}{{short}} % "short" edge
\newcommand{\sigmaDBd}{{\color{cKLink}\Rand{(\RR_{\geqslant 0}R)}{\kcP\dual}}}

\newcommand{\smallestint}[1]{\lceil #1 \rceil}

\newcommand{\Spec}{\operatorname{Spec}}

\newcommand{\straightExt}{{\CE}} % category of straight exts
\newcommand{\straightness}{freeness}
\newcommand{\Straightness}{Freeness}
\newcommand{\straight}{free}
\newcommand{\Straight}{Free}

\newcommand{\surj}{\twoheadrightarrow}%{\longrightarrow\hspace{-1.5em}\longrightarrow}
\newcommand{\tail}{\operatorname{tail}}
\newcommand{\tC}{\wt{C}} % 
\newcommand{\then}{\Rightarrow}

\newcommand{\too}{\ensuremath{\longrightarrow}}
\newcommand{\toric}{{\mathbb T \!\mathbb V}}
\newcommand{\tsigma}{{\tilde{\sigma}}}

\newcommand{\tX}{{{\widetilde{X}}}}

\newcommand{\vastP}{{\ensuremath{{v'_\ast}}}} % \vast'
\newcommand{\vast}{{\ensuremath{{v_\ast}}}} % reference vertex
\newcommand{\vect}[1]{#1}%\,{\scriptstyle \overrightarrow{{{\hspace{-1pt}#1}}}}}

\newcommand{\VrtxnZ}{\operatorname{Vert}_{\notin\mathbb{Z}}}
\newcommand{\VrtxZ}{\operatorname{Vert}_{\in\mathbb{Z}}}
\newcommand{\Vrtx}{\operatorname{Vert}}

\newcommand{\wteta}{\kteta}

\newcommand{\wt}[1]{\widetilde{#1}}
\newcommand{\ZZ}{\ensuremath{\mathbb{Z}}}
\newcommand{\Z}{\ensuremath{\mathbb{Z}}}

\newcommand{\givenpair}{\ensuremath{(\NN,\cone_\ZZ(\kP)\dual)}}
\newcommand{\givenpairXYmat}{\NN \ar@{^{(}->}[r] &  \cone_\ZZ(\kP)\dual}

\newcommand{\spann}[2]{%
  \ifx\relax\noexpand#1\relax  %% if #1 is empty then it does not leave a blank space
  \operatorname{span}\!\left\{#2\right\}
  \else
  \operatorname{span}_{#1}\left\{#2\right\}
  \fi
  }

\renewcommand{\star}{\operatorname{star}}

\renewcommand{\ker}{\mbox{\rm ker}\,}

\begin{document}
\title{Polyhedra, lattice structures, and extensions of semigroups} 
% \author{Klaus Altmann \and Alexandru Constantinescu \and Matej Filip}

\author[K.~Altmann]{Klaus Altmann%
}
\address{Institut f\"ur Mathematik,
FU Berlin,
Arnimalle 3,
D-14195 Berlin,
Germany}
\email{altmann@math.fu-berlin.de}
\author[A.~Constantinescu]{Alexandru Constantinescu}
\address{Institut f\"ur Mathematik,
FU Berlin,
Arnimalle 3,
D-14195 Berlin,
Germany}
\email{aconstant@zedat.fu-berlin.de}
\author[M.~Filip]{Matej Filip}
\address{Laboratory of applied mathematics and statistics,
% I. Department of Fundamentals of Electrical Engineering, 
% Mathematics and Physics
% Faculty of Electrical Engineering,
University of Ljubljana,
Tr\v{z}ai\v{s}ka cesta 25,
SI-1000 Ljubljana,
Slovenia}
\email{matej.filip@fe.uni-lj.si}
\thanks{%
         }

\begin{abstract}
  For an arbitrary rational polyhedron
  we consider its decompositions into Minkowski summands and, dual to this,
  the \straight~ extensions of the associated pair of semigroups.
  Being \straight~ for the pair of semigroups is equivalent to flatness for the corresponding algebras.
  Our main result is phrased in this dual setup:
  the category of \straight~ extensions always contains an initial object, which we describe explicitly.
  These objects seem to be related to unique liftings in log geometry.
  Further motivation  comes from the deformation theory of the associated toric singularity.
\end{abstract}
% \date{}
\maketitle

\setcounter{tocdepth}{1}
\tableofcontents

%%%%%%%%%%%%%%%%%%%%%%%%%%%%%%%%%%%%%%%%%%%%%%%%%%%%%%%%%%%%%%%%%%%%%%%%%%%
%%%%%%%%%%%%%%%%%%%%%%%%%%%%%%%%%%%%%%%%%%%%%%%%%%%%%%%%%%%%%%%%%%%%%%%%%%%
%%%%%%%%
\section{Introduction}\label{intro}
%%%%%%%%
%%%%%%%%%%%%%%%%%%%%%%%%%%%%%%%%%%%%%%%%%%%%%%%%%%%%%%%%%%%%%%%%%%%%%%%%%%%
%%%%%%%%%%%%%%%%%%%%%%%%%%%%%%%%%%%%%%%%%%%%%%%%%%%%%%%%%%%%%%%%%%%%%%%%%%%
The present paper deals entirely with objects from discrete mathematics, 
like semigroups, convex polyhedra and polyhedral cones,
and their relations to lattice points.
Nevertheless, most of the motivation comes from algebraic deformation theory
of affine toric varieties. 
%%%%%%%%%%%%%%%%%%%%%%%%%%%%%%%%%%%%%%%%%%%%%%%%%%%%%%%%%%%%%%%%%%%%%%%%%%%
\subsection{Minkowski sums of polyhedra}
\label{minkPolyhedra}
A central notion of the present paper is the Minkowski sum of two convex
polyhedra $A,B$ in some real vector space $\kQuotRD\cong\RR^d$. 
It is simply defined as
\[
A+B:=\{a+b\kst a\in A,\, b\in B\}.
\]
It is easy to see that the result is again a convex polyhedron.
In analogy to this, the ambient vector space is defined as $A-A:=\{a-a'\kst a,a'\in A\}$.
Recall that Minkowski decomposition can be used to write
every convex polyhedron $\kP$ as a Minkowski sum of a polytope, i.e.\ a bounded
polyhedron, and a polyhedral cone, namely its tail cone
\[
\tail(\kP):=\{a\in\kP-\kP\kst a+\kP\subseteq\kP\}.
\]
However, this is exactly the type of situation we will
{\em not} consider. Instead, for all Minkowski sums and decompositions in
this paper we will assume that all participating polyhedra share the same tail
cone. For example, if this tail cone is $0$, then we speak
about polytopes. An advantage of this general assumption is that
the Minkowski addition allows cancellation, i.e.\ $A+B=A'+B$ implies $A=A'$.
\\[1ex]
Starting with a polyhedron $\kP$,  one might look at all
possibilities of splitting $\kP$ into a Minkowski sum 
\[
\kP=\kP_0+\ldots+\kP_\minks.
\]
Even if one looks only at the most elementary or extreme
decompositions, they are far
from being unique. They do rather behave like a non-unique prime
factorization. Arguably the most convincing example is the following.
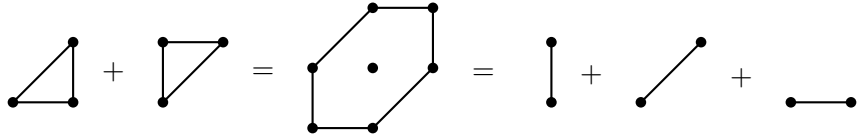
\[
\raisebox{2ex}{$
\begin{tikzpicture}[scale=0.8]
\draw[thick,  color=black]  
  (0,0) -- (1,0) -- (1,1) -- cycle;
\fill[thick,  color=black]
  (0,0) circle (2.5pt) (1,0) circle (2.5pt) (1,1) circle (2.5pt);
\draw[thick,  color=black]
  (1.5, 0.5) node{$\hspace{1.5em}+\hspace{0.8em}$};
\end{tikzpicture}
$}
\raisebox{2ex}{$
\begin{tikzpicture}[scale=0.8]
\draw[thick,  color=black]  
  (0,0) -- (0,1) -- (1,1) -- cycle;
\fill[thick,  color=black]
  (0,0) circle (2.5pt) (0,1) circle (2.5pt) (1,1) circle (2.5pt);
\draw[thick,  color=black]
  (1.5, 0.5) node{$\hspace{1.5em}=\hspace{0.8em}$};
\end{tikzpicture}
$}
\begin{tikzpicture}[scale=0.8]
\draw[thick,  color=black]  
  (0,0) -- (1,0) -- (2,1) -- (2,2) -- (1,2) -- (0,1) -- cycle;
\fill[thick,  color=black]
  (0,0) circle (2.5pt) (1,0) circle (2.5pt) (0,1) circle (2.5pt)
  (1,1) circle (2.5pt) (2,1) circle (2.5pt) (1,2) circle (2.5pt) 
  (2,2) circle (2.5pt);
\draw[thick,  color=black]
  (2.5, 0.9) node{$\hspace{2.2em}=\hspace{0.8em}$};
\end{tikzpicture}
\raisebox{2ex}{$
\begin{tikzpicture}[scale=0.8]
\draw[thick,  color=black]  
  (0,0) -- (0,1) ;
\fill[thick,  color=black]
  (0,0) circle (2.5pt) (0,1) circle (2.5pt);
\draw[thick,  color=black]
  (0.5, 0.45) node{$\hspace{1.5em}+\hspace{0.8em}$};
\end{tikzpicture}
$}
\raisebox{2ex}{$
\begin{tikzpicture}[scale=0.8]
\draw[thick,  color=black]  
  (0,0) -- (1,1);
\fill[thick,  color=black]
  (0,0) circle (2.5pt) (1,1) circle (2.5pt);
\draw[thick,  color=black]
  (1.5, 0.4) node{$\hspace{1.5em}+\hspace{0.8em}$};
\end{tikzpicture}
$}
\raisebox{2.0ex}{$
\begin{tikzpicture}[scale=0.8]
\draw[thick,  color=black]  
  (0,0) -- (1,0);
\fill[thick,  color=black]
  (0,0) circle (2.5pt) (1,0) circle (2.5pt);
\end{tikzpicture}
$}
\]
It is well-known that the set of Minkowski summands of scalar multiples
of $\kP$ (see Definition~\ref{def-MinkSum}) carries the structure of a convex, polyhedral cone
$\kMC(\kP)$,
i.e.\ each $\xi\in\kMC(\kP)$ represents a Minkowski summand
$\kP_\xi$ \cite{versalG}.
For the 
previous hexagon example, it is the four-dimensional cone over a
double tetrahedron. Its vertices, i.e.\ the fundamental rays of the cone,
correspond to the five summands displayed in the figure above.

%%%%%%%%%%%%%%%%%%%%%%%%%%%%%%%%%%%%%%%%%%%%%%%%%%%%%%%%%%%%%%%%%%%%%%%%%%%
\subsection{Considering families}
\label{considerFamilies}
The concept of studying Minkowski summands of scalar multiples of $\kP$
can be reformulated
into a relative setting. We may look at homomorphisms 
$\kpr_+:\kMtC\surj \kMC$ of polyhedral cones such that
$\kpr_+^{-1}(\xi+\xi')=\kpr_+^{-1}(\xi)+\kpr_+^{-1}(\xi')$ for all
$\xi,\xi'\in \kMC$, where the common tail cone of all the fibers $\kpr_+^{-1}(\xi)$ is $\kpr_+^{-1}(0)$. A trivial example of this can be obtained
by taking the affine cone over $\kP$ in $\kQuotRD\oplus\RR$ (with $\kP$ embedded in height 1) and considering
its natural height function $\cone(\kP)\to\RR_{\geq 0}$.
Another example is the projection $\kMtC(\kP)\surj \kMC(\kP)$
with
\[
\kMtC(\kP):=\{(\xi,v)\kst \xi\in\kMC(\kP),\,v\in\kP_\xi\}.
\]
The latter is even universal, namely it is the
terminal object in the category of all those families around
$\cone(\kP)\to\RR_{\geq 0}$, cf.~Proposition~\ref{prop-versalCP}. 
However, while this might just look like an arming of 
language, the striking point consists of the combination of
the following two observations:
\begin{enumerate}[label=(\roman*)]
\item \label{item:justAbove}
One may dualize these notions, looking at injections
$\kpr_+\dual:\kMC\dual\hookrightarrow\kMtC\dual$. Then, 
the property of Minkowski linearity translates into an interesting property
we call \straightness, cf.\
Proposition~\ref{prop-straightMinkLin}.
This property addresses the splitting of $\kMtC\dual$ into a product of
$\kMC\dual$ and a boundary part.
\item \label{item:Above}
The advantage of~\ref{item:justAbove} is that it allows porting the whole setup into the
category of finitely generated semigroups. Doing so, one can again ask
for universal (so, after dualizing, initial) objects of the appropriate
categories.
\end{enumerate}

%%%%%%%%%%%%%%%%%%%%%%%%%%%%%%%%%%%%%%%%%%%%%%%%%%%%%%%%%%%%%%%%%%%%%%%%%%%
\subsection{Extensions of semigroups}
\label{extSG}
We take the observations of Subsection~\ref{considerFamilies} as our
starting point of the whole paper. We will begin in Sections~\ref{straightPairs}
and~\ref{extStraightPairs} from scratch
with developing the appropriate notions in the category of semigroups.
Then, insisting on finite generation, the general approach naturally 
splits into two different setups. We have called them
the {\em cone} and the {\em discrete} setup, and we will focus on them in
Sections~\ref{sectConeSetup} and~\ref{discreteSetupStart}, respectively.
While the cone setup will recover the (duals of the) cones
$\kMC(\kP)$ and $\kMtC(\kP)$, the comparison of both setups
will lead to a new vector space $\kMT(\kP)$ together with a lattice
$\kMTZ(\kP)\subset\kMT(\kP)$, and a rational, polyhedral cone
$\kMTP(\kP)\subset\kMT(\kP)$ generalizing $\kMC(\kP)$.
Studying their dual level, we obtain a finer structure, i.e.\
there is a (unique) finitely generated 
subsemigroup $\ktT$ of the dual Abelian group $\kMTZD(\kP)$ 
fulfilling the universal property~\ref{item:Above} above,
i.e.\ it is the base for a universal \straight\ extension.\\[1ex]
The existence of a universal object in the discrete setup is our  main result.
It is formulated in Theorem~\ref{th-initialObject}, the proof of which 
occupies the whole of  Section~\ref{ssec:initialObject}.
It seems to be an interesting question if the existence and structure
of initial extensions is linked to results like 
\cite[Proposition 3.38]{logGross}
addressing unique liftings in log geometry;
see the remark after Proposition~\ref{prop:characterizeCoCartesian}. Note that unique liftings in log geometry were important for producing smoothings in \cite{fil} and \cite{clm}, see also \cite{gs}, \cite{rud}, \cite{simon}.

\subsection{Involving a lattice structure}
\label{involveLattice}
Let us return to Subsection\,\ref{minkPolyhedra} and
let us assume that we have fixed a lattice structure in our
ambient $\RR$-vector space. For instance, let us start with a free Abelian
group $\kQuotD$ of rank $d$, i.e.\ $\kQuotD\cong\ZZ^d$, and take
$\kQuotRD:=\kQuotD\otimes_\ZZ\RR\cong\RR^d$ as our ambient vector space.
If $\kP$ is a lattice polyhedron, i.e.\ if all vertices belong to $\kQuotD$,
then it is a natural question to look for all lattice decompositions,
i.e.\ for those
Minkowski decompositions such that the summands $\kP_\nu$ are lattice
polytopes, too. One might expect that those $\kP_\nu$ correspond to special
points inside the parametrizing cone $\kMC(\kP)$.
\\[1ex]
However, in the present paper, we go far beyond lattice
polyhedra. Instead, we will deal with arbitrary rational polyhedra, but
we study their interaction with the lattice.
In particular, lattice decompositions do no longer make sense.
Instead, in Section~\ref{latticeFriends}, we introduce the
weaker notion of
lattice-friendly decompositions, cf.~Definition~\ref{def-latticeFriends}.
Then, it is our second main result of this paper that the parameters
$\xi\in\kMTP(\kP)\cap\kMTZ(\kP)$ introduced in Subsection\,\ref{extSG}
correspond to exactly those Minkowski
summands $\kP_\xi$ occurring in lattice friendly decompositions,
cf.~Theorem~\ref{thm-latticeDecompPsi}.
Moreover, similarly to the definition of $\kMtC(\kP)$ in
Subsection\,\ref{considerFamilies}, we have combined in
Theorem~\ref{thm-tautCone} all
Minkowski summands $\kP_\xi$, lattice friendly or not, into
a common polyhedral, so-called tautological  cone $\kMTtP(\kP)$
fibered over $\kMTP(\kP)$.
That is, the lattice $\kMTZD(\kP)$
occurs twice in this paper -- as the ambient space of some universal object
$\ktT$,
but also as the right tool to check Minkowski decompositions for the lattice
friendly property.

%%%%%%%%%%%%%%%%%%%%%%%%%%%%%%%%%%%%%%%%%%%%%%%%%%%%%%%%%%%%%%%%%%%%%%%%%%%
%%%%%%%%%%%%%%%%%%%%%%%%%%%%%%%%%%%%%%%%%%%%%%%%%%%%%%%%%%%%%%%%%%%%%%%%%%%
\vspace{1ex}
{\em Acknowledgements. }
The whole project started in 2017 when the first two authors were guests 
at the University of Genova. We would like to thank for the hospitality and
the financial support.

%%%%%%%%%%%%%%%%%%%%%%%%%%%%%%%%%%%%%%%%%%%%%%%%%%%%%%%%%%%%%%%%%%%%%%%%%%%
%%%%%%%%%%%%%%%%%%%%%%%%%%%%%%%%%%%%%%%%%%%%%%%%%%%%%%%%%%%%%%%%%%%%%%%%%%%
%%%%%%%%
\section{\Straight\ pairs of semigroups}\label{straightPairs}
%%%%%%%%
%%%%%%%%%%%%%%%%%%%%%%%%%%%%%%%%%%%%%%%%%%%%%%%%%%%%%%%%%%%%%%%%%%%%%%%%%%%
%%%%%%%%%%%%%%%%%%%%%%%%%%%%%%%%%%%%%%%%%%%%%%%%%%%%%%%%%%%%%%%%%%%%%%%%%%%

%%%%%%%%%%%%%%%%%%%%%%%%%%%%%%%%%%%%%%%%%%%%%%%%%%%%%%%%%%%%%%%%%%%%%%%%%%%
\subsection{Relative boundaries of semigroups within two different setups}
\label{relBoundary}

Let $\kT\subseteq\kS$ be two 
commutative 
and cancellative ($a+c=b+c\then a=b$) semigroups
with identity $(0+a=a+0=a)$, 
satisfying $\kS\cap(-\kS)=\{0\}$,
i.e.\ $\kS$ (and hence $\kT$) is pointed.
This situation gives rise to the following notion of a relative boundary.

%%%%%%%%%%%%%%%%%%
%%%%%%%%%%%%%%%%%%
\begin{definition}
\label{def-boundary}
The \kbox{boundary of $\kS$ relative to $\kT$} is defined as
\[
  \Rand{\kT}{\kS} = \{ s\in \kS\kst(s-\kT)\cap \kS = \{s\}\}.
\]

This setting comes with a natural addition map
$\ksum:\Rand{\kT}{\kS}\times\kT\to\kS$.
\end{definition}
\begin{example}
In the context of numerical semigroups, that is subsemigroups of $\NN$, 
the so-called Ap\'ery sets are relative boundaries with respect to the subgroup generated by the smallest element. 
\end{example}
The following examples illustrate 
that the relative boundary is almost never a semigroup itself. 

%%%%%%%%%%%%%%%
%%%%%%%%%%%%%%%
\begin{example}
  \label{ex-4Bilder}
  Consider the real cone $S_\RR:=\RR_{\geqslant 0} \cdot[-2,1]+ \RR_{\geqslant 0}\cdot [2,1]\subset \RR^2$, and the finitely generated semigroup $S = S_\RR\cap \ZZ^2$. In the following we consider the boundary of $S_\RR$ with respect to an internal ray, and boundaries of $S$ relative to different subsemigroups; Figures~\ref{fig:3}~and~\ref{fig:4} show how different embeddings of $\NN$ in $S$ give rise to different boundaries (see also Example~\ref{ex-discreteA} for more details). 
  %%%%%%%%%%%%%%%
  %% PICTURE 4 %%
  %%%%%%%%%%%%%%%
  \begin{figure}[h!]
    \centering
    \begin{minipage}{0.49\textwidth}
      \centering      
      \begin{tikzpicture}
        % parameters for size of picture  
        \newcommand{\xa}{5} %width
        \newcommand{\ya}{5} %height
        \newcommand{\ca}{1.7}
        % grid and axes
        \draw[very thin,color=white!80!black, step=.5] (-\xa.2/2,-0.1) grid
        (\xa.2/2,\ya.2/2);    
        \draw[->,name path=xaxis] (-\xa.2/2,0) -- (\xa.2/2,0) node[pos=0.5, below]
        {\tiny $[0,0]$};
        \draw[->,name path=yaxis] (0,-0.2) -- (0,\ya.4/2) node[pos=0.25,left ] {};
        % lines  
        \draw[thick, color=colboundr, domain=0:\xa.2/2, mylabel=at 0.5 above with
        {}] plot (\x,\x/2);
        \draw[thick, color=colboundr, domain=0:\xa.2/2, mylabel=at 0.2 above right
        with {}] plot (-\x,{\x/2}) ;%%};
        % corners of fill area
        \node (a) at (0,0) {};
        \node (b) at (-\xa.2/2,\xa.2/4) {};
        \node (c) at (-\xa.2/2,\ya.2/2) {};
        \node (d) at (\xa.2/2,\ya.2/2) {};
        \node (e) at (\xa.2/2,\xa.2/4) {};
        % color area = semigroup   
        \fill[fill=colsemigr,fill opacity=0.2] (a.center)  -- (b.center) --
        (c.center)  -- (d.center) -- (e.center) --cycle;
        % subsemigroup
        \draw[ thick, color=colsubsgr, domain=0:\ya.2/2] plot (2*\x/3,\x);
        % boundary
%        \node[anchor =west] at (-1.4,-1) {\tiny $T=\RR_{\geqslant 0}\cdot [2,3]\subseteq S_\RR$};        
      \end{tikzpicture}
      \caption{\mbox{\small $T=\RR_{\geqslant 0}\cdot [2,3]\subseteq S_\RR$ }\label{fig:1}}
    \end{minipage}
    \begin{minipage}{0.45\textwidth}
      \centering
      \begin{tikzpicture}
        % parameters for size of picture  
        \newcommand{\xa}{5} %width
        \newcommand{\ya}{5} %height
        \newcommand{\ca}{1.7}
        % grid and axes
        \draw[very thin,color=white!80!black, step=.5] (-\xa.2/2,-0.1) grid
        (\xa.2/2,\ya.2/2);    
        \draw[->,name path=xaxis] (-\xa.2/2,0) -- (\xa.2/2,0) node[pos=0.5, below]
        {\tiny $[0,0]$};
        \draw[->,name path=yaxis] (0,-0.2) -- (0,\ya.4/2) node[pos=0.25,left ] {};
        % lines  
        \draw[name path=line1,domain=0:\xa.2/2, mylabel=at 0.5 above with {}] plot
        (\x,\x/2);
        \draw[name path=line2,domain=0:\xa.2/2, mylabel=at 0.2 above right with
        {}] plot (-\x,{\x/2}) ;%%};
        % corners of fill area
        \node (a) at (0,0) {};
        \node (b) at (-\xa.2/2,\xa.2/4) {};
        \node (c) at (-\xa.2/2,\ya.2/2) {};
        \node (d) at (\xa.2/2,\ya.2/2) {};
        \node (e) at (\xa.2/2,\xa.2/4) {};
        % color area    
        \fill[fill=white!80!black,fill opacity=0.2] (a.center)  -- (b.center) --
        (c.center)  -- (d.center) -- (e.center) --cycle;
        % color nodes
        % semigroup
        \foreach \y in {1,...,\ya}{
          \pgfmathsetmacro\mytemp{2*\y}
          \foreach \x in {1,...,\mytemp}{
            \pgfmathsetmacro\bound{\xa+1}
            \ifthenelse{\x<\bound}{
              \fill[thick, color=colsemigr] (\x/2,\y/2) circle (\ca pt);
              \fill[thick, color=colsemigr] (-\x/2,\y/2) circle (\ca pt);
              \fill[thick, color=colsemigr] ( 0,\y/2) circle (\ca pt);
            }
            {};
          }}
        % subgroup
        \foreach \x in {0,...,\xa}{
          % \pgfmathsetmacro\mytemp{2*\y}
          \foreach \y in {0,...,\x}{
            \pgfmathsetmacro\bound{\xa+1}
            \ifthenelse{\x<\bound AND \x<\y}{
              \fill[thick, color=colsubsgr] (\x/2-\y,\x/2) circle (\ca pt);
              % \fill[thick, color=colsubsgr] (-\x/2,\y/2) circle (\ca pt);
            }
            {};
          }}
        \foreach \x in {0,...,\ya}{
          \draw[thin, color=colsubsgr] (-\x/2,\x/2)--(\ya.2/2 - \x,\ya.2/2);
          \draw[thin, color=colsubsgr] (\x/2,\x/2)--(-\ya.2/2 + \x,\ya.2/2);
          }
        %     % boundary
        \pgfmathsetmacro\i{int((\xa)/2)}
        \foreach \x in {0,...,\i}{
          \fill[thick, color=colboundr] (\x,\x/2) circle (\ca pt);
          \fill[thick, color=colboundr] (-\x,\x/2) circle (\ca pt);
        }
     %   \node[anchor =west] at (-1.7,-1) {\tiny $T=\spann{}{[-1,1],[1,1]}\subseteq S$};
      \end{tikzpicture}      
      \caption{\mbox{\small$T=\spann{\NN}{[-1,1],[1,1]}\subseteq S$}\label{fig:2}}
    \end{minipage}
  \end{figure}

  \begin{figure}[h!]
    \centering
    \begin{minipage}{0.49 \textwidth}
      \centering
      \begin{tikzpicture}
        % parameters for size of picture  
        \newcommand{\xa}{5} %width
        \newcommand{\ya}{5} %height
        \newcommand{\ca}{1.7}
        % grid and axes
        \draw[very thin,color=white!80!black, step=.5] (-\xa.2/2,-0.1) grid
        (\xa.2/2,\ya.2/2);    
        \draw[->,name path=xaxis] (-\xa.2/2,0) -- (\xa.2/2,0) node[pos=0.5, below]
        {\tiny $[0,0]$};
        \draw[->,name path=yaxis] (0,-0.2) -- (0,\ya.4/2) node[pos=0.25,left ] {};
        % lines  
        \draw[name path=line1,domain=0:\xa.2/2, mylabel=at 0.5 above with {}] plot
        (\x,\x/2);
        \draw[name path=line2,domain=0:\xa.2/2, mylabel=at 0.2 above right with {}] plot (-\x,{\x/2}) ;%%};
        % corners of fill area
        \node (a) at (0,0) {};
        \node (b) at (-\xa.2/2,\xa.2/4) {};
        \node (c) at (-\xa.2/2,\ya.2/2) {};
        \node (d) at (\xa.2/2,\ya.2/2) {};
        \node (e) at (\xa.2/2,\xa.2/4) {};
        % color area    
        \fill[fill=white!80!black,fill opacity=0.2] (a.center)  -- (b.center) --
        (c.center)  -- (d.center) -- (e.center) --cycle;
        % color nodes
        % semigroup
        \foreach \y in {1,...,\ya}{
          \pgfmathsetmacro\mytemp{2*\y}
          \foreach \x in {1,...,\mytemp}{
            \pgfmathsetmacro\bound{\xa+1}
            \ifthenelse{\x<\bound}{
              \fill[thick, color=colsemigr] (\x/2,\y/2) circle (\ca pt);
              \fill[thick, color=colsemigr] (-\x/2,\y/2) circle (\ca pt);
            }
            {};
          }}
        % subgroup
        \foreach \y in {1,...,\ya}
        \fill[color=colsubsgr] (0,\y/2) circle (\ca pt);
        \draw[thin, color=colsubsgr] (0,0)--(0,\ya/2);
        % boundary
        \foreach \x in {0,...,\xa}{
          % \pgfmathparse{int(round({(\x+1)/2}))}
          \pgfmathsetmacro\i{int((\x+1)/2)}
          \fill[thick, color=colboundr] (\x/2,\i/2) circle (\ca pt);
          \fill[thick, color=colboundr] (-\x/2,\i/2) circle (\ca pt);
        }
        % \node[anchor =west] at (-1.2,-1) {\tiny $T_0=\spann{}{[0,1]}\subseteq S$};
      \end{tikzpicture}
      \caption{\mbox{\small $T_0=\spann{\NN}{[0,1]}\subseteq S$}\label{fig:3}}
      
    \end{minipage}
    % parameters for size of picture  
    \begin{minipage}{0.45 \textwidth}
      \centering
      \begin{tikzpicture}
        \newcommand{\xa}{5} %width
        \newcommand{\ya}{5} %height
        \newcommand{\ca}{1.7}
        % grid and axes
        \draw[very thin,color=white!80!black, step=.5] (-\xa.2/2,-0.1) grid
        (\xa.2/2,\ya.2/2);    
        \draw[->,name path=xaxis] (-\xa.2/2,0) -- (\xa.2/2,0) node[pos=0.5, below]
        {\tiny $[0,0]$};
        \draw[->,name path=yaxis] (0,-0.2) -- (0,\ya.4/2) node[pos=0.25,left ] {};
        % lines  
        \draw[name path=line1,domain=0:\xa.2/2, mylabel=at 0.5 above with {}] plot
        (\x,\x/2);
        \draw[name path=line2,domain=0:\xa.2/2, mylabel=at 0.2 above right with
        {}] plot (-\x,{\x/2}) ;%%};
        % corners of fill area
        \node (a) at (0,0) {};
        \node (b) at (-\xa.2/2,\xa.2/4) {};
        \node (c) at (-\xa.2/2,\ya.2/2) {};
        \node (d) at (\xa.2/2,\ya.2/2) {};
        \node (e) at (\xa.2/2,\xa.2/4) {};
        % color area    
        \fill[fill=white!80!black,fill opacity=0.2] (a.center)  -- (b.center) --
        (c.center)  -- (d.center) -- (e.center) --cycle;
        % color nodes
        % semigroup
        \foreach \y in {1,...,\ya}{
          \pgfmathsetmacro\mytemp{2*\y}
          \fill[thick, color=colsemigr] (0,\y/2) circle (\ca pt);
          \foreach \x in {1,...,\mytemp}{
            \pgfmathsetmacro\bound{\xa+1}
            \ifthenelse{\x<\bound}{
              \fill[thick, color=colsemigr] (\x/2,\y/2) circle (\ca pt);
              \fill[thick, color=colsemigr] (-\x/2,\y/2) circle (\ca pt);
            }
            {};
          }}
        % subgroup
        \foreach \y in {1,...,\ya}
        \fill[color=colsubsgr] (\y/2,\y/2) circle (\ca pt);
        \draw[thin,color=colsubsgr] (0,0)--(\ya.2/2,\ya.2/2);
        % boundary
        \foreach \x in {0,...,\xa}{
          % \pgfmathparse{int(round({(\x+1)/2}))}
          \pgfmathsetmacro\i{int((\x+1)/2)}
          \ifthenelse{\isodd{\x}}
          {\fill[thick, color=colboundr] (-\i+1,\i/2) circle (\ca pt);}
          {\fill[thick, color=colboundr] (\x/2,\i/2) circle (\ca pt);}
          \fill[thick, color=colboundr] (-\x/2,\i/2) circle (\ca pt);
        }

        %\node[anchor =west] at (-1.2,-1) {\tiny $T_1=\spann{}{[1,1])}\subseteq S$};
      \end{tikzpicture}
      \caption{\mbox{\small $T_1=\spann{\NN}{[1,1]}\subseteq S$}\label{fig:4}}
    \end{minipage}
  \end{figure}
  \begin{figure}[h!]
    \centering
    \begin{tikzpicture}
      \newcommand{\xsg}{6}
      \newcommand{\xss}{0}
      \newcommand{\xbd}{3}
      \fill[fill=colsemigr,fill opacity=0.6] (\xsg,0)  -- (\xsg+1,0) --
      (\xsg+1,0.3) -- (\xsg,0.3) --cycle;
      \node at (\xsg+2.4,0.17) {= $S \setminus(T \cup\Rand{T}{S})$};
      \fill[fill=colsubsgr,fill opacity=0.6] (\xss,0)  -- (\xss+1,0) --
      (\xss+1,0.3) -- (\xss,0.3) --cycle;
      \node at (\xss+1.4,0.17) {= $T$};
      \fill[fill=colboundr,fill opacity=0.6] (\xbd,0)  -- (\xbd+1,0) --
      (\xbd+1,0.3) -- (\xbd,0.3) --cycle;
      \node at (\xbd+1.6,0.17) {= $\Rand{T}{S}$};
    \end{tikzpicture}    
  \end{figure}
\end{example}
%%%%%%%%%%%%%%% 
%%%%%%%%%%%%%%%
There are two quite different classes of semigroups we have in mind.
Both are, in their own way, 
% somehow 
finitely generated.
%%%%%%%%%%%%%%%%%%%%%%%%%%%%%%%%%%%%%%%%%%%%%%%%%%%%%%%%%%%%%%%%%%%%%%%%%%%
\subsubsection{The cone setup}
\label{coneSetup}
Here we take for $\kT\subseteq\kS$ polyhedral cones in some 
finitely-dimensional real vector space.
These gadgets are finitely generated by their fundamental rays as 
``$\RR_{\geqslant 0}$-modules'', but not at all
finitely generated as semigroups.

%%%%%%%%%%%%%%%
%%%%%%%%%%%%%%%
\begin{example}
\label{ex-coneA}
Assume that $\kT\subseteq\kS$ is a ray, i.e.\ $\kT=\RR_{\geqslant 0}\cdot R$
for some $R\in\kS\setminus\{0\}$.
Then there is unique face $F=F(\kT)\leq\kS$ such that $R$ and 
hence $\kT\setminus\{0\}$ is contained in the relative interior $\innt(F)$.
Then $S\setminus \Rand{\kT}{\kS} =\star(F):= \bigcup_{F\leq G\leq\kS} \innt(G)$.
Note that both $\innt(\kS)$ and $\innt(F)$ are parts of this set, i.e.\
$\Rand{\kT}{\kS} \subseteq \partial\kS\setminus \innt(F)$
with $\partial\kS:=\kS\setminus\innt(\kS)$ denoting the classical topological
boundary. 
\\
For the special case $R\in\innt(\kS)$, as in Figure~\ref{fig:1},
we even have that $\Rand{\kT}{\kS}=\partial\kS$.
In particular, in this situation the relative boundary does not depend on 
a further specification of $R$.
\end{example}

%%%%%%%%%%%%%%%%%%%%%%%%%%%%%%%%%%%%%%%%%%%%%%%%%%%%%%%%%%%%%%%%%%%%%%%%%%%
\subsubsection{The discrete setup}
\label{discreteSetup}
Here we  suppose
that both $\kT$ and $\kS$ are finitely generated as semigroups,
i.e.\ as ``$\NN$-modules''.
In this case, the so-called Hilbert basis, consisting
of all irreducible elements, provides even a minimal, hence canonical,
finite generating system.
A typical example of this situation is the intersection
of a cone setup with an underlying lattice.

%%%%%%%%%%%%%%%
%%%%%%%%%%%%%%%
\begin{example}
\label{ex-discreteA}
Let $\kS:=\spann{\RR_{\geq 0}}{[-2,1],\,[2,1]}\cap\ZZ^2$
as in the three discrete figures of Example~\ref{ex-4Bilder}, and in Figure~\ref{fig:6}, 
the Hilbert basis of this semigroup is
\[
H=\{[-2,1],\,[-1,1],\,[0,1],\,[1,1],\,[2,1]\}.
\]
The semigroup $\kS$ contains the inner ``discrete rays''
$\kT_0=\NN\cdot [0,1]$ and $\kT_1=\NN\cdot [1,1]$, and their respective
relative boundaries are
\[
\Rand{\kT_0}{\kS}=
\{[\pm 2b,\,b]\kst b\in\NN\}\;\cup\; \{[\pm(2b-1),\,b]\kst b\in\NN_{\geqslant 1}\}
\]
and
\[
\Rand{\kT_1}{\kS}=
\{[\pm 2b,\,b]\kst b\in\NN\}\;\cup\; \{[-2b+1,\,b],\; [-2b+2,\,b]
\kst b\in\NN_{\geqslant 1}\}.
\]
That is, while both $\kT_0$ and $\kT_1$ come from the ``interior'' of $\kS$,
they lead to different relative, discrete boundaries;
see Figures~\ref{fig:3}~and~\ref{fig:4}, respectively.
\end{example}

%%%%%%%%%%%%%%%%%%%%%%%%%%%%%%%%%%%%%%%%%%%%%%%%%%%%%%%%%%%%%%%%%%%%%%%%%%%
\subsection{\Straightness}\label{defStraight}
We will use the addition map $\ksum:\Rand{\kT}{\kS}\times\kT\to\kS$
for decomposing elements of the semigroup $\kS$. 
In general, i.e.\ if we are in the cone setup~(\ref{coneSetup}) 
or the discrete setup~(\ref{discreteSetup}),
the existence of those decompositions is not a problem.
This is established by the following lemma.

%%%%%%%%%%%%%%%
%%%%%%%%%%%%%%%
\begin{lemma}
\label{lem-ksumSurj}
Assume that we are either in the cone or the discrete setup.
Then the canonical addition map
$\ksum:\Rand{\kT}{\kS}\times\kT\to\kS$ is automatically surjective.
\end{lemma}

%%%%%%%%%%%%%%%
%%%%%%%%%%%%%%%
\begin{proof}
Let $\kT=\spann{}{ t_1,\ldots,t_k}$ and $\kS=\spann{}{s_1,\dots,s_r}$, which we consider either as  $\NN-$modules, or as $\RR_{\geqslant 0}-$modules. 
For each $s\in\kS$ write $s=a_1s_1+\dots+a_rs_r$, and for each $i=1,\dots,k$ write $t_i=b_1s_1+\dots + b_rs_r$.
By the pointedness assumption, we have that for every $n\in\NN$ with $a_i<nb_i$, for all $i$, we get $s-nt\notin\kS$.
So in both setups, there exists a maximal $n^*\in\RR_{\geqslant 0}$, respectively $\in\NN$,
with $s-n^*t\in\kS$. Continuing this process with all generators of $\kT$
eventually leads to an element $s^*\in\kS$ which cannot be decreased via $\kT$.
\end{proof}

The injectivity of $\ksum$ is less common, but, as we will see, very powerful. Therefore, we introduce the following key terminology.

%%%%%%%%%%%%%%%
%%%%%%%%%%%%%%%
\begin{definition}
\label{def-straight}
The semigroups $\kT\subseteq\kS$  form a
\kbox{\straight\ pair $(\kT,\kS)$} 
(or $\kiota:\kT\hookrightarrow\kS$ is called a \straight\ embedding) if
the addition map $\ksum:\Rand{\kT}{\kS}\times\kT\to\kS$ is bijective.
\end{definition}

% %%%%%%%%%%%%%%%
% %%%%%%%%%%%%%%%

\begin{example}
\label{ex-coneB}
Let $\kS:=\spann{\RR_{\geq 0}}{[-2,1],\,[2,1]}$ and
$\kT:=\spann{\RR_{\geq 0}}{[-1,1],\,[1,1]}$. We are thus in in the cone setup, and
 $\Rand{\kT}{\kS} = \partial\kS$
as in the situation at the end of Example~\ref{ex-coneA}.
However, the surjective map $\ksum:\Rand{\kT}{\kS}\times\kT\to\kS$
is not injective. For instance,
$[0,0] + [2,2]  
= [2,1] + [0,1]$ displays two different decompositions
of $[2,2]\in\kS$.
Applying Proposition~\ref{lem-decomposableByQuot} will make this even more obvious:
We obtain $\kQuot=\RR^2/\RR^2=0$, hence $\kqr:\Rand{\kT}{\kS}\to\kQuot$
has no chance to become injective.
\\[0.3ex]
Note that literally the same remains true if we intersect everything with the
lattice $\ZZ^2$. This yields a non-\straight\ example in the discrete setup,
too.
Alternatively, in Figure~\ref{fig:2} above,
where we have that
$
  \kT:=\spann{\NN}{ [-1,1],[1,1]}
$
and
$ 
% \subseteq 
\kS:=\spann{\NN}{[-2,1],[-1,1],[0,1],[1,1],[2,1]}
% =:\kS,
$
we can take for a non-unique decomposition $[0,0] + [4,4]  = [4,2] + [0,2]$.
\end{example}

%%%%%%%%%%%%%%%%%%%%%%%%%%%%%%%%%%%%%%%%%%%%%%%%%%%%%%%%%%%%%%%%%%%%%%%%%%%
\subsection{The decomposition operators}
\label{decOp}
By definition, \straight\ pairs $(\kT,\kS)$ allow a unique decomposition of
every element $s\in \kS$  into a sum
\[
s = \bound(s) + \height(s)
\hspace{1em}\mbox{with}\hspace{1em}
\bound(s)\in\Rand{\kT}{\kS} 
\hspace{0.5em}\mbox{and}\hspace{0.5em}
\height(s)\in\kT.
\]
In other words, there are retraction maps $\bound:\kS\surj \Rand{\kT}{\kS}$
and $\height:\kS\surj\kT$
with $\,\bound+\height=\id\,$
satisfying
\[
\bound|_{\Rand{\kT}{\kS}}=\id,
\hspace{0.5em}
\bound|_{\kT}=0
\hspace{1em}\mbox{and}\hspace{1em}
\height|_{\Rand{\kT}{\kS}}=0,
\hspace{0.5em}
\height|_{\kT}=\id.
\]
Note that $\height$ is in general not linear, i.e.\ not a semigroup
homomorphism.
Moreover, for $\bound$, linearity does not even make sense, since
the target $\Rand{\kT}{\kS}$ is not a semigroup.
Finally, in the discrete setup,
the Hilbert basis $H$ of $\kS$ hosting a \straight\ pair
$(\kT,\kS)$ splits into two parts, namely
\[
H = \big(H\cap\Rand{\kT}{\kS}\big) \sqcup \big(H\cap\kT\big).
\]

%%%%%%%%%%%%%%%%%%%%%%%%%%%%%%%%%%%%%%%%%%%%%%%%%%%%%%%%%%%%%%%%%%%%%%%%%%%
\subsection{Rays yield \straight\ pairs}
\label{raysStraight}
While Example~\ref{ex-coneB} 
has shown that \straightness\ is not always satisfied, there is, nevertheless, 
a standard situation where this property is guaranteed.

%%%%%%%%%%%%%%%
%%%%%%%%%%%%%%% 
\begin{definition}
\label{def-ray}
In both setups, we call $\kT$ a ray if it is saturated in 
the ambient Abelian group $\kS-\kS$ and if its canonical poset structure
($t\leq t'$ $:\iff$ $t'-t\in\kT$) is a total order. 
\end{definition}

In the cone setup~(\ref{coneSetup}), this
means $\kT\cong\RR_{\geqslant 0}$; in the discrete setup~(\ref{discreteSetup}),
the ray property implies that $\kT\cong\NN$.
In both situations, there exists  an $R\in\kS$ such that
$\kT\subseteq\kS$ consists of all ``allowed'' multiples of $R$,
i.e.\ using $\RR_{\geqslant 0}$ or $\NN$ as coefficients, respectively. 

%%%%%%%%%%%%%%%
%%%%%%%%%%%%%%%
\begin{proposition}
\label{prop-ksumInj}
\label{lem:decompositionDonwstairs}
If $\kT$ is a ray, then $\ksum$ is injective, i.e.\
$(\kT,\kS)$ is a \straight\ pair.
\end{proposition}

%%%%%%%%%%%%%%%
%%%%%%%%%%%%%%%
\begin{proof}
Let $b,b'\in\Rand{\kT}{\kS}$ and $t,t'\in\kT$ with
$b+t=b'+t'$. We may, without loss of generality, assume that $t\geq t'$.
Then, the cancellation property implies that $b+(t-t')=b'\in\Rand{\kT}{\kS}$
with $t-t'\in\kT$.
By definition of the relative boundary, this means that $t-t'=0$,
i.e.\ $t=t'$ and hence $b=b'$.
\end{proof}

%%%%%%%%%%%%%%%%%%%%%%%%%%%%%%%%%%%%%%%%%%%%%%%%%%%%%%%%%%%%%%%%%%%%%%%%%%%
\subsection{Involving the ambient Abelian groups}
\label{invAmbGrp}
Since $\kT\subseteq\kS$ are both cancellative, we may embed them into their
respective linear hulls
\[
\kSpanT:=\kT-\kT\subseteq\kS-\kS=:\kSpanS.
\]
These ambient objects $\kSpanT,\kSpanS$ are torsion free Abelian groups.
In the cone or in the discrete setup, they are finitely generated
$\RR$-, respectively $\ZZ$-modules. That is, $\kSpanT$ and $\kSpanS$ are
finitely dimensional vector spaces or free Abelian groups of finite rank.
We denote by $\kQuot:=\kSpanS/\kSpanT$ the quotient
(which might have torsion in the discrete setup).
This leads to the quotient map
\[
\kqr:\kS\surj\kQuotS\subseteq\kQuot
\hspace{1em}\mbox{with}\hspace{1em}
\kQuotS:=(\kS-\kT)/(\kT-\kT)
\]
denoting its image.
The usage of the ambient groups and their quotient $\kQuot$
yields the following criterion of \straightness\ in terms of the injectivity of
$\kqr|_{\Rand{}{\kS}}$.

%%%%%%%%%%%%%%%
%%%%%%%%%%%%%%%
\begin{lemma}
\label{lem-decomposableByQuot}
Let $(\kT,\kS)$ be a pair of semigroups such that
$\ksum:\Rand{\kT}{\kS}\times\kT\to\kS$ is surjective. We have:
\begin{enumerate}[label=(\roman*)]
\item \label{item:sbqi}The restriction $\kqr|_{\Rand{}{\kS}}:\Rand{\kT}{\kS}\to\kQuotS$ is
surjective.
\item \label{item:sbqii} The map $\kqr|_{\Rand{}{\kS}}:\Rand{\kT}{\kS}\to\kQuot$ is injective if
and only if $(\kT,\kS)$ is \straight.
\end{enumerate}
\end{lemma}

%%%%%%%%%%%%%%%
%%%%%%%%%%%%%%%
\begin{proof}
\ref{item:sbqi} The surjectivity of $\kqr|_{\Rand{}{\kS}}$ is a direct consequence
from the surjectivity of the addition map~$\ksum$.
\\[1ex]
\ref{item:sbqii} The direct implication is obvious.
For the converse, assume that $(\kT,\kS)$ is \straight\ and that
$\kqr(b)=\kqr(b')$
for some $b,b'\in \Rand{\kT}{\kS}$. This implies
$b-b'\in\kT-\kT$, so there are $t,t'\in\kT$ with $b+t=b'+t'$. The
latter displays two decompositions of the same element into summands from
$\Rand{\kT}{\kS}$ and $\kT$. Hence \straightness\ implies $b=b'$.
\end{proof}

%%%%%%%%%%%%%%%%%%%%%%%%%%%%%%%%%%%%%%%%%%%%%%%%%%%%%%%%%%%%%%%%%%%%%%%%%%%

%%%%%%%%%%%%%%%%%%%%%%%%%%%%%%%%%%%%%%%%%%%%%%%%%%%%%%%%%%%%%%%%%%%%%%%%%%%
%%%%%%%%%%%%%%%%%%%%%%%%%%%%%%%%%%%%%%%%%%%%%%%%%%%%%%%%%%%%%%%%%%%%%%%%%%%
%%%%%%%%
\section{\aextend ing \straight\ pairs}\label{extStraightPairs}
%%%%%%%%
%%%%%%%%%%%%%%%%%%%%%%%%%%%%%%%%%%%%%%%%%%%%%%%%%%%%%%%%%%%%%%%%%%%%%%%%%%%
%%%%%%%%%%%%%%%%%%%%%%%%%%%%%%%%%%%%%%%%%%%%%%%%%%%%%%%%%%%%%%%%%%%%%%%%%%%

%%%%%%%%%%%%%%%%%%%%%%%%%%%%%%%%%%%%%%%%%%%%%%%%%%%%%%%%%%%%%%%%%%%%%%%%%%%
\subsection{\aextend ing semigroups}\label{extSemigroups}
Starting with a free pair $\kT\hookrightarrow\kS$ we are going to consider all possibilities
to put this in relation  with other free pairs  $\ktT\hookrightarrow\ktS$ having isomorphic boundaries.

%%%%%%%%%%%%%%%
%%%%%%%%%%%%%%%
\begin{definition}
  \label{def-extension}
  We call a semigroup homomorphism $\kpi:\ktS\to\kS$ an \kbox{\aextension} 
  if it has trivial kernel, that is if $\ker \kpi= \{\wt{s}\in\ktS\kst \kpi(\wt{s})=0\}= 0$, which is equivalent%
  \footnote{Note that for semigroup homomorphisms, having a trivial kernel does not imply injectivity.}
  to 
  \[
    \kpi\big(\ktS\setminus\{0\}\big)\subseteq \kS\setminus\{0\}.
  \]
Let $\kT\hookrightarrow\kS$ be a pair of semigroups (not necessary free). A commutative diagram of semigroup maps
\[
 \xymatrix{
   \ktT~ \ar@{^{(}->}[r] \ar@{->}[d]_{\kpi_{\kT}}
   &
   \ktS \ar@{->}[d]^{\kpi_{\kS}}
   \\
   \kT~ \ar@{^{(}->}[r]
   &
   \kS 
 }
\]
is an \kbox{\aextension~ of the pair} $(\kT,\kS)$  if $\kpi_\kS $ (and thus also $\kpi_\kT$) is an \aextension.
An \aextension~ is called \kbox{\cocartesian}~if the following two conditions are satisfied:
\begin{enumerate}[label=(\roman*)]
  \item the addition maps $\ksum$ and $\ktsum$ are surjective, and
  \item $\kpi$ induces a bijection on the boundaries: $\Rand{\ktT}{\ktS}\stackrel{\sim}{\longrightarrow}\Rand{\kT}{(\kS)}$.
\end{enumerate}
\end{definition}

We will frequently denote both vertical maps simply by $\kpi$.
Note that the above diagram alone immediately implies that
$\kpi\big(\ktS\setminus \Rand{\ktT}{\ktS}\big)\subseteq
\kS\setminus \Rand{\kT}{\kS}$. On the other hand, $\kpi$ generally fails
to map $\Rand{\ktT}{\ktS}$ into $\Rand{\kT}{\kS}$,
cf.~Example~\ref{ex-pairExtension}.\ref{item:ex-pair2} 

%%%%%%%%%%%%%%%
%%%%%%%%%%%%%%%
\begin{example}
  \label{ex-pairExtension}
  \begin{enumerate}[label={\bf \arabic*.}, ref={\arabic*.}]
  \item  A trivial possibility for extending pairs is to first define $\ktS:=\kS\times \kF$
    with $\kF$ any semigroup of the scenario in question.
    However, the plain projection $\pr_{\kS}:\kS\times \kF\surj \kS$
    does not meet our requirements, because its kernel equals $\kF$.
    This can be corrected by choosing any semigroup map $\ell:\kF\to\kS$
    with trivial kernel and defining $\kpi_\ell:=\pr_{\kS}+\ell$, i.e.\ $\kpi_\ell(s,f):=s+\ell(f)$. 
    Using this notation,  the forbidden plain projection corresponds to the forbidden $\ell=0$.
    To obtain an extension of the pair, take $\ell:\kF\to\kT\subseteq\kS$ with $\ker\ell=0$, and define $\ktT:=\kT\times \kF$.
    Note that $\Rand{\kT\times\kF}{(\kS\times\kF)}=\Rand{\kT}{(\kS)}\times\{0\}$.
    Hence the \straightness\ property of $(\kT,\kS)$ is equivalent to
    the similar one for $(\ktT,\ktS)$.
\item  \label{item:ex-pair2} We consider an example in the cone setup~(\ref{coneSetup}). 
To be able to draw what is going
on, we intersect both cones $\ktS$ and $\kS$ with affine hyperplanes
-- displaying convex polytopes 
(the origin of the cones being behind the screen):
\begin{figure}[h!]
\begin{tikzpicture}[scale=0.5]
\draw[thin,  color=black]
(0,0) -- (2,-3) -- (5,-2) -- (4,2) -- (0,0);
\fill[fill=colsemigr,fill opacity=0.2] (0,0)  -- (2,-3) --
   (5,-2)  -- (4,2) -- cycle;
\draw[thick,  color=colsubsgr]
(0,0) -- (4.5,0);

\draw[thick,  color=colsubsgr]
  (3,-0.7) node{$\ktT$};
\draw[thick,  color=black]  
  (-0.7,-1.9) node{$\ktS$};
\draw[thick,  color=colboundr]
  (0.1,0.05) -- (4,2) (0.03,-0.06) -- (2,-3) -- (5,-2);
\draw[thick,  color=colboundr]
  (4.2,-2.9) node{$\ks \Rand{\ktT}{\ktS}$}
  (1.2,1.3) node{$\ks \Rand{\ktT}{\ktS}$};
\draw[thick,  color=black]  
  (7,0) node{$\stackrel{\kpi}{\longrightarrow}$};
\draw[thick,  color=colsemigr]
  (10,2) -- (10,-3);
\fill[thick,  color=colsubsgr]
  (10,0) circle (3pt);
\draw[thick,  color=colsubsgr]
  (10.7,-0.0) node{$\kT$};
\draw[thick,  color=black]
  (11.7,-1.4) node{$\kS$};
\fill[thick,  color=colboundr]
  (10,2) circle (3pt) (10,-3) circle (3pt);  
\draw[thick,  color=colboundr]
  (10.9,-3.0) node{$\ks \Rand{\kT}{\kS}$}
  (10.9,2) node{$\ks \Rand{\kT}{\kS}$};
\end{tikzpicture}
\caption{$\kpi$ does not always map boundary to boundary.}\label{fig:5}
\end{figure}
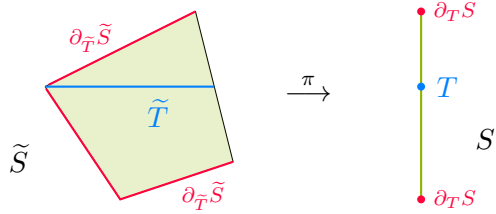

Extending the formula of Example~\ref{ex-coneA}, we obtain that, in the cone setup,
\[
\displaystyle
\kS\setminus \Rand{\kT}{\kS}=\bigcup_{\small \substack{R\in\kT\\ R\neq  0}}
\big(\kS\setminus \Rand{R}{\kS}\big)
= \bigcup_{\small \substack{R\in\kT\\R\neq 0}}\bigg(\bigcup_{\small \substack{R\in G\\ G\leqslant\kS}} \innt(G)\bigg)
= \bigcup_{\small \substack{G\leqslant\kS\\ G\cap\kT\neq 0}} \innt(G).
\]
In particular, in Figure~\ref{fig:5}, we have that $\kpi^{-1}(\kT)=\ktT$,
but $\Rand{\ktT}{\ktS}$ does not map to $\Rand{\kT}{\kS}$.

  \end{enumerate}

\end{example}

%%%%%%%%%%%%%%%%%%%%%%%%%%%%%%%%%%%%%%%%%%%%%%%%%%%%%%%%%%%%%%%%%%%%%%%%%%%
\subsection{Keeping it \straight}\label{keepStraight}
The main point of the present subsection is to keep track of the 
\straightness\ property along extensions of pairs.
The next result shows two important consequences of a diagram being \cocartesian, and that if the vertical maps are surjective, each of these consequences are also sufficient. Let the following diagram define an \aextension~ of the free pair $(\kT,\kS)$
\[
\xymatrix{
\ktT~ \ar@{^{(}->}[r] \ar@{->}[d]_{\kpi_{\kT}} & \Rand{\ktT}{(\ktS)} \times \ktT  \ar@{->>}[r]^-\ktsum & \ktS \ar@{->}[d]^{\kpi_{\kS}} \\
\kT~ \ar@{^{(}->}[r]                           & \Rand{\kT}{(\kS)} \times \kT\ar@{->>}[r]^-\ksum       & \kS 
}
\]
with the addition map $\ktsum$  surjective. Denote by $\ktQuot:=(\ktS-\ktS)/(\ktT-\ktT)$ and by $\kQuot:=(\kS-\kS)/(\kT-\kT)$.
Consider the following three conditions.
\begin{enumerate}[label=(C\arabic*)]
\item \label{condition1} The \aextension~ is \cocartesian.
\item \label{condition2} The pair $(\ktT,\ktS)$ is free and $\ko{\kpi}:\ktQuot\to\kQuot$ is an isomorphism.
\item \label{condition3} For all $\kts_1,\kts_2 \in \ktS$ with $\kpi(\kts_1)=\kpi(\kts_2)$, there exist $\ktt_1,\ktt_2 \in \ktT$ such that $\kts_1-\ktt_1=\kts_2-\ktt_2\in\ktS$.
\end{enumerate}

\begin{proposition}\label{prop:characterizeCoCartesian}%
  \begin{enumerate}[label=(\roman*)]
  \item In the above situation we have the following logical relations:
\[
  \begin{tikzcd}[arrows=Rightarrow,every arrow/.append style={shift left=0.8ex},column sep=1.5em]
    \text{\ref{condition1}} \arrow[xshift = 3pt]{r} &
    \text{\ref{condition2}} \arrow[r,xshift=3pt] \arrow["/"{anchor=center,sloped}]{l} &
    \text{\ref{condition3}} \arrow["/"{anchor=center,sloped}]{l}
  \end{tikzcd}
\]
\item If the maps $\kpi_\kS$ and $\kpi_\kT$ are surjective, 
then
\[
\begin{tikzcd}[arrows=Rightarrow,every arrow/.append style={shift left=0 ex},column sep=1.5em]
    \text{\ref{condition3}} \arrow{r} &
    \text{\ref{condition1}},
  \end{tikzcd}
\]
thus the three conditions are equivalent in this case. 
  \end{enumerate}
\end{proposition}
\begin{proof}
\fbox{\ref{condition1}$\then$\ref{condition2}}
 {\em The decomposability of $(\ktT,\ktS)$.}
\label{lem:decompositionUpstairs}
Assume that $\wt{b}_1+\ktt_1=\wt{b}_2+\ktt_2$, with $\wt{b}_i\in\Rand{\ktT}{\ktS}$ and $\ktt_i\in\ktT$. Applying $\kpi$ we obtain
\[
  \kpi(\wt{b}_1)+\kpi(\ktt_1)=\kpi(\wt{b}_2)+\kpi(\ktt_2).
\]
By~\ref{condition1} we have $\kpi(\wt{b}_1),\kpi(\wt{b}_2)\in\Rand{\kT}{\kS}$, and the diagram condition implies $\kpi(\ktt_1),\kpi(\ktt_2)\in\kT$.
So, by the decomposability of $(\kT,\kS)$, that $\kpi(\wt{b}_1)=\kpi(\wt{b}_2)$. Again by~\ref{condition1} we obtain $\wt{b}_1=\wt{b}_2$, and thus $\ktt_1=\ktt_2$, so the decomposition is unique.\\[1ex]
{\em The group isomorphism.} Since the addition maps are surjective, every element of $\kQuot$ and of $\ktQuot$  can be represented by a corresponding boundary element. So the surjectivity of the restriction to the boundary implies the surjectivity of the map $\ko{\kpi}:\ktQuot\to\kQuot$.
By Lemma~\ref{lem-decomposableByQuot}, we have
$\kqr|_{\Rand{}{\kS}}:\Rand{\kT}{\kS}\stackrel{\sim}{\to}\kQuotS
\subseteq\kQuot$ on both levels, $\ktQuot$ and $\kQuot$.
Hence $\ko{\kpi}:\ktQuot\to\kQuot$ is an isomorphism
on the images of $\ktS$ and $\kS$ in $\ktQuot$ and $\kQuot$, respectively.
Since these images generate the two groups, we are done.
\\[1ex]
\fbox{\ref{condition2}$\then$\ref{condition3}}
Let $\kts_1,\kts_2 \in \ktS$ with $\kpi(\kts_1)= \kpi(\kts_2)$ in $\kS$, hence in $\kQuot$.
Then the second part of~\ref{condition2}, i.e.\ the fact that $\ko{\kpi}$ is an isomorphism, 
implies that $\kts_1$ and $\kts_2$ become equal in $\ktQuot$, i.e.\
$\ktqr(\kts_1)=\ktqr(\kts_2)$.
Now, we consider the unique decompositions
\[
\kts_1=\kts_1\!\!'+\ktt_1
\hspace{1em}\mbox{and}\hspace{1em}
\kts_2=\kts_2\!\!'+\ktt_2
\hspace{1em}\mbox{within}\hspace{1em}
\Rand{\ktT}{(\ktS)}\times\ktT.
\]
We still have $\ktqr(\kts_1\!\!')=\ktqr(\kts_2\!\!')$,
but now we can use the decomposability of $(\ktT,\ktS)$ in the way provided by
Lemma~\ref{lem-decomposableByQuot}, namely as the injectivity of
$\ktqr:\Rand{\ktT}{\ktS}\to\ktQuot$. This implies $\kts_1\!\!'=\kts_2\!\!'=:\kts~\!'$,
hence
$\kts_1-\ktt_1=\kts_2-\ktt_2=\kts~\!'\in\ktS$.
\\[1ex]
\fbox{\ref{condition2}$\not\then$\ref{condition1}}
Take the following extension with surjective addition maps:
\[
  \begin{tikzcd}
    0 \arrow[r, hook]\arrow[d]& \braket{2,3}\subseteq \NN \arrow[d,xshift = -1em] \\
    0 \arrow[r, hook]& \braket{1} =  \NN  
  \end{tikzcd}
\]
with the second vertical map being the canonical inclusion. The two pairs on the rows are free, and even the first projection is surjective. Also, the groups $\ktQuot$ and $\kQuot$ are both isomorphic to $\ZZ$, and $\kpi_\kS$ induces the identity as isomorphism. However, the restriction to the boundary is a strict inclusion.
\\[1ex]
\fbox{\ref{condition3}$\not\then$\ref{condition2}}
Take the following extension with surjective addition maps:
\[
  \begin{tikzcd}
    0 \arrow[r, hook]\arrow[d]& \RR_{\ge 0}\arrow[d] \\
    \RR_{\ge 0} \arrow[r, hook]& \RR_{\ge 0}^2 
  \end{tikzcd}
\]
with both maps $\RR_{\ge 0}\too \RR_{\ge 0}^2$ given by $t\mapsto (t,t)$. Even if the two pairs are free and the groups $\kQuot$ and $\ktQuot$ are isomorphic to $\RR$, the map induced by the vertical one is the zero map, so not an isomorphism.
\\[1ex]
 \fbox{\ref{condition3}$\then$\ref{condition1} if $\kpi_\kS$ is surjective.}
\label{lem:boundaryBijection}
 \emph{The restriction $\kpi_{\Rand{}{(\ktS)}}$ is well defined.} 
Let $\wt{b}\in\Rand{\ktT}{(\ktS)}$ and $b=\kpi(\wt{b})$. Then $b$ admits a unique decomposition into $s+\kt
\in \Rand{\kT}{(\kS)} \times \kT$
and, by surjectivity,  we may lift both summands 
to $\kts\in\ktS$ and 
$\ktt\in\ktT$, respectively.
Thus, $\wt{b}$ and $\kts+\ktt$ have the same image under $\kpi$, and~\ref{condition3} 
implies the existence of $\ktt_1, \ktt_2\in\ktT$ with
\[
\wt{b} - \ktt_1 = \kts+\ktt -\ktt_2\in\ktS.
\]
The hypothesis $\wt{b}\in\Rand{\ktT}{(\ktS)}$ enforces $\ktt_1=0$.
Hence $\wt{b}= \kts+\ktt -\ktt_2$. After applying $\pi$, this means
\[
b = \kpi(\wt{b}) = \kpi(\kts) + \kpi(\ktt) -\kpi(\ktt_2)
= s+\kt-\kpi(\ktt_2).
\]
Comparing with our original equation $b=s+\kt$, this implies
$\kpi(\ktt_2)=0$, i.e.\ $\ktt_2\in\ker\kpi_{\kT}=\{0\}$.
Hence $\kts+\ktt=\wt{b}\in\Rand{\ktT}{(\ktS)}$ which again  
enforces $\ktt=0$. Finally, we apply $\kpi$ to the equation
$\wt{b}=\kts$, leading to $b=s\in\Rand{\kT}{(\kS)}$.
\\[1ex]
 \emph{Injectivity.} 
Let $\wt{b}_1, \wt{b}_2\in\Rand{\ktT}{(\ktS)}$ with
$\kpi(\wt{b}_1)=\kpi(\wt{b}_2)$. By~\ref{condition3} we obtain elements
$\ktt_1, \ktt_2\in\ktT$ with
\[
\wt{b}_1 - \ktt_1 = \wt{b}_2-\ktt_2\in\ktS.
\]
Again, the defining property of $\Rand{\ktT}{(\ktS)}$ implies
$\ktt_1=\ktt_2=0$.
\\[1ex]
 \emph{Surjectivity.} 
Let $b \in \Rand{\kT}{(\kS)}$. By the surjectivity of $\kpi_\kS$, this may be lifted to an element
$\wt{b}=\kts+\ktt\in \Rand{\ktT}{(\ktS)} + \ktT$.
Applying $\kpi$ yields
$b=s+\kt\in\Rand{\kT}{(\kS)} + \kT$, thus $\kpi(\ktt)=\kt=0$.
Again we conclude that $\ktt\in \ker\kpi_{\kT}=\{0\}$.
\end{proof}

Note that the property~\ref{condition3} is called the integrality of the
embedding $\ktT\hookrightarrow\ktS$ in \cite[Remark 3.25]{logGross}.
There and in \cite[Proposition 4.1]{logKato}, this notion is related to 
the flatness among the associated semigroup algebras.

%%%%%%%%%%%%%%%
%%%%%%%%%%%%%%%
\begin{example}
  \begin{enumerate}[label={\bf \arabic*.}, ref={\arabic*.}]
  \item 
It could have been that both $(\kT,\kS)$ and $(\ktT,\ktS)$ are
\straight, but the diagram is not \cocartesian. For example take $\ktT=\kT=\kS=\RR_{\geqslant 0}$ and 
$\ktS=\RR^2_{\geqslant 0}$ containing $\ktT$ as the ray 
$\RR_{\geqslant 0}\cdot (1,1)$ with $\kpi=\frac{1}{2}\,(1,1)$. 
By Proposition~\ref{prop-ksumInj} we know that $(T,S)$
and $(\ktT,\ktS)$ are free but $\ko{\kpi}:\ktQuot\to\kQuot$ is 
not an isomorphism, i.e.\ even~\ref{condition2} fails.

\item Any \cocartesian\ diagram is automatically Cartesian, i.e.\ it follows that 
$\ktT=\kpi_\kS^{-1}(\kT)\subseteq\ktS$. 
However, as it can be seen in Example~\ref{ex-pairExtension},~\ref{item:ex-pair2}, 
this condition does not suffice.
\end{enumerate}
\end{example}

\subsection{Boundary independence}
\label{subs:boundaryIndep}
The goal of Section~\ref{sec:mainPlayers} is to construct a universal \cocartesian~ \aextension~ for any free pair. To this aim, we have to identify the essential structures and concepts that a \cocartesian~\aextension~ has to preserve. The first is the concept of independence. This is defined for tuples of elements in the boundary. The second is a family of special elements in the smaller semigroup ($\kT$, respectively $\ktT$) which can be defined in terms 
of the boundary, and has to be compatible with the bijection on the boundary.
\begin{definition}
  \label{def:boundaryIndependence}
  Let $(\kT,\kS)$ be a free pair of semigroups.  A  collection of $r$ (not necessarily distinct) boundary elements  $b_1,\dots,b_r\in \Rand{\kT}{\kS}$ is called \emph{boundary independent} if their sum is still in the boundary, that is if
  \[
    b_1+\dots + b_r = \bound(b_1+\dots+b_r).
  \]
In contrast, such a collection is called boundary \emph{dependent}, if it is not boundary independent, and \emph{minimally dependent} if it is dependent and  every proper subset is independent.
\end{definition}
Let $(\kT,\kS)$ be a free pair, and $(\ktT,\ktS)$ be a \cocartesian~ \aextension~ of it.
We thus have an induced bijection $\kpi^{-1}_\bound:\Rand{\kT}{(\kS)}\stackrel{\sim}{\to}\Rand{\ktT}{(\ktS)}$, and for every $b\in\Rand{\kT}{\kS}$ we simply denote
\[
  \awt{b}:=\kpi^{-1}_\bound(b).
\]
 Let us denote the retractions upstairs by $\widetilde{\bound}:\ktS\surj\Rand{\ktT}{\ktS}$ and $\widetilde{\lambda}:\ktS\surj\ktT$ respectively.

\begin{proposition}
  \label{prop:boundaryIndepEquivalence}
  For any \cocartesian~ \aextension~ $\kpi:(\ktT,\ktS)\too(\kT,\kS)$ and for any (not necessarily distinct) elements $b_1,\dots,b_r\in \Rand{\kT}{\kS}$ we have
  \begin{eqnarray*}
    \awt{\bound}(\awt{b}_1+\dots+\awt{b}_r) &=& \kpi^{-1}_\bound(\bound(b_1+\dots+b_r))\\
    \kpi_\kT\big(\awt{\lambda}(\awt{b}_1+\dots+\awt{b}_r)\big) &=& \lambda(b_1+\dots+b_r)\\
\awt{b}_1,\dots,\awt{b}_r\text{~are boundary independent} &\iff& b_1,\dots,b_r\text{~are boundary independent}.
  \end{eqnarray*}
  In particular,
  \[
    \kpi^{-1}_\partial(b_1)+\dots+\kpi^{-1}_\partial(b_r) - \kpi^{-1}_\partial\big(\bound(b_1+\dots+b_r)) \in \ktT.
  \]
Furthermore, if we choose $\awt{s}_1,\dots,\awt{s}_r\in\ktS$ and denote by $s_i:=\kpi(\awt{s}_i)$, then the first two relations  above still hold.
\end{proposition}
\begin{proof}
  From $\kpi(\awt{b}_1+\dots+\awt{b}_r)= b_1+\dots+b_r$, using the unique boundary decompositions we get
  \[    \kpi_\bound\big(\awt{\bound}(\awt{b}_1+\dots+ \awt{b}_r)\big)+ \kpi_T\big(\awt{\la}(\awt{b}_1+\dots+ \awt{b}_r)\big) =  \bound(b_1+\dots+b_r) + \lambda(b_1+\dots+b_r).   \]
  By the uniqueness of the decomposition, by the bijectivity of $\kpi_\bound$ and by $\ker\kpi_\kT=0$, we conclude.
\end{proof}

%%%%%%%%%%%%%%%%%%%%%%%%%%%%%%%%%%%%%%%%%%%%%%%%%%%%%%%%%%%%%%%%%%%%%%%%%%%
\subsection{The category of \straight\ \aextension s of a pair}\label{catStraight}
% \subsection{Induced \straight\ extensions}\label{indStraight}
Assume that $(\kT,\kS)$ is a \straight\ pair. Then the co-Cartesian extensions
$\kpi:(\ktT,\ktS)\to(\kT,\kS)$ form a 
\kbox{category $\straightExt_{(\kT,\kS)}$} where the morphisms are
defined in the obvious way. Moreover, we have the following construction
imitating base change from algebraic geometry and equipping
$\straightExt_{(\kT,\kS)}$ with the structure of being fibered in groupoids.

%%%%%%%%%%%%%%%
%%%%%%%%%%%%%%%
\begin{proposition}
\label{prop-groupoids}
Assume that $(\ktT,\ktS)\in\straightExt_{(\kT,\kS)}$ and that
$\kttpiT:\kttT\to\kT$ is another extension of $\kT$. Then,
for any semigroup homomorphism $f:\ktT\to\kttT$ over $\kT$,
there is a unique extension $\kttpiS:\kttS\to\kS$ such that
$\kttpi=(\kttpiT,\kttpiS):(\kttT,\kttS)\to(\kT,\kS)$
belongs to $\straightExt_{(\kT,\kS)}$ and that
$f$ extends to a morphism $(\ktT,\ktS)\to(\kttT,\kttS)$ in this category.
\[
  \begin{tikzcd}
     \ktT
     \arrow[hook]{r}
     \arrow{rd}
     \arrow[bend left=30]{rr}{\forall~f}
     & \ktS
     \arrow[bend left=30, dotted, crossing over]{rr}{}
     & \ktT'
     \arrow{ld}
     \arrow[hook,dotted]{r}
     & {\color{darkgreen}\ktS'}
     \arrow[dotted]{ld}
  \\
     & \kT
     \arrow[hook]{r}
     & \kS
     \arrow[crossing over,leftarrow]{lu}
     &    
   \end{tikzcd}
\]
\end{proposition}

\begin{proof}
It suffices to prove that the canonical map
$\Rand{\ktT}{\ktS}\times\kttT\to\ktS\oplus_\ktT\kttT$ is a bijection
where the latter denotes the push-out
$\ktS\oplus_\ktT\kttT:=(\ktS\times\kttT)/\!\sim\,$ defined by modding out
the equivalence relation 
\[
(\kts_1,\ktt_1')\sim(\kts_2,\ktt_2') :\iff
\exists \ktt_1,\ktt_2\in\ktT\!:\hspace{0.7em}
\kts_1+\ktt_2=\kts_2+\ktt_1
\hspace{0.7em}\mbox{and}\hspace{0.7em}
\ktt_1'+f(\ktt_1)=\ktt_2'+f(\ktt_2).
\]
However, it is straightforward to check that the assignment
$(\kts,\ktt')\mapsto (\bound\kts,\,f(\height\kts)+\ktt')$
yields a correctly defined inverse map
$\ktS\oplus_\ktT\kttT\to\Rand{\ktT}{\ktS}\times\kttT$.
\end{proof}

%%%%%%%%%%%%%%%%%%%%%%%%%%%%%%%%%%%%%%%%%%%%%%%%%%%%%%%%%%%%%%%%%%%%%%%%%%%
\subsection{Initial objects in $\straightExt_{(\kT,\kS)}$}
\label{catStraightInit}

The main result of this paper is the following.
\begin{theorem}
  \label{thm:universalObjectExists}
Le us assume that we are in the discrete setup , cf. Section~\ref{discreteSetup}.
The category of \cocartesian~\aextension s of $(\kT,\kS)$ contains an initial object.
\end{theorem}

We will provide a very explicit construction of this universal object in the discrete setup. We start analysing first the cone setup in Section~\ref{sectConeSetup}, where we get a terminal object, cf. Proposition~\ref{prop-versalCP}. 
Just to get an impression of what this initial object may look like we provide the following example.

%%%%%%%%%%%%%%%
%%%%%%%%%%%%%%%
\begin{example}
\label{ex-discreteB}
Let us return to Example~\ref{ex-discreteA} and Figure~\ref{fig:3},
i.e.\ $\kS=\braket{ [-2,1],\,[-1,1],\,[0,1],\,[1,1],\,[2,1]}$
with $\kT=\NN\cdot \kR$ and $\kR=[0,1]$. In Example~\ref{ex-MinFourB}
this semigroup  will be understood starting from the 
1-dimensional polytope $\kP=[-\frac{1}{2},\frac{1}{2}]\subset\RR$;
the link between these two approaches is that the polyhedral cone
$\sigma$ over $P\times\{1\}\subset\RR^2$ is dual to
$S_\RR=\sigma\dual$ from Example~\ref{ex-4Bilder} which contains
$\kS$ as the set of lattice points.
Anyway, in algebraic geometry, this setup gives rise to the toric 
singularity $X=\toric(\sigma)\subseteq\A^5_k$ which can, alternatively,
be understood as the vanishing set of the six minors encoded by
the condition
\[
\rank\left(\begin{array}{cccc}
z_{-2}& z_{-1} & z_0 & z_1\\ z_{-1} & z_0 & z_1 & z_2
\end{array}\right) \leq 1.
\]
The elements $[k,1]\in\kS$ can be recovered as the multidegrees
of the variables $z_i$. 
In \cite{m} we discuss the deformation theory of those toric singularities.
In this context, the present example became famous in the last century, 
because Pinkham has detected that the deformation space of $X$ admits two 
different components. In \cite{m} we will recall that this corresponds
to two different lattice friendly decompositions of $\kP$ as we will
meet them here in Section~\ref{sec:MinkowskiDec},
cf.\ Example~\ref{ex-calComponents}.
Finally, it comes full circle by the fact that these two decompositions
correspond to the following two \cocartesian~\aextension s of $(\kT,\kS)$,
which we represent in Figures~\ref{fig:artinComponent} and~\ref{fig:qGcomponent} only through the generators.
The blue points correspond to  $\kT$ and $\ktT$, respectively.
The semigroups are recovered from the pictures by taking the cone over the convex hull.
Figure~\ref{fig:6} in Example~\ref{ex-MinFourB} depicts this explicitly for  $\kS$.

\begin{figure}[h!]
  \newcommand{\hght}{3}%%this controls the height of the both pictures
  \centering
  \begin{minipage}{0.49 \textwidth}
    \centering
    \begin{tikzpicture}[scale=.6]
      \draw[color=gray,thin] (-2,0) -- (2,0);
      \draw[color=colsubsgr, thin] (0,\hght) -- (0,\hght+1)  (-5,\hght)--(-5,\hght+1);
      \fill[color=colsemigr, opacity = 0.1] (-2,\hght+1) -- (0,\hght+1) -- (2,\hght) -- (0,\hght) --cycle;
      \draw[color=gray, thin] (-2,\hght+1) -- (0,\hght+1) -- (2,\hght) -- (0,\hght) --cycle;
      
      \fill[color=colboundr]
      (-2,0) circle (0.1)      (-2,\hght+1) circle (0.1)     (-2,\hght+2.1) circle (0)
      (-1,0) circle (0.1)      (-1,\hght+1) circle (0.1)
      (1,0) circle (0.1)      (1,\hght) circle (0.1)
      (2,0) circle (0.1)      (2,\hght) circle (0.1);

      \fill[color=colsubsgr]
      (0,0) circle (0.1)     (0,\hght) circle (0.1)   (0,\hght+1) circle (0.1)
      (-5,0) circle (0.1)     (-5,\hght) circle (0.1)   (-5,\hght+1) circle (0.1);

      \node at (0,\hght/2) { $\downarrow$};
      \node at (-5,\hght/2) { $\downarrow$};
      \node at (-3.5,\hght+0.5) { $\hookrightarrow$};
      \node at (-3.5,0) { $\hookrightarrow$};

    \end{tikzpicture}
    \caption{The Artin component}
    \label{fig:artinComponent}
  \end{minipage}
  \begin{minipage}{0.49 \textwidth}
    \centering
    \begin{tikzpicture}[scale=.6]
      \draw[color=gray,thin] (-2,0) -- (2,0);
      \draw[color=colsubsgr, thin] (0,\hght) -- (0,\hght+1) (-5,\hght)--(-5,\hght+1);
      \fill[color=colsemigr, opacity = 0.1] (-2,\hght+2) -- (2,\hght) -- (0,\hght) --cycle;
      \draw[color=gray, thin] (-2,\hght+2) -- (2,\hght) -- (0,\hght) --cycle;

      \fill[color=colboundr]    
      (-2,0) circle (0.1)       (-2,\hght+2) circle (0.1)
      (-1,0) circle (0.1)       (-1,\hght+1) circle (0.1)
      (1,0) circle (0.1)       (1,\hght) circle (0.1)
      (2,0) circle (0.1)       (2,\hght) circle (0.1);
      \fill[color=colsubsgr]
      (0,0) circle (0.1)     (0,\hght) circle (0.1)   (0,\hght+1) circle (0.1)
      (-5,0) circle (0.1)     (-5,\hght) circle (0.1)   (-5,\hght+1) circle (0.1);
      \draw[color=colsubsgr, thin] (0,\hght) -- (0,\hght+1);
      \node at (0,\hght/2) { $\downarrow$};
      \node at (-5,\hght/2) { $\downarrow$};
      \node at (-3.5,\hght+0.5) { $\hookrightarrow$};
      \node at (-3.5,0) { $\hookrightarrow$};
      
    \end{tikzpicture}
    \caption{The qG component}
    \label{fig:qGcomponent}
  \end{minipage}
\end{figure}

Now, by Theorem~\ref{thm:universalObjectExists}, we know that both extensions
can be merged to a common one. This leads to a 4-dimensional semigroup,
i.e.\ to a semigroup filling a 4-dimensional polyhedral cone where
its 3-dimensional crosscut is depicted in Figure~\ref{fig:PinkhamCone};
see Example~\ref{ex-MinFourC} for the detailed calculations.
Note that this establishes a remarkable difference to the algebro-geometric
setup: There, the two deformation components cannot be dominated by a
higher-dimensional joint deformation. In Figure~\ref{fig:-4curve} we represent the core of this merger,
by drawing only the generators of the semigroups: on the left the generators of the subsemigroups, and on the right, those of the semigroups.
The colors in this figure deviate from the convention established in Figure~\ref{fig:1}:
we have drawn in gray the generators of given semigroups and in orange those of the initial object.

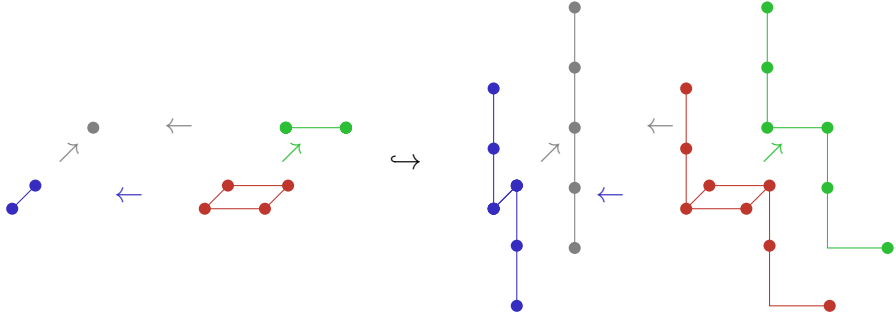
\begin{figure}[h!]
  \definecolor{intOrange}{rgb}{0.75,0.2118,0.1725}
  \definecolor{intBlue}{rgb}{0.2118,0.1725,0.75}
  \definecolor{intGreen}{rgb}{0.1725,0.75,0.2118}
  \centering
  \begin{tikzpicture}[scale=.8]
    % parameters for size of picture
    \newcommand{\xl}{-3.2}
    \newcommand{\xr}{0}
    \newcommand{\zf}{2}
    \newcommand{\zb}{-.5}
    \newcommand{\spc}{8}

        \fill[intOrange]
        (\xr,4,\zf+1) circle (0.1)
        (\xr,3,\zf+1) circle (0.1)
        (\xr,2,\zf+1) circle (0.1)
        (\xr,2,\zf+0) circle (0.1)
        (\xr+1,2,\zf+1) circle (0.1)
        (\xr+1,2,\zf) circle (0.1)
        (\xr+1,1,\zf) circle (0.1)
        (\xr+2,0,\zf) circle (0.1);
        \draw[thin,  color=intOrange] (\xr+0,4,\zf+1)--(\xr+0,2,\zf+1) -- (\xr+1,2,\zf+1) -- (\xr+1,2,\zf+0) -- (\xr+1,1,\zf+0);
        \draw[thin,  color=intOrange] (\xr+1,1,\zf+0)--(\xr+1,0,\zf+0) --(\xr+2,0,\zf+0);
        \draw[thin,  color=intOrange] (\xr+0,2,\zf+1)--(\xr+0,2,\zf+0) --(\xr+1,2,\zf+0);

      \fill[intBlue]
        (\xl,4,\zf+1) circle (0.1)
        (\xl,3,\zf+1) circle (0.1)
        (\xl,2,\zf+1) circle (0.1)
        (\xl,2,\zf+0) circle (0.1)
        (\xl,2,\zf+1) circle (0.1)
        (\xl,2,\zf+0) circle (0.1)
        (\xl,1,\zf+0) circle (0.1)
        (\xl,0,\zf+0) circle (0.1);
        \draw[thin,  color=intBlue] (\xl,4,\zf+1)--(\xl,2,\zf+1) -- (\xl,2,\zf+1) -- (\xl,2,\zf+0) -- (\xl,1,\zf+0);
        \draw[thin,  color=intBlue] (\xl,1,\zf+0)--(\xl,0,\zf+0) --(\xl,0,\zf+0);
        \draw[thin,  color=intBlue] (\xl,2,\zf+1)--(\xl,2,\zf+0) --(\xl,2,\zf+0);

        \fill[intGreen]
        (\xr+0,4,\zb) circle (0.1)
        (\xr+0,3,\zb) circle (0.1)
        (\xr+0,2,\zb) circle (0.1)
        (\xr+1,2,\zb) circle (0.1)
        (\xr+1,1,\zb) circle (0.1)
        (\xr+2,0,\zb) circle (0.1);
        \draw[thin,  color=intGreen] (\xr+2,0,\zb)--(\xr+1,0,\zb) -- (\xr+1,2,\zb) -- (\xr,2,\zb) -- (\xr,4,\zb);

         \fill[gray]
         (\xl,0,\zb) circle (0.1)
         (\xl,1,\zb) circle (0.1)
         (\xl,2,\zb) circle (0.1)
         (\xl,3,\zb) circle (0.1)
         (\xl,4,\zb) circle (0.1);
         \draw[thin,  color=gray] (\xl,0,\zb)--(\xl,4,\zb);
         
         %%%%%%% T level
         
         \fill[intOrange]
        (\xr-\spc,2,\zf+1) circle (0.1)
        (\xr-\spc,2,\zf+0) circle (0.1)
        (\xr-\spc+1,2,\zf+1) circle (0.1)
        (\xr-\spc+1,2,\zf) circle (0.1);
        \draw[thin,  color=intOrange](\xr-\spc+0,2,\zf+1) -- (\xr-\spc+1,2,\zf+1) -- (\xr-\spc+1,2,\zf+0) -- (\xr-\spc+0,2,\zf+0)--cycle;

        \fill[intBlue]
        (\xl-\spc,2,\zf+1) circle (0.1)
        (\xl-\spc,2,\zf+0) circle (0.1);
        \draw[thin,  color=intBlue] (\xl-\spc,2,\zf+1)--(\xl-\spc,2,\zf+0);

        \fill[intGreen]
        (\xr-\spc+0,2,\zb) circle (0.1)
        (\xr-\spc+0,2,\zb) circle (0.1)
        (\xr-\spc+1,2,\zb) circle (0.1)
        (\xr-\spc+1,2,\zb) circle (0.1);
        \draw[thin,  color=intGreen] (\xr-\spc+0,2,\zb)--(\xr-\spc+1,2,\zb);

        \fill[gray]
        (\xl-\spc,2,\zb) circle (0.1);

        \node at (\xl/2.2,2,\zf+.5) {\color{intBlue} $\leftarrow$};
        \node[anchor=north] at (\xl,2.35,\zf/4) {\color{gray} \rotatebox{45}{$\rightarrow$}};
        \node at (\xl/1.8,2,\zb) {\color{gray} $\leftarrow$};         
        \node[anchor=north] at (\xr+.5,2.35,\zf/4) {\color{intGreen} \rotatebox{45}{$\rightarrow$}};
       
        \node at (\xl/2.2 -\spc,2,\zf+.5) {\color{intBlue} $\leftarrow$};
        \node[anchor=north] at (\xl -\spc,2.35,\zf/4) {\color{gray} \rotatebox{45}{$\rightarrow$}};
        \node at (\xl/1.8 -\spc,2,\zb) {\color{gray} $\leftarrow$};         
        \node[anchor=north] at (\xr+.5 -\spc,2.35,\zf/4) {\color{intGreen} \rotatebox{45}{$\rightarrow$}};

        \node[anchor=west] at (\xl-\spc/3,2,\zf/2) {\color{black} $\hookrightarrow$};
  \end{tikzpicture}
  \caption[-4 curve]{The generators of the semigroups, and the way the fit together.}
  \label{fig:-4curve}
\end{figure}
In Figure~\ref{fig:PinkhamCone} we return to the colors used throughout the paper,
and depict the convex bodies which define the 2-, 3-, and 4-dimensional semigroups involved. 
\begin{figure}[!h]
  \centering
  \begin{tikzpicture}[scale=.8]
    \newcommand{\ha}{4} %front layer 
    \newcommand{\lo}{0} %back layer
      % dotted cube
      \draw[thin, dotted, color=gray] (0,0,\ha) -- (6,0,\ha) -- (6,4,\ha) -- (0,4,\ha)--cycle;%front square
      \draw[thin, dotted, color=gray] (0,0,\lo) -- (6,0,\lo) -- (6,4,\lo) -- (0,4,\lo)--cycle;% back square
      \draw[thin, dotted, color=gray] (0,0,\ha) -- (0,0,\lo);
      \draw[thin, dotted, color=gray] (6,0,\ha) -- (6,0,\lo);
      \draw[thin, dotted, color=gray] (0,4,\ha) -- (0,4,\lo);
      \draw[thin, dotted, color=gray] (6,4,\ha) -- (6,4,\lo);
      \draw[thin, dotted, color=colsubsgr] (0,2,\ha) -- (6,2,\ha) -- (6,2,\lo) -- (0,2,\lo)--cycle;%blue square

    % the 4-dimensional part:
    % visible faces:
    \fill[color = colsemigr, opacity=0.2] (4,4,\ha) -- (5,2,\ha) -- (4,2,\ha) --cycle;
    \fill[color = colsemigr, opacity=0.1] (6,0,\ha-\ha/4) -- (5,2,\ha-\ha/4) -- (4,4,\ha) -- (5,2,\ha) --cycle;
    \fill[color = colsemigr, opacity=0.4] (6,0,\ha-\ha/4) -- (5,2,\ha) -- (4,2,\ha) --cycle;
   
    \draw[thin,color=black] (6,0,\ha-\ha/4) -- (5,2,\ha-\ha/4) -- (4,4,\ha) -- (5,2,\ha) --cycle;
    \draw[thin,color=black] (6,0,\ha-\ha/4) -- (4,2,\ha) -- (4,4,\ha);
    \draw[thin,dashed, color=black] (6,0,\ha-\ha/4) -- (4,2,\ha-\ha/4) -- (4,4,\ha);
    \draw[thin, color=colsubsgr, opacity =1] (4,2,\ha) -- (5,2,\ha) -- (5,2,\ha-\ha/4);%%thin blue square
    \draw[thin, dashed, color=colsubsgr, opacity=1] (4,2,\ha) -- (4,2,\ha-\ha/4) -- (5,2,\ha-\ha/4); %%thin blue square
    
    \fill[color=colboundr]
    (6,0,\ha-\ha/4) circle (0.1)
    (5,1,\ha-\ha/4) circle (0.1)
    (4,3,\ha) circle (0.1)
    (4,4,\ha) circle (0.1);
    \fill[color=colsubsgr]
    (5,2,\ha-\ha/4) circle (0.1)
    (5,2,\ha) circle (0.1)
    (4,2,\ha-\ha/4) circle (0.1)
    (4,2,\ha) circle (0.1);
          
    % the Artin component:
    \fill[color = colsemigr, opacity=0.1] (0,0,\ha-\ha/4) -- (0,2,\ha-\ha/4) -- (0,4,\ha) -- (0,2,\ha) --cycle;
    \draw[color=black, thin] (0,0,\ha-\ha/4) -- (0,2,\ha-\ha/4) -- (0,4,\ha) -- (0,2,\ha) --cycle;
    \draw[thin, color=colsubsgr, opacity=1] (0,2,\ha) -- (0,2,\ha-\ha/4);%%thin blue line
    
    \fill[color=colboundr]
    (0,0,\ha-\ha/4) circle (0.1)
    (0,1,\ha-\ha/4) circle (0.1)
    (0,3,\ha) circle (0.1)
    (0,4,\ha) circle (0.1);
    \fill[color=colsubsgr]
    (0,2,\ha) circle (0.1)
    (0,2,\ha-\ha/4) circle (0.1);
    
    % the q-G component:
    \fill[color = colsemigr, opacity=0.1] (6,0,\lo)--(4,2,\lo) -- (4,4,\lo) --cycle;
    \draw[color=black, thin] (6,0,\lo)--(4,2,\lo) -- (4,4,\lo) --cycle;
    \draw[thin, color=colsubsgr, opacity =1] (4,2,\lo) --(5,2,\lo);%thin blue line
    
    \fill[color=colboundr]
    (6,0,\lo) circle (0.1)
    (5,1,\lo) circle (0.1)
    (4,3,\lo) circle (0.1)
    (4,4,\lo) circle (0.1);
    \fill[color=colsubsgr]
    (4,2,\lo) circle (0.1)
    (5,2,\lo) circle (0.1);

    % the cone over the -4 curve:
    \draw[thin, color=black] (0,0,\lo)-- (0,4,\lo);

    \fill[color=colboundr]
    (0,0,\lo) circle (0.1)
    (0,1,\lo) circle (0.1)
    (0,3,\lo) circle (0.1)
    (0,4,\lo) circle (0.1);
    \fill[color=colsubsgr]
    (0,2,\lo) circle (0.1);

\end{tikzpicture}
  \caption{The full picture: Artin- and qG-components as projections of the initial object.}
  \label{fig:ArtinComp_of-4}
  \label{fig:PinkhamCone}
\end{figure}
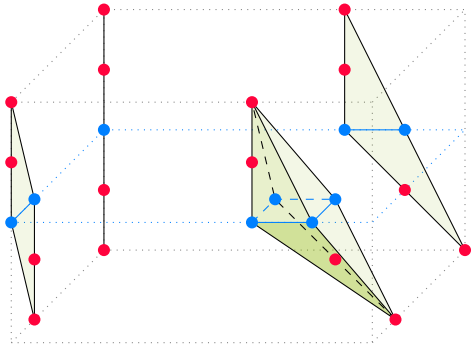

\end{example}
%%%%%%%%%%%%%%%%%%%%%%%%%%%%%%%%%%%%%%%%%%%%%%%%%%%%%%%%%%%%%%%%%%%%%%%%%%% 
%%%%%%%%%%%%%%%%%%%%%%%%%%%%%%%%%%%%%%%%%%%%%%%%%%%%%%%%%%%%%%%%%%%%%%%%%%%

\section{The cone setup}\label{sectConeSetup}
%%%%%%%%
%%%%%%%%%%%%%%%%%%%%%%%%%%%%%%%%%%%%%%%%%%%%%%%%%%%%%%%%%%%%%%%%%%%%%%%%%%%
%%%%%%%%%%%%%%%%%%%%%%%%%%%%%%%%%%%%%%%%%%%%%%%%%%%%%%%%%%%%%%%%%%%%%%%%%%%

%%%%%%%%%%%%%%%%%%%%%%%%%%%%%%%%%%%%%%%%%%%%%%%%%%%%%%%%%%%%%%%%%%%%%%%%%%%
\subsection{Dualizing the cone setup}
\label{dualWorld}
In the present section, we deal exclusively with the situation introduced in (\ref{coneSetup}).
The main result is the construction of a terminal object in a certain category (Proposition~\ref{prop-versalCP}).
This proposition is a  much easier to prove analog of  Theorem~\ref{thm:universalObjectExists}.
One of the striking features of the cone setup 
is that it allows dualization of both
the cones $\kT$ and $\kS$
and their ambient vector spaces $\kSpanS:=\kS-\kS=\spann{\RR}{\kS}$ 
and $\kSpanT:=\kT-\kT=\spann{\RR}{\kT}$,
respectively. In particular, considering
$\kiota:\kSpanT\hookrightarrow\kSpanS$ with $\kiota(\kT)\subseteq\kS$,
there is a dual linear map
\[
\kpr:\kSpanSD \surj \kSpanTD
\hspace{1em}\mbox{with}\hspace{1em}
\kpr(\kSD)\subseteq\kTD.
\]
While $\kpr$ is always surjective on the level of vector spaces,
the surjectivity of its restriction $\kpr_+:\kSD \to \kTD$ to the 
level of semigroups
is equivalent to the property $\kS\cap \kSpanT
% =\kS\cap \spann{\RR}{\kT}
=\kT$.
This is, however, automatically fulfilled for \straight\ pairs,
cf.~Definition~\ref{def-straight}.

%%%%%%%%%%%%%%%
%%%%%%%%%%%%%%%
\begin{lemma}
\label{lem-straightPrSurj}
If $(\kT,\kS)$ is \straight, then $\kS\cap \spann{\RR}{\kT}=\kT$.
\end{lemma}

%%%%%%%%%%%%%%%
%%%%%%%%%%%%%%%
\begin{proof}
If $t_1-t_2\in \kS$ (with $t_i\in\kT$),
then this element can be decomposed into
$t_1-t_2=b+t$ with $b\in\Rand{\kT}{(\kS)}$ and $t\in\kT$.
Hence $0+t_1=b+(t+t_2)$, but this displays two decompositions
within $\kS=\Rand{\kT}{(\kS)}+\kT$. Thus, $0=b$, and this means
$t_1-t_2=0+t\in\kT$.
\end{proof}

%%%%%%%%%%%%%%%%%%%%%%%%%%%%%%%%%%%%%%%%%%%%%%%%%%%%%%%%%%%%%%%%%%%%%%%%%%%
\subsection{A dual characterization of \straightness}
\label{collCharStraight}
In the cone setup, \straightness\
can be characterized by the following enhancement of the surjectivity of
$\kpr_+:\kSD \to \kTD$.
%%%%%%%%%%%%%%%
%%%%%%%%%%%%%%%
\begin{proposition}
\label{prop-dualStraight}
The pair $(\kT,\kS)$ is \straight\ if and only if
$\kpr_+:\kSD \surj \kTD$  is surjective and maps faces onto faces.
\end{proposition}
%%%%%%%%%%%%%%%
%%%%%%%%%%%%%%%
\begin{proof}
For each $s\in\kS$ we define $\kT_s:=(s+\spann{\RR}{\kT})\cap\kS$.
Note that, for every $s'\in\kT_s$, one has $\kT_s=\kT_{s'}$.
It is a polyhedron in (a translate of) $\spann{\RR}{\kT}$ with tail cone $\kT$,
i.e.\ its normal fan $\normal(\kT_s)$ has $\kTD$ as its support,
cf.\ Subsection~\ref{spaceMinkSummands} for reviewing these notions.
Now, the \straightness\ of the pair $(\kT,\kS)$ is equivalent to the fact
that the polyhedra $\kT_s$ are cones, i.e.\ that their normal fans
$\normal(\kT_s)$ equal the face fan of their supporting cone $\kTD$.
That is, for each $s\in\kS$ the normal cone $\normal(s,\kT_s)$
has to be a face of $\kTD$.
% \\
On the other hand, these normal cones equal 
\[
\normal(s,\kT_s)=\{\varphi:\spann{\RR}{\kT}\to\RR\kst
\varphi(s)\leq\varphi(\kT_s)\}
\]
where the description requires choosing some 
lift of $\varphi\in\kSpanTD$ 
along the surjection $\kpr:\kSpanSD\surj\kSpanTD$.
Hence, these normal cones are the images under $\kpr$ of the faces 
\[
\face(\kSD,s):=\kSD\cap s^\bot=\{\varphi\in\kSD\kst \varphi(s)= 0\}.
\]
Finally, since these faces are running, with $s\in\kS$, 
through all faces $F\leq\kSD$, we get exactly
the characterization claimed in the proposition.
\end{proof}

%%%%%%%%%%%%%%%%%%%%%%%%%%%%%%%%%%%%%%%%%%%%%%%%%%%%%%%%%%%%%%%%%%%%%%%%%%%
\subsection{\Straightness\ and Minkowski linearity}
\label{collCharMinLin}
We have already seen that it was important to distinguish between the map
$\kpr:\kSpanSD \to \kSpanTD$ on the level of vector spaces and
its restriction $\kpr_+:\kSD\to\kTD$.
This is particularly relevant when we deal with fibers.
For an element $\xi\in\kTD$ we will call
\[
\kpr_+^{-1}(\xi):=\kpr^{-1}(\xi)\cap\kSD
\]
the \kbox{{\em positive fiber}} of $\xi$ (under $\kpr$).

%%%%%%%%%%%%%%%
%%%%%%%%%%%%%%%
\begin{definition}
\label{def_minkLinear}
We call the pair $(\kT,\kS)$ \kbox{Minkowski linear} if
$\kpr_+:\kSD \to \kTD$ is surjective, and if for each
$\xi,\xi'\in\kTD$ we have
$\kpr_+^{-1}(\xi)+\kpr_+^{-1}(\xi')=\kpr_+^{-1}(\xi+\xi')$.
\end{definition}

Note that the inclusion ``$\subseteq$'' as well as equality on the level
of vector spaces, i.e.\ replacing $\kpr_+$ by $\kpr$,
is always satisfied. However, the linearity among the positive fibers
becomes equivalent to \straightness.

%%%%%%%%%%%%%%%
%%%%%%%%%%%%%%%
\begin{proposition}
\label{prop-straightMinkLin}
The pair $(\kT,\kS)$ is \straight\ if and only if it is Minkowski linear.
\end{proposition}

%%%%%%%%%%%%%%%
%%%%%%%%%%%%%%%
\begin{proof}
Let us visualize the present situation.
Denoting $\kQuotS:=\kqr(\kS)\subseteq\kSpanS/\kSpanT$,
we obtain the exact sequence
\[
\begin{tikzcd}
  0  \ar[r]  &
  \kSpanT\ar[r]&
  \kSpanS\ar[r]{}{\kqr}&
  \kSpanS/\kSpanT\ar[r]
  &0
  \\
  &
  \kT_s\ar[r] \ar[hook,u] & 
  \kS \supseteq \Rand{\kT}{\kS} \ar[twoheadrightarrow,r]{\kqr_+}\ar[hook,u]&
  \kQuotS\ni\kqr(s) \ar[hook,u]{}&  
\end{tikzcd}
\]
where, for an element $s\in\kS$,
the positive fiber $\kqr_+^{-1}(\kqr(s))$ is just another 
way of writing $\kT_s$ from the proof of Proposition~\ref{prop-dualStraight}.
In particular, for a fixed $s$, the linearity of the function
\[
\Phi_s:
\varphi\mapsto \min\Braket{ \kqr_+^{-1}(\kqr(s)), \, \varphi}
\]
is equivalent to the fact that $\kT_s$ is a shifted copy of $\kT$.
Hence, the linearity of $\Phi_s$ for all $s\in\kS$ is equivalent to
the \straightness\ of the pair $(\kT,\kS)$.
% \\[1ex]
%
On the other hand, the dual picture is
\[
  \begin{tikzcd}
    0&
    \kSpanTD\ar[l]&
    \kSpanSD\ar[l,swap]{}{\kpr}&
    \kSpanT^\bot \ar[l] &
    0\ar[l]
    \\
    &
    \xi\in\kTD \ar[hook,u]&
    \kSD \ar[two heads ,l,swap]{}{\,\kpr_+} \ar[hook,u]&
    \kpr_+^{-1}(\xi). \ar[hook',l] \ar[hook,u]&
  \end{tikzcd}
\]
Writing $\xi=\kpi(\varphi)$ and $\xi'=\kpi(\varphi')$,
the linearity of
\[
\Psi_s:
\varphi\mapsto
\min\Braket{ s, \, \kpr_+^{-1}(\kpr(\varphi))}
\]
means that the polyhedra $\kpi_+^{-1}(\xi+\xi')$
and $\kpi_+^{-1}(\xi)+\kpi_+^{-1}(\xi')$ provide the same
values after applying $\min\braket{ s,\kbb}$ for all $s\in\kS$. i.e.,
that both polyhedra coincide. That is, the linearity
of $\Psi_s$ for all $s\in\kS$ is equivalent to the Minkowski linearity of the
pair $(\kT,\kS)$.
\\[1ex]
Finally, starting with two elements $s\in\kS$ and $\varphi\in\kSD$,
we obtain from the proof of \cite[Proposition 8.5]{tvar0}
the equality
\[
\min\Braket{ \kqr_+^{-1}(\kqr(s)), \, 
\kpr_+^{-1}(\kpr(\varphi))} = 0,
\]
or the equivalent version
\[
\Phi_s(\varphi)+\Psi_s(\varphi)\;=\;
\min\Braket{\kqr_+^{-1}(\kqr(s)), \, \varphi}
+
\min\Braket{ s, \, \kpr_+^{-1}(\kpr(\varphi))}
\;=\;
\braket{ s,\varphi}.
\]
It follows that the linearity of $\Phi_s$ is equivalent to the linearity of $\Psi_s$.
\end{proof}

%%%%%%%%%%%%%%%%%%%%%%%%%%%%%%%%%%%%%%%%%%%%%%%%%%%%%%%%%%%%%%%%%%%%%%%%%%%
\subsection{The cone of Minkowski summands}
\label{spaceMinkSummands}
Let us assume that $(\kT,\kS)$ is \straight, meaning  that the map $\kpr_+:\kSD \to \kTD$ is
Minkowski linear and, in particular, surjective.
We fix a ray in the interior of the cone $\kTD\subseteq\kSpanTD$,
that is we fix a linear map $\kpiD:\RR_{\geqslant 0}\hookrightarrow\kTD$ such that
$\kpiD(1)\in\innt\kTD$. The preimage $\kP:=\kpr_+^{-1}(\kpiD(1))\subseteq\kSD$
can be understood, well-defined up to some shift, as a polyhedron
in $\kSpanT^\bot=:\kQuotRD$. Consequently, the preimage of the whole ray
$\kpiD^*(\kSD):=\kpr_+^{-1}\big(\kpiD(\RR_{\geqslant 0})\big)$ equals
$\sigma:=\cone(\kP)$.

%%%%%%%%%%%%%%%
%%%%%%%%%%%%%%%
\begin{remark}
\label{rem-notationMN}
In Sections~\ref{straightPairs} and~\ref{extStraightPairs}
we had originally denoted $\kQuot=(\kS-\kS)/(\kT-\kT)$.
This fits perfectly well for the discrete setup: $\kQuot$ is then a finitely
generated Abelian group, and in Subsection~\ref{catStraightInit} 
we had denoted by $\kQuotR:=\kQuot\otimes_{\Z}\RR$
the associated vector space. However, in the cone setup,
$(\kS-\kS)/(\kT-\kT)=\kSpanS/\kSpanT$ is already an $\RR$-vector space, and it
seems appropriate to denote it by $\kQuotR$ instead of $\kQuot$.
The same applies for the dual gadgets $\kQuotD$ and $\kQuotRD$,
i.e.\ in particular $\kSpanT^\bot$ becomes $\kQuotRD$.
\end{remark}

The tail cone 
$
\,\tail(\kP)=\{a\in\kQuotRD\kst a+\kP\subseteq \kP\}
$
equals $\ker(\kpr_+)=\kSD\cap\kQuotRD=\kQuotS\dual$, and,
more general, this is the common tail cone of all other positive fibers
$\kpr_+^{-1}(\xi)$. 
In particular, $\kQuotS$ is the common support of their
normal fans $\normal\big(\kpr_+^{-1}(\xi)\big)$.

%%%%%%%%%%%%%%%
%%%%%%%%%%%%%%%
\begin{definition}
\label{def-MinkSum}
A (convex) polyhedron $Q$ with $\tail(Q)=\tail(\kP)$ is called a \kbox{Minkowski summand} of 
$\RR_{\geqslant 0}\cdot \kP$ if there is another polyhedron $Q'$
and a scalar $\lambda\in\RR_{\geqslant 0}$ such that $Q+Q'=\lambda\cdot \kP$.
\end{definition}

It is well-known fact that $Q$ is a Minkowski summand of 
$\RR_{\geqslant 0}\cdot\kP$ if and only if
the normal fan $\normal(\kP)$ is a refinement, i.e.\ a
subdivision of $\normal(Q)$. 
Exactly this property applies to all positive fibers $Q=\kpr_+^{-1}(\xi)$.
For interior points $\xi\in\innt\kTD$, we even have
$\normal\big(\kpr_+^{-1}(\xi)\big)=\normal(\kP)$.
In \cite{versalG} we have constructed a linear surjective map of polyhedral
cones 
\[
\kpr_{\kMC}:\kMtC(\kP)\surj \kMC(\kP)
\] 
where the elements $\xi\in \kMC(\kP)$ of the target
parametrize the set of (translation classes of)
Minkowski summands $\kP_\xi$ of $\kP$.
The summands $\kP_\xi$ are encoded via the 
associated dilation factors $t_{ij}(\xi)\in\RR_{\geqslant 0}$ of the bounded edges 
$\vect{\aed}_{ij}=v^j-v^i$ connecting the vertices $v^i,v^j\in\kP$. 
For instance $t_{ij}(\xi)=1$ for all $i,j$ leads to
$\one\in \kMC(\kP)$ with $\kP_{\one}=\kP$.
In general, the parameters $t_{ij}(\xi)$ are supposed to meet the
closing conditions $\sum_it_{i,i+1}(\xi)\cdot \vect{\aed}_{i,i+1}=0$ along the 
oriented boundaries of all 2-dimensional, compact faces of $\kP$.
The source cone of $\kpr_\kMC$ is the ``universal Minkowski summand'';
it is defined as
\[
  \kMtC(\kP):=\{(\xi,w)\kst \xi\in\kMC(\kP),\; w\in\kP_\xi\},\;
  \quad\mbox{i.e.\ }
\kpr_{\kMC}^{-1}(\xi)=\{\xi\}\times \kP_\xi.
\]
While the Minkowski summands are only well-defined up to translation,
one can make the previous definition precise by fixing a vertex $v_*\in\kP$
and placing all associated vertices $(v_*)_\xi\in\kP_\xi$ in the origin.
In general, every vertex $v\in\kP$ provides
a linear section $v:\kMC(\kP)\hookrightarrow\kMtC(\kP)$
of $\kpr_{\kMC}$ sending $\xi\mapsto v_\xi$.
The formula
$v^j_\xi-v^i_\xi=t_{ij}(\xi)\cdot(v^j-v^i)$ summarizes the situation.

%%%%%%%%%%%%%%%
%%%%%%%%%%%%%%%
\begin{remark}
\label{rem-propsCQ}
\begin{enumerate}[label={(\roman*)}]
\item\label{item:pcq}  The special element $\one\in\kMC(\kP)$, representing $\kP$ itself,
provides a linear embedding $\kpiD_{\kMC}:\RR_{\geqslant 0}\hookrightarrow\kMC(\kP)$.
As in the beginning of this subsection, the fiber product, i.e.\
the preimage of this ray $\,\kpr_{\kMC}^{-1}\big(\kpiD_{\kMC}(\RR_{\geqslant 0})\big)$
equals $\sigma=\cone(\kP)$.
\[
\xymatrix@C=3em@R=3ex{
\cone(\kP) \ar[r] \ar[d]&
\kMtC(\kP) \ar[d]^-{\kpr_{\kMC}} 
\\
\RR_{\geqslant 0} \ar@{^(->}[r]^-{\kpiD_{\kMC}} &
\kMC(\kP) 
}
\hspace{6em}
\xymatrix@C=3em@R=3ex{
\cone(\kP_0*\ldots*\kP_\minks) \ar[r] \ar[d]&
\kMtC(\kP) \ar[d]^-{\kpr_{\kMC}} 
\\
\RR^{\minks+1}_{\geqslant 0} \ar[r] &
\kMC(\kP) 
}
\]
\item \label{item:rem-propsCQ}If $\kP=\kP_0+\ldots+\kP_\minks$ is a Minkowski decomposition of $\kP$,
then the summands induce elements
$[\kP_0],\ldots,[\kP_\minks]\in\kMC(\kP)$, and
thus a linear map $\RR_{\geqslant 0}^{m+1}\to\kMC(\kP)$ sending the
$i$-th unit vector $e^i\mapsto [\kP_i]$. The 
fiber product becomes equal to the cone over the
so-called \kbox{Cayley product} $\kP_0*\ldots*\kP_\minks$.
\end{enumerate}
\end{remark}

%%%%%%%%%%%%%%%%%%%%%%%%%%%%%%%%%%%%%%%%%%%%%%%%%%%%%%%%%%%%%%%%%%%%%%%%%%%
\subsection{A terminal object in the cone setup}
\label{catMinkLin}
This subsection is the dual of the cone setup variant of Subsection~\ref{catStraightInit}. 
Let $\kP$
be a rational, convex polyhedron in some $\RR$-vector space $\kQuotRD$; 
for instance, it could arise from the situation in
Subsection~\ref{spaceMinkSummands}. 
Taking the height induces a natural 
map $\kR:\cone(\kP)\to\RR_{\geqslant 0}$. Then the pairs $(\kpr_+,\kpiD)$
consisting of a surjective homomorphism of polyhedral cones 
$\kpr_+:\kMtC\surj \kMC$ 
and an embedding $\kpiD:\RR_{\geqslant 0}\hookrightarrow \kMC$ form a \kbox{category $\MinkFam_{\kP}$}. 
The embedding $\kpiD$ gives thus an element $\kpiD(1)\in \kMC$, such that
\begin{enumerate}[label=(\roman*)]%
\item for $\xi,\xi'\in \kMC$ one has $\kpr_+^{-1}(\xi)+\kpr_+^{-1}(\xi')=
\kpr_+^{-1}(\xi+\xi')$ and
\item
$\kR:\cone(\kP)\to\RR_{\geqslant 0}$ is obtained  from $\kpr_+:\kMtC\surj \kMC$
via base change $\kpiD:\RR_{\geqslant 0}\hookrightarrow \kMC$, that is 
$\cone(\kP)=\kMtC\times_{\kMC}\RR_{\geqslant 0}$
\end{enumerate}
The two examples
$\big[\kR:\cone(\kP)\to\RR_{\geqslant 0}\ni 1\big]$ and
$\big[\kpr_{\kMC}:\kMtC(\kP)\surj\kMC(\kP)\ni\one\big]$ 
yield two objects in $\MinkFam_{\kP}$.
Moreover, if $\kP$ arises from Subsection~\ref{spaceMinkSummands},
then also $\big[\kpr_+:\kSD \to \kTD\ni\kpiD(1)\big]$ becomes an object
in $\MinkFam_{\kP}$.
\\[1ex]
By Proposition~\ref{prop-straightMinkLin}, $\MinkFam_{\kP}$
equals the opposite category $\straightExt_{\kP}^{\opp}$
of the cone setup variant of
$\straightExt_{\kP}=\straightExt_{(\kT,\kS)}$ from
Subsection~\ref{catStraightInit}.
Amplifying this comparison, the (dual) analog to 
Proposition~\ref{prop-groupoids} is just base
change. Moreover, it is clear that $[\kR,1]\in\MinkFam_{\kP}$ is an 
initial object.
However, the true analog to Theorem~\ref{thm:universalObjectExists} (restricted to this setting) is the existence of a 
terminal object in $\MinkFam_{\kP}$.

%%%%%%%%%%%%%%%
%%%%%%%%%%%%%%%
\begin{proposition}
\label{prop-versalCP}
The pair $[\kpr_{\kMC},\one]$ is a terminal object in $\MinkFam_{\kP}$.
That is, for any $[\kpr_+:\kMtC\surj \kMC\ni\kpiD(1)]$ in the category
$\MinkFam_{\kP}$, there is a \kbox{unique} linear 
$\kpiD':\kMC\to\kMC(\kP)$ such that
$\kpr_+$ is induced from $\kpr_{\kMC}$ via~$\kpiD'$
\[
\xymatrix@C=3em@R=3ex{
\cone(\kP) \ar[d]^-{\kR}\ar[r] \ar@/^1.3pc/[rr]^(0.4){} &
\kMtC \ar[r]^-{} \ar[d]^-{\kpr_+} & \kMtC(\kP) \ar[d]^-{\kpr_{\kMC}}
\\
 \RR_{\geqslant 0} \ar[r]^-\kpiD \ar@/_1.3pc/[rr]^(0.4){\kpiD_{\kMC}}&
\kMC \ar[r]^-{\kpiD'} &\kMC(\kP)
}
\]
and $\kpiD_{\kMC}(1)=\one=(\kpiD'\circ{\kpiD})(1)$.
Moreover, the map $\kMtC\to\kMtC(\kP)$
is supposed to induce the identity map
$\id_{\kP}:\kpr_+^{-1}(\kpiD(1))\to\kP_{\one}$
on the two distinguished fibers.
\end{proposition}

%%%%%%%%%%%%%%%
%%%%%%%%%%%%%%%
\begin{proof}
  Let $\xi\in \kMC$. Then, since we have that $\kpiD(1)$ is an interior point of $\kMC$,
there is an $n\in\NN$ such that $\xi':=\kpiD(n)-\xi\in \kMC$.
That is, by Minkowski linearity, the decomposition $\kpiD(n)=\xi+\xi'$ within
$\kMC$ provides a Minkowski decomposition
$n\cdot\kP=\kpr_+^{-1}(\xi)+\kpr_+^{-1}(\xi')$, i.e.\ $\kpr_+^{-1}(\xi)$ is
a Minkowski summand of a scalar multiple of $\kP$.
Now, since the points of $\kMC(\kP)$ are in a one-to-one correspondence
to the Minkowski summands of scalar multiples of $\kP$,
the polyhedron $\kpr_+^{-1}(\xi)$ corresponds to a unique 
$\kpiD'(\xi)\in\kMC(\kP)$.
This establishes the map $\kpiD'$. It is clearly additive, and
one easily checks the remaining properties.
\end{proof}

%%%%%%%%%%%%%%%%%%%%%%%%%%%%%%%%%%%%%%%%%%%%%%%%%%%%%%%%%%%%%%%%%%%%%%%%%%%
%%%%%%%%%%%%%%%%%%%%%%%%%%%%%%%%%%%%%%%%%%%%%%%%%%%%%%%%%%%%%%%%%%%%%%%%%%%
%%%%%%%%

%%%%%%%%%%%%%%%%%%%%%%%%%%%%%%%%%%%%%%%%%%%%%%%%%%%%%%%%%%%%%%%%%%%%%%%%%%%%%
%%%%%%%%%%%%%%%%%%%%%%%%%%%%%%%%%%%%%%%%%%%%%%%%%%%%%%%%%%%%%%%%%%%%%%%%%%%%%
%%%%%%%%%%%%%
\section{The discrete setup}
\label{discreteSetupStart}
%%%%%%%%%%%%%
%%%%%%%%%%%%%%%%%%%%%%%%%%%%%%%%%%%%%%%%%%%%%%%%%%%%%%%%%%%%%%%%%%%%%%%%%%%%%
%%%%%%%%%%%%%%%%%%%%%%%%%%%%%%%%%%%%%%%%%%%%%%%%%%%%%%%%%%%%%%%%%%%%%%%%%%%%%

%%%%%%%%%%%%%%%%%%%%%%%%%%%%%%%%%%%%%%%%%%%%%%%%%%%%%%%%%%%%%%%%%%%%%%%%%%%
\subsection{The pair of semigroups associated to a polyhedron}
\label{buildSg}
Let $\kQuotD$ be a lattice of rank $d$, that is a finitely generated free 
Abelian group $\kQuotD\simeq \ZZ^d$, 
and let $\kQuot=\Hom_\ZZ(\kQuotD,\ZZ)$ be the dual lattice.
We denote  the ambient real vector spaces by  $\kQuotRD:=\kQuotD\otimes_\ZZ\RR\simeq \RR^d$, respectively by $\kQuotR:=\kQuot\otimes_\ZZ \RR$.
Let $\kP \subset \kQuotRD$ be a rational polyhedron, which means that $\kP$ is the intersection of finitely many halfspaces defined by linear inequalities with rational coefficients. 
Recall from Subsection~\ref{minkPolyhedra} that
polyhedra are not necessarily bounded, but 
we will assume that they have at least one vertex. 
Then, they split into a Minkowski sum
\[
  \kP=\kP^c+\tail(\kP)
\]
where $\kP^c$ is 
the (bounded) convex hull of the vertices of $\kP$, 
and $\tail(\kP)$ is a pointed polyhedral cone.
We associate a cone to the polyhedron $\kP$ by embedding it in the hyperplane 
of height one of $\kQuotRD\oplus \RR$, and taking the 
(closure of) the cone over it:
\[
  \kcP:=\cone(\kP)\subseteq\kQuotRD\oplus\RR.
\]
The dual cone will be denoted by
\(
  \kcP\dual=\cone(\kP)\dual\subseteq\kQuotR\oplus\RR.
\)
We call $\kp$ the canonical projection $\kQuotRD\oplus\RR\surj\RR$. Alternatively, we
can understand this map as the element $\kp=[0,1]\in \kQuot\oplus\ZZ$, 
which defines a ray $\kp:\RR_{\geq 0}\hookrightarrow \kcP\dual$. 
Note that the intersections $\sigma\cap R^\bot=:\sigma\cap[R=0]$
and $\sigma\cap[R=1]$ recover $\tail(\kP)$ and $\kP$, respectively.
The pair of discrete semigroups associated to $\kP$ is given by 
\begin{equation}
  \label{eq:pairDefinedBykP}
  \kT:=\NN,\qquad  \kS:=\cone_\ZZ(\kP)\dual:=\cone(\kP)\dual\cap(\kQuot\oplus\ZZ),\qquad\text{and}\quad  \kp:\kT \hookrightarrow \kS.
\end{equation}
By Proposition~\ref{prop-ksumInj}, this forms a \straight\ pair.
One goal of this work is to construct explicitly the universal \cocartesian~\aextension\  of $\kT\hookrightarrow\kS$  (cf.~Theorem~\ref{th-initialObject}). The inspiration for the constructions to come was the
knowledge of the space of infinitesimal deformations $T^1$
of the toric singularity $\toric(\kcP)=\Spec\CC[\kS]$ in the
multidegree  $-\kp\in M$. 
\subsection{The structure of $\kcP\dual$}
\label{sigmaDual} 
We want to understand the relative boundaries of both $\kcP\dual$ and $\kcP\dual\cap(\kQuot\oplus\ZZ)$, with respect to the rays given by $\kp$. To this aim we introduce the following.

 \begin{definition}\label{d:eta(c)}
For every linear form $\kc\in \tail(\kP)\dual\subseteq\kQuot_\RR$ define
  \begin{eqnarray*}
\keta(\kc)  & := & -\min_{v\in\kP}\braket{ v,\kc} \in\RR,\\
\ketaZ(\kc) & := & \ceil{\keta(\kc)} \in\ZZ, 
  \end{eqnarray*}
 where $\ceil{\keta}$ denotes the ceiling, i.e.\ the least integer not smaller than 
the real number $\keta$. 
It is easy to see that the set $f_c:=\{v\in\kP~:~-\braket{ v,\kc} = \keta(\kc)\}$ is a face and contains at least one vertex. We choose and fix one such vertex and denote it by $v(\kc)$. So we have   
\[
  \keta(\kc) =  -\braket{ v(\kc),\kc} \le \ketaZ(\kc).
\]
The reason for not taking any $\kc\in\kQuot_\RR$ is that we want $\kc$ to be bounded below on $\kP$. Obviously, when $\kP$ is compact, $\tail(\kP)\dual=\kQuot_\RR$.
\end{definition}

\begin{notation}
  \label{not:verticesLattice/Nonlattice}
  We denote the set of all vertices of $\kP$ by $\Vrtx(\kP)$. We will need to distinguish between lattice and non-lattice vertices, and for this use the notation
\[
  \VrtxZ(\kP) := \Vrtx(\kP)\cap \kQuotD\qquad \VrtxnZ(\kP) := \Vrtx(\kP)\setminus \VrtxZ(\kP).
\]
Moreover, for real numbers $z\in \RR$ we will use the following notation:
\begin{equation}\label{eq:etaFrac}%{eq fun f}
\etaFrac{z}:=\roundup{z}-z.
\end{equation}
\end{notation}

For $\kc\in \tail(\kP)\dual\cap\kQuot$ and $v(\kc)\in\VrtxZ(\kP)$ we have $\keta(\kc)=\ketaZ(\kc)$.
If $\keta(\kc)\notin\ZZ$, then the integer $\ketaZ(\kc)$ equals the 
value of $-\braket{\kbb,\kc}$ at some point sitting on a moving affine
$\kc$-hyperplane \emph{before} it reaches our polyhedron $\kP$.
\begin{remark}
  \label{rem:firstPropertiesOfeta}
  \begin{enumerate}[label=(\roman*)]
  \item If $0\in\kP$, then $\keta(\kc)\geqslant 0$.
  \item The elements $[\kc,\keta(\kc)]$ form the relative boundary $\sigmaDBd\subseteq\partial\kcP\dual$ (the inclusion is strict if the tail cone is not trivial). This implies 
  \[
    \kcP\dual=\{[\kc,\keta(\kc)]\kst \kc\in \tail(\kP)\dual\} + \RR_{\geqslant 0}\cdot[\underline{0},1].
  \]
  \item The semigroup  $\kS:=\kcP\dual\cap(\kQuot\oplus\ZZ)$ is generated by the Hilbert basis of $\kcP\dual$, which has the form
  \[
    \big\{[\ahb_1,\ketaZ(\ahb_1)],\dots,[\ahb_k,\ketaZ(\ahb_k)],[\underline{0},1]\big\},
  \]
with uniquely determined elements $\ahb_i\in \tail(\kP)\dual\cap\kQuot$. We will use the font ``$\ahb$'' only for the Hilbert basis elements.
\item We defined $\kT:=\NN$ and embedded it in $\kS$ by $1\mapsto [\underline{0},1]$. For $\kc\in\tail(P)\dual\cap\kQuot$, the elements $[\kc,\ketaZ(\kc)]$ are always in $\Rand{\kT}{\kS}$, but not in $\sigmaDBd$ whenever $\keta(\kc)\notin\ZZ$.
  Note that the latter implies that $v(\kc)$ does not belong to the lattice, i.e.\ that $v(\kc)\in\VrtxnZ(\kP)$.
  \end{enumerate}
\end{remark}

%%%%%%%%%%%%%%%
%%%%%%%%%%%%%%%
\begin{example}
\label{ex-MinFourB}
Let $\kP = \conv(-\frac{1}{2},\frac{1}{2}) \subseteq \RR$.
Then $\kcP\subseteq \RR^2$ is spanned by the rays $\RR_{\geqslant0}\cdot (-1,2)$ and $\RR_{\geqslant0}\cdot (1,2)$.
The dual $\kcP\dual$ is spanned by $\RR_{\geqslant0}\cdot [-2,1]$ and $\RR_{\geqslant0}\cdot [2,1]$, see Figure~\ref{fig:6}
continuing the story of Figure~\ref{fig:3} and of Example~\ref{ex-discreteB}.
We obtain:
\[
\renewcommand{\arraystretch}{1.3}
\begin{array}{c|ccccccc}
\kc & \dots & -2 & -1 & 0 & 1 & 2& \dots\\
\hline
v(\kc) & \dots  & v_2 & v_2 & v_2\mbox{ or } v_1 & v_1 & v_1 & \dots\\
\keta(\kc)& \dots & 1 & \frac{1}{2} & 0 & \frac{1}{2} & 1 &\dots \\
\ketaZ(\kc)& \dots & 1 & 1 & 0 & 1 & 1 &\dots \\
\end{array}
\]

\definecolor{colsemigr}{rgb}{0.55, 0.71, 0.0}
\definecolor{colsubsgr}{rgb}{0.0, 0.5, 1.0}
\definecolor{colboundr}{rgb}{1.0, 0.01, 0.24}
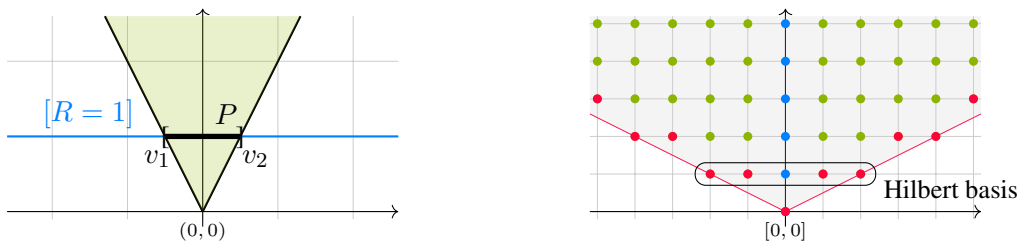
\begin{figure}[h!]
  \begin{tikzpicture}
    %parameters for size of picture  
    \newcommand{\xa}{5} %width
    \newcommand{\ya}{5} %height
    \newcommand{\ca}{1.7}
    %grid and axes
    \draw[very thin,color=white!80!black, step=1] (-\xa.2/2,-0.1) grid (\xa.2/2,\ya.2/2);
    \draw[->,name path=xaxis] (-\xa.2/2,0) -- (\xa.2/2 ,0) node[pos=0.5, below] {\tiny$(0,0)$};
    \draw[->,name path=yaxis] (0,-0.2) -- (0,\ya.4/2) node[pos=0.25,left ] {};
    % lines  
    \draw[thick, color=black, domain=0:\xa.2/2, mylabel=at 0.5 above with {}] plot (\x/2,\x);
    \draw[thick, color=black, domain=0:\xa.2/2, mylabel=at 0.2 above right with {}] plot (-\x/2,{\x}) ;%%};
    % corners of fill area
    \node (a) at (0,0) {};
    \node (b) at (-\xa.2/4,\xa.2/2) {};
    \node (c) at (\xa.2/4,\xa.2/2) {};
    % color area = semigroup   
    \fill[fill=colsemigr,fill opacity=0.2] (a.center)  -- (b.center) -- (c.center)  --cycle;
    % [R=1] hyperplane
    \draw[ thick, color=colsubsgr, domain=-\xa.2/2:\xa.2/2] plot (2*\x/2,1);
    \node[color=colsubsgr] (R) at (-\xa/3+0.15,1.3) {$[R=1]$};
    % Polyhedron
    \draw[line width=2pt, domain = -1/2:1/2, mylabel=at 0.5 above right with {$P$}] plot (\x,1);
    \node at (-.5,.98) {\small [ };
    \node at (.5,.98) {\small ] };
    \node at (-.6,.7) {$v_1$};
    \node at (.7,.7) {$v_2$};
  \end{tikzpicture}
  \qquad\qquad\qquad
    \begin{tikzpicture}
    %parameters for size of picture  
    \newcommand{\xa}{5} %width
    \newcommand{\ya}{5} %height
    \newcommand{\ca}{1.7}
    %grid and axes
    \draw[very thin,color=white!80!black, step=.5] (-\xa.2/2,-0.1) grid (\xa.2/2,\ya.2/2);    
    \draw[->,name path=xaxis] (-\xa.2/2,0) -- (\xa.2/2,0) node[pos=0.5, below] {\tiny $[0,0]$};
    \draw[->,name path=yaxis] (0,-0.2) -- (0,\ya.4/2) node[pos=0.25,left ] {};
    % lines  
    \draw[ultra thin, color=colboundr, domain=0:\xa.2/2, mylabel=at 0.5 above with {}] plot (\x,\x/2);
    \draw[ultra thin, color=colboundr, domain=0:\xa.2/2, mylabel=at 0.2 above right with {}] plot (-\x,{\x/2}) ;%%};
    % corners of fill area
    \node (a) at (0,0) {};
    \node (b) at (-\xa.2/2,\xa.2/4) {};
    \node (c) at (-\xa.2/2,\ya.2/2) {};
    \node (d) at (\xa.2/2,\ya.2/2) {};
    \node (e) at (\xa.2/2,\xa.2/4) {};
    % color area    
    \fill[fill=white!80!black,fill opacity=0.2] (a.center)  -- (b.center) -- (c.center)  -- (d.center) -- (e.center) --cycle;
    % color nodes
    % semigroup
    \foreach \y in {1,...,\ya}{
      \pgfmathsetmacro\mytemp{2*\y}
      \foreach \x in {1,...,\mytemp}{
        \pgfmathsetmacro\bound{\xa+1}
        \ifthenelse{\x<\bound}{
          \fill[thick, color=colsemigr] (\x/2,\y/2) circle (\ca pt);
          \fill[thick, color=colsemigr] (-\x/2,\y/2) circle (\ca pt);
        }
        {};
      }}
    % subgroup
    \foreach \y in {1,...,\ya}
    \fill[color=colsubsgr] (0,\y/2) circle (\ca pt);
    %boundary
    \foreach \x in {0,...,\xa}{      
      \pgfmathsetmacro\i{int((\x+1)/2)}
      \fill[thick, color=colboundr] (\x/2,\i/2) circle (\ca pt);
      \fill[thick, color=colboundr] (-\x/2,\i/2) circle (\ca pt);
    }
    % highlight
%    \node  at (0,.5) [draw, rounded rectangle] {$\qquad\quad\qquad\quad$};
    \draw[rounded corners] (-1.2,.35) rectangle (1.2,.65);
    \node at (2.2,.3) {\small Hilbert basis};
  \end{tikzpicture}
  \caption{The cone and the semigroups for the 1-dimensional 
$P=[-\frac{1}{2},\frac{1}{2}]\subset \RR$.}
  \label{fig:6}
\end{figure}
\end{example}

%%%%%%%%%%%%%%%%%%%%%%%%%%%%%%%%%%%%%%%%%%%%%%%%%%%%%%%%%%%%%%%%%%%%%%%%%%%
\subsection{The ambient space of the universal extension}
\label{sec:theAbmientSpace}
In this section we will define  $\kMT(\kP)$. The two subgroups which will later give the initial object live in the dual spaces of $\kMT(\kP)$ and $\kQuotRD\times\kMT(\kP)$, respectively. 
%%%%%%%%%%%%%%%%%%%
%% NOTATION:%%%%%%%
%%%%%%%%%%%%%%%%%%%
\begin{notation}
  \label{not:verticesEdges}
The set of compact edges of $\kP$ is
\(
\cedges(\kP) = \{\aed_1,\dots,\aed_\kPr\}.
\)
We  write
\(
\aed=[v^i,v^j]
\)
for the edge connecting $v^i$ and $v^j$ and 
we write
\(
[v^i,v^j)
\)
for the half-open edge. We will abuse notation and denote oriented edges also by $\vect{\aed}:=v^j-v^i\in\kQuotRD$; in the few ambiguous situations we will use words to specify which of the two $\aed$ refers to.
We also use the convention that round brackets denote vectors: $(1,2,3)\in \kQuotRD$,
square brackets denote linear forms $[4,5,6]\in\kQuotR$, and 
pointed brackets denote the standard perfect pairing: $\braket{(1,2,3),[4,5,6]} = 32\in\RR$.
We will also fix
\[
  \kPr=\card{\cedges(\kP)}\qquad\qquad \kPm=\card{\Vrtx(\kP)}.
\]

\end{notation}
The $\RR$-vector space $\kMT(\kP)$ will be a subspace of $\RR^{\kPr}\oplus\RR^{\kPm}$.
In order to describe its equations we need to introduce the following notions.

\begin{definition}
  \label{d:faceCycleOrientation}
  For every compact two-dimensional face $F$ of $\kP$ let
  $\orient_F:\cedges(\kP)\longrightarrow \{-1,0,1\}$ 
  be one of the two functions satisfying
  \[
  \begin{array}{rl}
    \orient_F(\aed)\in\{-1,1\} &\iff \aed \subseteq F\\
    \{\orient_F(\aed)\cdot\vect{\aed}\kst \aed \subseteq F\}&\text{forms an oriented cycle along~} \partial F.
  \end{array}
  \]
The last property implies (but is not equivalent to)
\[
\sum_{\aed\in\cedges(\kP)} \orient_F(\aed)\cdot\vect{\aed} = 0.
\]
\end{definition}

\begin{definition}
  \label{def-shortEdge}
  \begin{enumerate}[label={(\roman*)}]
  \item  To each half open edge $\aed=[v,w)$ of $\kP$
we associate the positive integer
\[
g=g(\aed) := \min\{g\in\ZZ_{\geqslant 1}\kst \text{the {\em affine line} through }gv \text{ and } gw \text{ contains lattice points}\}.
\]
\item We call $\aed=[v,w)$ a {\em\se\ half open edge} if
\[
\card{\{g\cdot [v,w)\cap\kQuotD\}}\leq g-1.
\]
If this is the case, then it follows that $v\notin \kQuotD$.
Finally, we call $\aed$ a \emph{\se\ edge} if 
both $[v,w)$ and $(w,v]$ are \se\ half open edges.
\end{enumerate}
\end{definition}

\begin{remark}
  \label{rem:afterDefShortedge}
If at least one of the half open edges $[v,w)$ or $[w,v)$ is \se,
then $\ell(w-v)<1$, where $\ell$ denotes the lattice length\footnote{ defined as the homogeneous function on $\kQuotRD$ such
that any primitive element of $\kQuotD$ has lattice length one.}.
While the property $\ell(w-v)<1$
is responsible for the name ``\se'',
$\ell<1$ alone does not suffice for \se ness. For example $\aed=[-\frac{1}{2},\frac{1}{3}]\subset\RR$
has lattice length $\ell=\frac{5}{6}<1$, but still,
neither of the two half open edges is \se.
\end{remark}

%%%%%%%%%%%%%%%
%%%%%%%%%%%%%%%
\begin{definition}
  \label{def:Clin}
For each compact edge $\aed_i=[v^j,v^k]$ of $\kP$ we introduce a 
parameter which we denote
by $t_{\aed_i}$, $t_i$,  or $t_{jk}$ \footnote{\label{denoteFootnote} always choosing the most convenient one in the given context.}.  We then collect all of these in a vector $\bft\in\RR^\kPr$ and  define the linear subspace
\[
\textstyle
\kMV(\kP) := \{\bft\in\RR^\kPr\kst\sum_{\aed\subset F}\orient_F(\aed)t_{\aed}\cdot\vect{\aed} = 0 \,\text{ for all compact 2-faces }F \text{ of }\kP\}.
\]
The intersection $\kMV(\kP)\cap R^\kPr_{\geqslant 0}$  parametrizes the Minkowski summands of positive multiples of $\kP$. 
\end{definition}
\begin{definition}
  \label{def:TP}
For each vertex $v=v^i\in\Vrtx(\kP)$ we introduce a parameter which we denote
  by $s_v$ or $s_i$. We then define
\[
\textstyle
\kMT(\kP):=\big\{ (\bft,\bfs) \in \kMV(\kP)\oplus \RR^\kPm \kst 
\begin{array}[t]{ll}
s_i = 0 &\mbox{if} \hspace{0.5em} v^i\in\kQuotD,\\
  s_i = s_j & \mbox{if}\hspace{0.5em} [v^i,v^j]\in\cedges(\kP) \text{ with } [v^i,v^j]\cap\kQuotD=\emptyset, \text{ and }\\
s_i=t_{ij} & \mbox{if}\hspace{0.5em} [v^i,v^j) 
\text{ is a half open \se\ edge} \big\}.
\end{array}
\]
\end{definition}
Note that the vector space $\kMT(\kP)$ contains a distinguished element
$\oneone=[\kP]$ which is defined by $s_i:=0$ for $v^i\in\kQuotD$ and
$s_j:=1$ and $t_{ij}:=1$ for all remaining coordinates, cf.\ Remark~\ref{rem-oneone}.
In the upcoming sections we will often deal with the dual vector space
$\kMTD(\kP)$. Then, its elements $s_i,t_{ij}\in\kMTD(\kP)$ form 
a generating set of this space. We could easily omit the elements 
$s_i=0$ for $v^i\in\kQuotD$. However, while they are just zero, 
there existence will simplify some formulae. 

%%%%%%%%%%%%%%%%%%%%%%%%%%%%%%%%%%%%%%%%%%%%%%%%%%%%%%%%%%%%%%%
\subsection{Relation to algebraic geometry}
In Example~\ref{ex-discreteB} we have already mentioned that
our theory of extensions of semigroups has strong links to 
deformation theory in algebraic geometry, cf.\ \cite{m} for addressing
this in detail. Nevertheless, we would like to mention here that
for a singularity $X$ there is the vector space $T_X^1$
of so-called infinitesimal deformations. In case of a toric singularity,
it is $\kQuot$-graded, and we denote by $T_X^1(-\kR)$ the contribution
in multidegree $-R$.

%%%%%%%%%%%%%%%
%%%%%%%%%%%%%%%
\begin{proposition}
\label{prop-compTOne}
For $X=\toric(\cone(\kP))$ we obtain that
$\,T_X^1(-\kR)=\big(\kMT(\kP)\otimes_{\RR}\CC\big)\big/\CC\cdot\oneone$.
\end{proposition}

\begin{proof}
Essentially, this corresponds to Theorem 2.5 in \cite{ka-flip}.
One has just to check that the equations called $G_{jk}$ in
\cite[(2.6)]{ka-flip} coincide with those
in the definition of the $\RR$-vector space $\kMT(\kP)$.
\end{proof}

\begin{example}
  \label{ex-CQS}
  \begin{enumerate}[label={\bf \arabic*.}, ref={\arabic*.}]
  \item\label{item:CQS1} The mother of all examples is $\kP=[-\frac{1}{2},\frac{1}{2}]\subset\RR$
    from Example~\ref{ex-discreteB}.
    Here we have only one edge $\aed=\kP$ with $g_{\kP}=1$. The interval has length
    one, and it contains exactly one lattice point, i.e.\ $\card{\{\kP\cap\kQuotD\}}=1$.
    In particular, it gives rise to two non-\se\ half open edges.
  \item\label{item:CQS2} The interval $\kP=[-\frac{1}{2},\frac{1}{3}]\subset\RR$
    has lattice length $\ell=\frac{5}{6}<1$, but still,
    neither of the two half open edges is \se.
  \item\label{item:CQS3} Since the interval $\kP=[\frac{1}{2},\frac{3}{4}]$ does not contain
    any lattice points, but the affine line through $\kP$ does,
    we obtain that $g=1$, and both half open edges are \se. 
  \item\label{item:CQS4} Take $\kP:=\conv\big\{(-\frac{1}{6},\frac{1}{2}),\;
    (\frac{2}{3},\frac{1}{2})\big\}\subset \RR^2$. 
    Here we need
    to multiply with $g=2$ to produce lattice points on the affine line.
    The resulting interval $g\kP=[-\frac{1}{3},\frac{4}{3}]$ has length
    $\frac{5}{3}<2$ and $\card{\{g\kP\cap\kQuotD\}}=2$. That is, 
    neither of the half open edges are \se. 
  \item\label{item:CQS5} At last, we consider $\kP:=\conv\big\{(-\frac{1}{2},\frac{1}{2}),\;
    (\frac{1}{3},\frac{1}{2})\big\}\subset \RR^2$.
    We still have $g=2$ and this leads to the interval $g\kP=[-1,\frac{2}{3}]$.
    In particular, one of the half open edges is \se, the opposite one is not.
  \end{enumerate}
\end{example}

%%%%%%%%%%%%%%%%%%%%%%%%%%%%%%%%%%%%%%%%%%%%%%%%%%%%%%%%%%%%%%%%%%%%%%%%%%%
\subsection{Understanding $\kMTPD(\kP)$ }
\label{ssec:understandMTPD}
By definition we have $\kMT(\kP)\subseteq \RR^\kPr\oplus\RR^\kPm$, which allows us to define
\[
  \kMTP(\kP):=\kMT(\kP)\cap(\RR^\kPr_{\geqslant 0}\oplus\RR^\kPm_{\geqslant 0}).
\]
For the construction of the universal object from Section~\ref{sec:mainPlayers} positivity does not play an important role.
However, positivity will be crucial for the correspondence with lattice-friendly Minkowski decompositions in Section~\ref{sec:MinkowskiDec}.
While positivity of he $t$-coordinates has a clear meaning, the necessity of positive $s$-coordinates becomes apparent in Example~\ref{ex-negativeS}.

We will denote the dual of $\RR^{\kPr+\kPm}$ by the same symbol\footnote{i.e.\ we will not add a star here.}: $\RR^{\kPr+\kPm}$.
Since $\kMT(\kP)\subseteq\RR^{\kPr+\kPm}$, we have a 
canonical projection $\RR^{\kPr+\kPm}\surj\kMTD(\kP)$ 
yielding elements $t_{ij},s_v\in\kMTD(\kP)$ and the equality
\begin{equation}
  \label{eq:TPdualAsQuotient}
  \kMTD(\kP) =\RR^{\kPr+\kPm}\big/\kMT(\kP)^\perp.  
\end{equation}
According to Definition~\ref{def:TP}, the subspace $\kMT(\kP)^\perp\subseteq\RR^{\kPr+\kPm}$ is generated by the following four types of elements:
\begin{eqnarray}
  \textstyle
  \label{eq:dualEquationsT(Q)1}  \chi(F):=\sum_{\aed\subset F}\orient_F(\aed)\,t_{\aed}\cdot\vect{\aed}&& \text{for all\footnotemark~ compact 2-faces $F$ of $\kP$,}\\
  \label{eq:dualEquationsT(Q)2}  \rule{0pt}{1.5em}  s_i&& \text{for all~}v^i\in\kQuotD.\\
  \label{eq:dualEquationsT(Q)3}  \rule{0pt}{1.5em}s_i - s_j &&\text{for all $[v^i,v^j]\in\cedges(\kP)$ with $[v^i,v^j]\cap\kQuotD = \emptyset$, and}\\
  \label{eq:dualEquationsT(Q)4}  \rule{0pt}{1.5em}t_{ij} - s_i &&\text{for all \se~ edges $[v^i,v^j)$.}
\end{eqnarray}\footnotetext{actually, these elements are in $\kMT(\kP)^\perp\otimes\kQuotRD$ and need to be evaluated by some $\kc\in\kQuotR$.}%
The relations $\chi(F)$ enable us to encode Minkowski summands 
via edge dilation, cf.\ Subsection~\ref{spaceMinkSummands} and
Definition~\ref{def:Clin}.
Their importance for the extensions of semigroups, however,
becomes apparent in the proof of 
Proposition~\ref{prop:equivalenceOfIndependence}.
From Equation~(\ref{eq:TPdualAsQuotient}) we get that 
the dual cone of $\kMTP(\kP)$ is
\[
\kMTPD(\kP) = 
\im(\RR_{\geqslant 0}^{\kPr+\kPm}
\longrightarrow \kMTD(\kP))= 
\frac{\RR_{\geqslant 0}^{\kPr+\kPm}+ 
\kMT(\kP)^\perp}{\kMT(\kP)^\perp}.
\]
\subsection{The lattice structure in $\kMT(\kP)$}
\label{ssec:theLatticeStructureInT(P)}
We start by defining  a subgroup of $\kMTZ(\kP)\subset\kMT(\kP)$, and then prove that this is a lattice. In the case in which $\kP$ is a lattice polytope with primitive edges, this lattice is simply $\Z^\kPr\cap \kMV(\kP)$, cf. Example~\ref{ex-backRoots}.
%%%%%%%%%%%%%%%
%%%%%%%%%%%%%%%
\begin{definition}\label{def:latticeInT}%\label{d:latticeInT} 
Define the subgroup $\kMTZ(\kP)\subset\kMT(\kP)$ by
\[
  \begin{array}{rcl}
    (\bft,\bfs)\in \kMTZ(\kP) &:\iff&
                                            \begin{cases}
                                              s_i\in \ZZ &,\forall~ v^i \in \Vrtx(\kP), \text{~and}\\
                                              (t_{ij}-s_i)v^i - (t_{ij}-s_j)v^j 
                                              \in \kQuotD &,\forall~ [v^i,v^j]\in \cedges(\kP).
                                            \end{cases}
  \end{array}
\]
\end{definition}
\noindent Clearly $\kMTZ(\kP)$ is a subgroup of $\kMT(\kP)$, thus it is Abelian and torsion-free.

\begin{lemma}\label{ilem-indepRefPointForLattice} 
  The subgroup $\kMTZ(\kP)$ is a free Abelian group
  satisfying
  \[
    \kMTZ(\kP)\otimes_{\ZZ}\RR=\kMT(\kP).
  \]
\end{lemma}

\begin{proof} 
We first show that $\kMTZ(\kP)$ is a discrete subgroup, i.e.\ that $0\in\kMTZ(\kP)$ is an isolated point.
If $(\bft,\bfs)\in\kMTZ(\kP)$ has sufficiently small coordinates, then the
integrality of $s_i$ implies that $s_i=0$. For the resulting
$(\bft,0)$ we thus get
\[
t_{ij}(v^i - v^j)\in\kQuotD.
\]
As we have finitely many compact edges, and as $\kQuotD$ is a lattice, thus not divisible, it follows that sufficiently small $t_{ij}$ are forced to be zero as well.

Every rational element of $\kMT(\kP)$  admits an integral multiple  contained in $\kMTZ(\kP)$.  Together with the discrete property, this implies that $\kMTZ(\kP)$ is an Abelian group satisfying 
  \[
    \kMTZ(\kP)\otimes_{\ZZ}\RR=\kMT(\kP).
  \] 
\end{proof}  

The dual lattice is by definition
\[
  \kMTZD(\kP)=\{f\in\kMTD(\kP)\kst f\big(\kMTZ(\kP)\big)\subseteq\ZZ\}.
\]
Thus, the dual lattice $\kMTZD(\kP)$ is generated by 
\[
  \{s_v ~:~ v\in\Vrtx(\kP)\}
  ~\cup~
  \{\kL_{ij}(\kc)~:~ \kc\in \tail(\kP)\dual\cap\kQuot\}.
\]

We will regard all the $t_{ij},s_i,s_j$ as coordinate functions, that is as elements of $\kMTD(\kP)$. 
The two conditions of Definition~\ref{def:latticeInT} can be thus rephrased as
\begin{eqnarray*}
  s_i&\in&\kMTZD(\kP),\\
\kL_{ij}:=(t_{ij}-s_i)\otimes v^i - (t_{ij}-s_j)\otimes v^j
&\in& \kMTZD(\kP)\otimes_{\ZZ}\kQuotD.
\end{eqnarray*}
\noindent We will often group the summands as:
\(
\kL_{ij} =   t_{ij}\otimes (v^i - v^j) + s_j \otimes v^j - s_i \otimes v^i. 
\)
\begin{remark}
  \label{rem:gensOfLatticeTensorLattice}
  The elements $\kL_{ij}$, together with $s_i\otimes\kQuotD$,  generate $\kMTZ(\kP)^*\otimes_{\ZZ}\kQuotD$.
\end{remark}
Note that $t_{ji}=t_{ij}$, but $\kL_{ji}=-\kL_{ij}$.
For any oriented compact edge $\aed=[v^i,v^j]$ we will write $\kL_\aed$ or $\kL_{v^iv^j}$ instead of $\kL_{ij}$ when it is more convenient to do so. 
Finally, let us point out that  the distinguished element $\oneone=:[\kP]$ belongs
to the lattice.
\begin{notation}
\label{not:LijOfcORofXi}
  As we are dealing with the tensor product of two linear forms, one on $\kMT(\kP)$ and one on $\kQuot$, it makes sense to apply $\kL_{ij}$ to both types of elements separately. We will denote as consistently as possible the elements of $\kMT(\kP)$ by $\xi$ and those of $\kQuot$ by $\kc$. Therefore, we will use the same  notation when we apply $\kL_{ij}$ to either of them:
  \begin{eqnarray*}
    \kL_{ij}(\xi)&:=& t_{ij}(\xi)\cdot(v^i - v^j) + s_j(\xi) \cdot v^j  -  s_i(\xi)\cdot v^i\in \kQuotRD, \text{~for~}\xi \in\kMT(\kP),   \\
    \kL_{ij}(\kc)&:=& \braket{v^i - v^j,\kc}\cdot t_{ij} +\braket{v^j,\kc}\cdot s_j  - \braket{v^i,\kc} \cdot s_i\in\kMTD(\kP), \text{~for~}\kc\in\kQuot.
  \end{eqnarray*}
\end{notation}

\begin{example}
  \label{ex-backRoots}
  \begin{enumerate}[label={\bf \arabic*.}, ref={\arabic*.}]
\item If $\kP=[v,w]\subset\RR$ is a rational line segment, 
then we denote by $d(\kP):=w-v$ the length of $\kP$ and by 
$\{v\}:= \roundup{v}-v$ and $\{w\}:=\roundup{w}-w$ the
discrepancies for $\kP$ to have integral limits.
Then, besides $s_v,s_w\in \kMTZD(\kP)$, this lattice is characterized
by the incidence
\[
d(\kP)\cdot t - \{v\}\cdot s_v + \{w\}\cdot s_w\in \kMTZD(\kP)
\vspace{-0.5ex}
\]
which has a straightforward geometric interpretation.
Indeed, this follows from
\vspace{-0.5ex}
\[
\begin{array}{rcl}
\kL_{d(\kP)}&=&-d(\kP)\cdot t + w\cdot s_w - v\cdot s_v
\\
&=& -d(\kP)\cdot t - \{w\} \cdot s_w + \{v\}\cdot s_v + 
\big(w+\{w\}\big)\cdot s_w - \big(v+\{v\}\big)\cdot s_v,
\end{array}
\vspace{-0.5ex}
\]
because the last two coefficients, $w+\{w\}$ and $-v-\{v\}$, are integers.

\item Whenever $\kP$ is a lattice polyhedron with primitive compact edges,
then $\kMT(\kP)=\kMV(\kP)$, and $\kMTZ(\kP)=\kMVZ(\kP)$ is determined by the
integrality of all $t_{ij}$.
\end{enumerate}
\end{example}

In Section~\ref{s:liftingsAndRelations} we will lift $\keta(\kc)$ and $\ketaZ(\kc)$ along the map
$\kpi:\kMTZD(\kP)\to\ZZ$ (or its rational version), where $t_{ij}\mapsto 1$ for all edges and $s_v\mapsto 1$ if $v\notin\kQuotD$.
This gives the vertical maps in the following diagram.
\[
\xymatrix@C=3em{
\kMTZD(\kP)\otimes_{\ZZ}\kQuotD \ar[d]_{\kpi:=\kpi\otimes\id_\kQuotD}
\ar[r]^-{\braket{\kbb,\kc}}         & 
\kMTZD(\kP) \ar[d]^{\kpi}
                                    & \kL_{ij} \ar@{|->}[r] \ar@{|->}[d] & \kL_{ij}(\kc) 
 \ar@{|->}[d]
 \\
\kQuotD \ar[r]^-{\braket{\kbb,\kc}} & \ZZ
                                    & (\delta^\ZZ v^i-\delta^\ZZ v^j) 
 \ar@{|->}[r]                    & \braket{\delta^\ZZ v^i-\delta^\ZZ v^j,c}
}
\]
In the previous diagram we have used the symbol
\[
\delta^\ZZ v:=\left\{\begin{array}{ll}
v & \mbox{if } \,v\in \kQuotD\\
0 & \mbox{if } \,v\in \kQuotRD\setminus\kQuotD.
\end{array}\right.
\]
In particular, $\kL_{ij}\in\ker(\kpi)$ if $v^i,v^j\in \VrtxnZ(\kP)$.

\section{Lifting the $\keta$ to $\kMTD$}
\label{s:liftingsAndRelations}
In this section we define liftings  $[\kc,\ketaZ(\kc)]$\reflectbox{$\mapsto$} $[\kc,\ktetaZ(\kc)]$
in the finite dimensional $\RR$-vector space $\kMTD(\kP)$ from Section~\ref{sec:theAbmientSpace}. 
The main idea behind constructing the universal extension is to use these liftings of the relative boundary
and define $\ktT$ using the relations among them: $\ktetaZ(\kc_1,\dots,\kc_\ell)$ (cf. Definition~\ref{def-independent}).
Then $\ktS$ will be the sum of the lifted boundary with $\ktT$.

\subsection{Lifting the boundary}
\label{subs:LiftingTheBoundary}
We start by  fixing a \emph{reference vertex} $ \vast\in\Vrtx(\kP)$.
\begin{convention}
\label{conv-refPtsB}
Whenever $\vast\in\kP$ belongs to the lattice, we will
set $\vast=0$. 
\end{convention}
As the $\kc$ in $[\kc,\keta(\kc)]$ will be unchanged when lifting, we will only focus on the $\keta$ part.
Let us state Definition~\ref{d:eta(c)} of $\keta(\kc)$ from a different point of view.
For $\kc\in \tail(\kP)\dual$ we choose a path 
$
\vast = v^0, v^1, \ldots ,v^k = v(\kc)
$
along the compact edges of $\kP$. 
% leading from $\vast$ to $v(\kc)$:
Then
\[
\begin{array}{rcl}
  -\keta(\kc) &=& \braket{ v(\kc),\,\kc}\\
      &=&\braket{ \vast,\,\kc}+\braket{ v(\kc)-\vast,\,\kc}
\\
      &=&\braket{\vast,\,\kc} + \sum_{j=1}^k \braket{ v^{j}-v^{j-1},\,\kc}.
\end{array}
\]
In complete analogy to this we define now the lifting
$\wteta(\kc)$ of $\keta(\kc)$.
%%%%%%%%%%%%%%%
%%%%%%%%%%%%%%%
\begin{definition}\label{d:etatilde(c)}
For every $\kc\in\tail(\kP)\dual$, we define $\wteta(\kc)\in \kMTD(\kP)$ as
\[
%\textstyle
\wteta(\kc) := -\braket{\vast,\,\kc}\cdot s_\vast - 
\sum_{j=1}^k  \braket{v^{j}-v^{j-1},\,\kc} \cdot t_{j-1,\,j}.
\]

\end{definition}

\begin{lemma}
\label{lem-etaTildeVc}
The definition of $\wteta(\kc)\in \kMTD(\kP)$ depends neither on the
choice of the vertex $v(\kc)$, nor on the choice of the path connecting
$\vast$ and $v(\kc)$.
\end{lemma}

\begin{proof}
The independence on the choice of the path follows by the usual argument,
namely by the presence of the closing relations, i.e.\ the elements
$\chi(F)$ providing
$\braket{\chi(F),\kc}\in \kMT(\kP)^\perp$ mentioned in 
Subsection~\ref{ssec:understandMTPD}.
\\[0.5ex]
For the independence on $v(\kc)$, let us choose another candidate
$v'(\kc)$. Then $\braket{ v'(\kc),\kc} = \braket{ v(\kc),\kc}$,
and we may connect both vertices by a path within the level face
$f_c=\{v\in\kP\kst \braket{v,\kc} =\braket{ v(\kc),\kc}\}$,
i.e.\ via edges $\aed\in\kc^\bot$. Thus, the two paths connecting $\vast$ with
$v(\kc)$ or $v'(\kc)$, respectively, can be chosen to differ only
by those edges. In particular, they produce the same result after being paired
with $\kc$.
\end{proof}

Note that $\wteta(\kc)$ is always a lifting of
$\keta(\kc)$ via the map $\pi$ due to our Convention~\ref{conv-refPtsB}, i.e. it holds that
\[
\kpi\big(\wteta(\kc)\big)= \keta(\kc).
\]

%%%%%%%%%%%%%%%
%%%%%%%%%%%%%%%
\begin{remark}
\label{rem-commentConvention}
Sending a certain point of $\kP$ to $0$ means to shift the 
polyhedron by some vector  $w$. This implies
$
\wteta(\kc)_{\kP-w} = \wteta(\kc)_{\kP} + \braket{ w,\kc}\cdot s_{\vast}.
$
In particular, if $\vast\in\kQuotD$, then shifting by $-\vast$ 
to meet our convention does
not change $\wteta(\kc)$.
\end{remark}

%%%%%%%%%%%%%%%%%%%%%%%%%%%%%%%%%%%%%%%%%%%%%%%%%%%%%%%%%%%%%%%%%%%%%%%%%%%
For integrality issues it is important to express $\wteta(\kc)$ in terms of the integral $\kL_{ij}\in \kMTZ(\kP)^*\otimes_{\ZZ}\kQuotD$.
%%%%%%%%%%%%%%%
%%%%%%%%%%%%%%%
\begin{lemma}
  \label{lem-etaTildeViaL}
  For every $\kc\in\tail(\kP)\dual\cap\kQuot$ and any path $\,\vast = v^0, v^1, \ldots ,v^k = v(\kc)$ along the compact edges of $\kP$ we have
  \[
    \wteta(\kc)=\keta(\kc)\cdot s_{v(\kc)} + \sum_{j=1}^k\kL_{j-1,j}(\kc).
  \]
\end{lemma}

%%%%%%%%%%%%%%%
%%%%%%%%%%%%%%%
\begin{proof}
We use the chosen path $\vast = v^0, v^1, \ldots ,v^k = v(\kc)$ to obtain
\[
\renewcommand{\arraystretch}{1.3}
\begin{array}{rcl}
-\wteta(\kc) &=& 
\braket{v^0,\,\kc}\cdot s_0 + 
\sum_{j=1}^k  \braket{v^{j}-v^{j-1},\,\kc} \cdot t_{j-1,\,j}
\\
&=&
\sum_{j=1}^k \Big(\braket{v^{j}-v^{j-1},\,\kc} \cdot t_{j-1,\,j}
- \braket{ v^j,\kc}\cdot s_j+\braket{ v^{j-1},\kc}\cdot s_{j-1}\Big)
+ \braket{v^k,\,\kc}\cdot s_k
\\
&=& \sum_{j=1}^k\kL_{j,j-1}(\kc) + \braket{v(\kc),\,\kc}\cdot s_{v(\kc)}. 
\end{array}
\vspace{-3.5ex}
\]
\end{proof}

%%%%%%%%%%%%%%%
%%%%%%%%%%%%%%%
\begin{corollary}\label{l:tildaLatticeEquivalence}
For every $\kc\in \tail(\kP)^\vee\cap\kQuot$ we have 
$\,\wteta(\kc)\in \kMTZD(\kP)$ if and only if $\,\keta(\kc)\in\ZZ$.  
\end{corollary}

%%%%%%%%%%%%%%%
%%%%%%%%%%%%%%%
\begin{proof}
Since $\kpi:\kMTZD(\kP)\to\ZZ$ maps $\wteta(\kc)$ to $\keta(\kc)$,
we obtain the first implication.
The converse  is a direct consequence of Lemma~\ref{lem-etaTildeViaL} and the integrality of $\kL_{ij}$..
\end{proof}

For $\kc\in \tail(\kP)^\vee\cap\kQuot$,
we recall $\ketaZ(\kc)=\roundup{\keta(\kc)}\in\ZZ$
from Definition~\ref{d:eta(c)}. Thus,
Lemma~\ref{lem-etaTildeViaL} suggests the following possibility to lift this
definition via $\kpi$.

%%%%%%%%%%%%%%%
%%%%%%%%%%%%%%%
\begin{definition}
  \label{def-etaTildeZ}
 For every $\kc \in \tail(\kP)^\vee\cap\kQuot$ we define $\ktetaZ(\kc)\in\kMTZD(\kP)$ as
 \begin{eqnarray*}
\ktetaZ(\kc) &:=& \ketaZ(\kc)\cdot s_{v(\kc)} + {\textstyle \sum_{j=1}^k}\kL_{j-1,j}(\kc) \\
             &=&\wteta(\kc) + \big(\ketaZ(\kc)-\keta(\kc)\big)\cdot s_{v(\kc)}\\
             &=&\wteta(\kc) + \etaFrac{\keta(\kc)}\cdot s_{v(\kc)}.
 \end{eqnarray*}
\end{definition}
\begin{remark}
  \label{rem-positiveEtaCC}
  \begin{enumerate}[label=(\roman*)]
  \item By Convention~\ref{conv-refPtsB} we have $\kpi\big(\kteta(\kc))=\keta(\kc)$. We also have 
    \[
      \kpi\big(\ktetaZ(\kc)\big)=\ketaZ(\kc).
    \]
\item The pair $[\kc,\ktetaZ(\kc)]$ is a quite natural lifting of 
$[\kc,\ketaZ(\kc)]$ from $\kQuot\times \NN$ to $\kQuot\times \kMTD(\kP)$. 
However, even when asking for some positivity,
it might be not the only lifting -- see \cite[3.7]{AK13} for an example.
  \end{enumerate}
\end{remark}
\begin{lemma}
  \label{lem:ktetaZindepOfv(c)}
  The definition of $\ktetaZ(\kc)$ does not depend on the choice of the vertex $v(\kc)$.
\end{lemma}
\begin{proof}
As we did in the proof of Lemma~\ref{lem-etaTildeVc},
we may connect $v(\kc)$ and $v'(\kc)$ by edges $\aed=[v^i,v^j]$ contained in the face
$f_c:=\{p\in\kP\kst \braket{ p,\kc} =\min\braket{ \kP,\kc}\}$.
If the shared coefficient of
our two heroes $s_{v(\kc)}$ and $s_{v'(\kc)}$ matters at all, that is, if $\ketaZ(\kc)-\keta(\kc)\neq 0$, then $\keta(\kc)\notin\ZZ$, i.e.\
the face $f_c$ contains no lattice points.
That is, any edge $\aed=[v^i,v^j]$ on the path between $v(\kc)$ and
$v'(\kc)$ satisfies the property $[v^i,v^j]\cap\kQuotD=\emptyset$.
This property occurs in Definition~\ref{def:TP}, and implies $s_i=s_j$ as elements of $\kMTZD(\kP)\subseteq\kMTD(\kP)$.
Altogether, it means that $s_{v(\kc)}=s_{v'(\kc)}$.
\end{proof}
\subsection{Relations}
Having defined $\keta(\kc)$ as a minimum and $\ketaZ(\kc)=\roundup{\keta(\kc)}$ (Definition~\ref{d:eta(c)}), we get
\(
\roundup{\keta(\kc_1)} +\dots+ \roundup{\keta(\kc_\ell)} \geq
\roundup{\keta(\kc_1)+\dots+\keta(\kc_\ell)}
\geq \roundup{\keta(\kc_1+\dots+\kc_\ell)}
\)
for any sequence $\kc_1,\dots,\kc_\ell$ of not necessarily distinct elements of $\tail(\kP)^\vee\cap\kQuot$.
This implies:
\begin{equation}
  \label{eq:inequalityOverSumsKeta}
  \textstyle
  \ketaZ(\kc_1)+\ldots+\ketaZ(\kc_\ell)\geq\ketaZ\big(\kc_1+\ldots+\kc_\ell\big).
\end{equation}

\begin{definition}
  \label{def-independent}
Let $\ell\geqslant 2$.  For each sequence $\kc_1,\dots,\kc_\ell\in\tail(\kP)^\vee\cap\kQuot$  of not necessarily distinct elements, and for each of the symbols $\keta,\ketaZ,\kteta,$ or $\ktetaZ$, which we represent bellow by a $\geta$,  we define
   \[ 
     \geta(\kc_1,\ldots,\kc_\ell):= {\textstyle \sum_{i=1}^\ell \geta(\kc_i) - \geta\big(\sum_{i=1}^\ell\kc_i\big)}
   \]
   A sequence  $\kc_1,\ldots,\kc_\ell$ is called \emph{$\geta$-independent} if $\geta(\kc_1,\ldots,\kc_\ell)=0$. 
   We use the convention that every sequence of length one is independent as well\footnote{This
     make sense, as the definition of $\geta(\kc_1,\dots,\kc_\ell)$ would give zero for $\ell=1$.
     However, due to the overlap in notation for $\ell=1$, we  define $\geta(\kc_1,\dots,\kc_\ell)$ only for $\ell\ge 2$.}. 
 \end{definition}
 This definition does not depend on the order of the $\kc_i$, just on the multiset. 

\begin{remark}
\label{rem-positiveEtaCC}
\begin{enumerate}[label=(\roman*)]
\item Convention~\ref{conv-refPtsB} and Remark~\ref{rem-positiveEtaCC} extend to:
  \begin{eqnarray*}
    \kpi\big(\kteta(\kc_1,\dots,\kc_\ell)\big)&=&\keta(\kc_1,\dots,\kc_\ell)\\
    \kpi\big(\ktetaZ(\kc_1,\dots,\kc_\ell)\big)&=&\ketaZ(\kc_1,\dots,\kc_\ell).
  \end{eqnarray*}
\item The fact that $\keta(\kc_1,\dots,\kc_\ell)\geq 0$ is a trivial consequence of $\keta(\kc)$ being defined as some minimum.
  However, for the integral variant $\ketaZ(\kc_1,\dots,\kc_\ell)$ one should keep in mind that this does \emph{not} need to be the roundup of $\keta(\kc_1,\dots,\kc_\ell)$;
even the inequality $\ketaZ(\kc_1,\kc_2)\geq\keta(\kc_1,\kc_2)$ might
fail. Nevertheless, the non-negativity of $\ketaZ(\kc_1,\dots,\kc_\ell)$ is given by Equation (\ref{eq:inequalityOverSumsKeta}), so $\ketaZ(\kc_1,\ldots,\kc_\ell)\in \NN$.
\item \label{item:inductiveStepKteta}For every $\geta\in\{\keta,\ketaZ,\kteta,\ktetaZ\}$, for every $\kc_1,\dots,\kc_\ell\in\tail(\kP)\dual\cap\kQuot$ with $\ell\geqslant 2$, and for every $i=2,\dots,\ell-1$ we have
  \begin{eqnarray*}
    \geta(\kc_1,\dots,\kc_\ell) &=& \phantom{+~}\geta(\kc_1) +\dots+ \geta(\kc_i) - \geta(\kc_1+\dots+\kc_i)+\\
                                  &&  +~ \geta(\kc_1+\dots+\kc_{i}) +\geta(\kc_{i+1})+\dots+\geta(\kc_\ell)- \geta(\kc_1+\dots+\kc_\ell)\\
    &=& \geta(\kc_1,\dots,\kc_i) + \geta(\kc_1+\dots+\kc_{i},\kc_{i+1},\dots,\kc_\ell).
  \end{eqnarray*}
  \item \label{item:sameSemigroupWith2orMany}
    In particular, the above recursive formula gives us the semigroup equality
    \[
      \spann{\NN}{\ktetaZ(\kc_1,\dots,\kc_\ell)~:~\forall~\ell\geqslant2} = \spann{\NN}{\ktetaZ(\kc_1,\kc_2)},
    \]
    where the $\kc_i$ run through all possible elements of $\tail(\kP)^\vee\cap\kQuot$.
\end{enumerate}
\end{remark}

%%%%%%%%%%%%%%%
%%%%%%%%%%%%%%%
\begin{example}
\label{ex-MinFourC}
Let us continue Example~\ref{ex-MinFourB}.
Denoting the variables associated to the two non-lattice
vertices $-\frac{1}{2}$ and $\frac{1}{2}$ by $s_1$ and $s_2$, respectively,
and denoting by $t$ the variable referring to the one and only edge
$\aed=\kP$, we obtain the following values:
\[
\renewcommand{\arraystretch}{1.3}
\begin{array}{c|cc|cc}
\kc & \keta(\kc) & \ketaZ(\kc) & \kteta(\kc) & \ktetaZ(\kc)
\\
\hline
-2 & 1 & 1 & -s_1+2t & -s_1+2t
\\
-1 & \frac{1}{2} & 1 & -\frac{1}{2} s_1+t & \frac{1}{2} (s_2-s_1) +t
\\
\hline
0 & 0 & 0 & 0 & 0 
\\
\hline
1 & \frac{1}{2} & 1 & \frac{1}{2} s_1 & s_1
\\
2 & 1 & 1 & s_1 & s_1
\end{array}
\]
Turning to the values for $\ktetaZ(\kc_1,\kc_2)$, this leads to
  \[
  \renewcommand{\arraystretch}{1.5}%
  \begin{array}{rclcrcl}
    \ktetaZ(\phantom{-}1,1)&=&s_1&\qquad\qquad&
    \ktetaZ(-1,-1)&=&s_2\\
  \ktetaZ(-1,1)&=&t+\frac{1}{2} (s_1+s_2)&&
  \ktetaZ(-2,\phantom{-}2)&=&2t \\
  \ktetaZ(-1,2)&=&\frac{1}{2} (s_2-s_1) +t &&
  \ktetaZ(-2,\phantom{-}1)&=&\frac{1}{2} (s_1-s_2) +t                     
  \end{array}%
\]
\end{example}

Our main goal in this section is to prove that the notions of $\ketaZ$-independence and $\ktetaZ$-independence from Definition~\ref{def-independent} are equivalent (Proposition~\ref{prop:equivalenceOfIndependence}).

%%%%%%%%%%%%%%%
%%%%%%%%%%%%%%%
\begin{lemma}
\label{lem-independent}
The property of $\ketaZ$-independence is bequeathed to subsequences and to partitioning.
Moreover, the latter is also true for the property of being ``minimally $\ketaZ$-dependent''.
\end{lemma}
\begin{proof}
  The first part follows from the equation (\ref{eq:inequalityOverSumsKeta}) and  Remark~\ref{rem-positiveEtaCC}~\ref{item:inductiveStepKteta}.
  If  $\ell\geqslant3$ and $\kc_1,\ldots,\kc_{\ell}$ is minimally dependent, then $\ketaZ(\kc_1,\dots,\kc_\ell)>0$ and $\ktetaZ(\kc_1,\kc_2) =0$, 
so $(\kc_1+\kc_2),\kc_3,\ldots,\kc_{\ell}$ is dependent too. The minimality of this property is clear.
\end{proof}
%%%%%%%%%%%%%%%%%%%%%%%%%%%%%%%%%
The next lemma will not be directly applied later on. Its proof however can be seen as a warm up in which some notation is fixed for the proof the main result in this section. 
\begin{lemma}
  \label{lem:etaIndependence}
  Let $\kc_1,\kc_2 \in \tail(\kP)^\vee\cap\kQuot$. If  $\keta(\kc_1,\kc_2)=0 $, then $\kteta(\kc_1,\kc_2)=0$.
\end{lemma}

\begin{proof}
  If $\vect{\aed}^1_1,\ldots,\vect{\aed}^1_k$ and $\vect{\aed}^2_1,\ldots,\vect{\aed}^2_l$ are the oriented edges forming paths 
leading from $v(\kc_1+\kc_2)$ to $v(\kc_1)$ and $v(\kc_2)$   with 
  $\braket{\vect{\aed}^1_i,\kc_1}, \braket{\vect{\aed}^2_j,\kc_2}\leq 0$, respectively,
we obtain
  \begin{equation}
    \label{eq:ktetac1c2AsSumOftij}
    \textstyle
    \kteta(\kc_1,\kc_2)=
    -\sum_{i=1}^k\braket{ \vect{\aed}^1_i,\kc_1} \cdot t_{\aed^1_i} 
    -\sum_{j=1}^l\braket{ \vect{\aed}^2_j,\kc_2} \cdot t_{\aed^2_j}     
  \end{equation}
showing non-negative coefficients.
This implies that the edges $\vect{\aed}^1_i$ and $\vect{\aed}^2_j$ have to be contained in
$\kc_1^\bot$ and $\kc_2^\bot$, respectively, because
  \[
    \textstyle
    \kpi\big(\kteta(\kc_1,\kc_2)\big) =
    -\sum_{i=1}^k\braket{ \vect{\aed}^1_i,\kc_1}  
    -\sum_{j=1}^l\braket{ \vect{\aed}^2_j,\kc_2} 
    = \keta(\kc_1,\kc_2)=0.
    \]
Thus, $\braket{ v(\kc_1),\kc_1}=\braket{ v(\kc_1+\kc_2),\kc_1}$
and $\braket{ v(\kc_2),\kc_2}=\braket{ v(\kc_1+\kc_2),\kc_2}$.
This means that we could have chosen, i.e.\ that we can assume now, that
\(
v(\kc_1)=v(\kc_1+\kc_2)=v(\kc_2),
\)
which implies  $\kteta(\kc_1,\kc_2)=0$.
\end{proof}
%%%%%%%%%%%%%%%%%%%%%%%%%%%%%%%%%
%%%%%%%% key lemma %%%%%%%%%%%%%%
%%%%%%%%%%%%%%%%%%%%%%%%%%%%%%%%%
In the next proof, the generators of $\kMT(\kP)^\perp$ described 
in (\ref{eq:dualEquationsT(Q)1}-\ref{eq:dualEquationsT(Q)4}) play a 
crucial role. 

%%%%%%%%%%%%%%%%%%%%%%
\begin{proposition}
  \label{prop:equivalenceOfIndependence}
 Being $\ketaZ$-independent is equivalent to being $\ktetaZ$-independent. 
\end{proposition}
\begin{proof}
  By Remark~\ref{rem-positiveEtaCC} $\ktetaZ$-independence implies $\ketaZ$-independence. For the other direction we use induction on the length $\ell$ of the sequence $\kc_1,\dots,\kc_\ell$. The case $\ell=1$ is trivial. The essential step is~$\ell=2$. 
  
So let $\kc_1,\kc_2\in  \tail(\kP)^\vee\cap\kQuot$ with $\ketaZ(\kc_1,\kc_2)=0$.
Note first that we have
\[
  \etaFrac{\keta(\kc)} = \roundup{\keta(\kc)} - \keta(\kc)
                       = \roundup{-\braket{v(\kc),c}} - (-\braket{v(\kc),\kc})
                       = \braket{v(\kc),\kc)} -\rounddown{\braket{ v(\kc),\kc}}.
\]
Combining the above relation with the definition of $\kteta$  and with the formula (\ref{eq:ktetac1c2AsSumOftij}) for $\kteta(\kc_1,\kc_2)$   we obtain:
\begin{eqnarray}
  \ktetaZ(\kc_1,\kc_2)&=& \kteta(\kc_1,\kc_2) + \etaFrac{\keta(\kc_1)}\cdot s_{v(\kc_1)} + \etaFrac{\keta(\kc_2)}\cdot s_{v(\kc_2)} - \etaFrac{\keta(\kc_1+\kc_2)}\cdot s_{v(\kc_1+\kc_2)}\notag\\
                                      &=&\textstyle                                          
                                          -\sum_{i=1}^k\braket{ \vect{\aed}^1_i,\kc_1}\cdot t_{\aed^1_i}  + \big(\braket{ v(\kc_1),\kc_1} - \rounddown{\braket{ v(\kc_1),\,\kc_1}} \big)\cdot s_{v(\kc_1)} \label{eq:EI1}\\
                                      &&\textstyle
                                         -\sum_{j=1}^l\braket{ \vect{\aed}^2_j,\kc_2}\cdot t_{\aed^2_j}  + \big(\braket{ v(\kc_2),\kc_2} - \rounddown{\braket{ v(\kc_2),\,\kc_2}} \big)\cdot s_{v(\kc_2)}\label{eq:EI2}\\
                                      && \phantom{\textstyle-\sum_{i=1}^k\braket{ \vect{\aed}^1_i,\kc_1}\cdot t_{\aed^1_i}}-  \big(\braket{ v(\kc_1+\kc_2),\kc_1+\kc_2} - \rounddown{\braket{ v(\kc_1+\kc_2),\,\kc_1+\kc_2}} \big)\cdot s_{v(\kc_1+\kc_2)}\notag
\end{eqnarray}
where the $\aed^1_i$ and $\aed^2_j$ are just as in the proof of
Lemma~\ref{lem:etaIndependence}.
Our goal is to show that, assuming $\eta_\ZZ$-independence, all the edges above are short. The proof is analogous for both paths, so we focus only on $\vect{\aed}^1_1,\ldots,\vect{\aed}^1_k$, and label the vertices  with $v(\kc_1+\kc_2)=v^{0},\ldots,v^{k}=v(\kc_1)$.
From  $\braket{\vect{\aed}^1_i,\kc_1}=\braket{v^i-v^{i-1},\kc_1}\leq 0$ we get that $\braket{ v^{i},\kc_1} \leq \braket{ v^{i-1},\kc_1}$.
Via a suitable choice of $v(\kc_1)$, we can even insist on strict inequalities:
\[
  \braket{ v^{i},\kc_1} <\braket{ v^{i-1},\kc_1}.
\]
So the vanishing of $\ketaZ(\kc_1,\kc_2)$, which is obtained from $\ktetaZ(\kc_1,\kc_2)$ by sending $t_{\kbb,\kbb}, s_\kbb\mapsto 1$,
means that all the coefficients in rows (\ref{eq:EI1})  and (\ref{eq:EI2}) added up
cancel with the single negative coefficient: that of $s_{v(\kc_1+\kc_2)}$, which is contained in the half-open real interval $[0,1)$.
In particular, the sum of all coefficients of $t_{\kbb,\kbb}$ and $s_\kbb$ below is positive and strictly less than 1:
\[
  \textstyle
S_1:=-\sum_{i=1}^k\braket{v^i-v^{i-1},\,\kc_1} \cdot t_{i-1,\,i}+ \big(\braket{ v(\kc_1),\kc_1} - \rounddown{\braket{v(\kc_1),\,\kc_1}} \big)\cdot s_{v(\kc_1)}.
\]
We would like to express the coefficient of $s_{v(\kc_1)}$ in a similar way as the coefficients of $t_{i-1,\,i}$.
For this, we choose a \emph{$\kc_1$-integral point}
\footnote{meaning that $\braket{v^{k+1},\kc_1}\in\ZZ$}
$\,v^{k+1}\in\kQuotRD$ such that 
\[
\braket{ v^{k+1},\kc_1} = \rounddown{\braket{ v(\kc_1),\,\kc_1}}
\in\ZZ.
\]
Note that integral points $v\in\kQuotD$ are always $\kc_1$-integral, but
the opposite is far from being true.
Moreover, note that,
unless $\keta(\kc_1)\in\ZZ$, the new point $v^{k+1}$ cannot be
contained in the polyhedron $\kP$.
On the other hand, if $\keta(\kc_1)\in\ZZ$, then we may and will choose 
$v^{k+1}:=v^k$.
Anyway, the $s_{v(\kc_1)}$- coefficient of $S_1$ becomes 
$\braket{v^k-v^{k+1},\,\kc_1}$, and we obtain
\[
\textstyle
\kpi(S_1)=-\sum_{i=1}^k\braket{v^i-v^{i-1},\,\kc_1}+\lan v^k-v^{k+1},c_1\ran=\braket{v^0,\,\kc_1} - \braket{v^{k+1},\,\kc_1}<1.
\]
We want to deduce that
\begin{equation}
  \label{eq:longChain}
s_{v(\kc_1+\kc_2)}=t_{0,1}=s_{v^1}=t_{1,2}=\ldots=s_{v^{k-1}}=t_{k-1,k}=s_{v(\kc_1)}
\hspace{1em}\mbox{in}\hspace{0.4em}
\kMTD(\kP)
\end{equation}
where the last $s_{v(\kc_1)}=s_{v^{k}}$ has to be omitted if 
$\keta(\kc_1)\in\ZZ$. These equalities (together with those for 
the analogous $\kc_2$-summand) obviously imply that
\[
\ktetaZ(\kc_1,\kc_2)=\ketaZ(\kc_1,\kc_2)\cdot s_{v(\kc_1+\kc_2)}=0.
\]
To obtain (\ref{eq:longChain}) we show that
for $i=1,\dots,k$ the edges $[v^{i-1},v^i]$  are \se\ edges,
cf.\ Definition~\ref{def-shortEdge}.
Actually, if $\keta(\kc_1)\in\ZZ$, then only 
the one half open $[v^{k-1},v^k)$ is needed to be \se\ for the last segment.
Assume that this fails for one of them.
Then we have
\[
  \card{[g\hspace{0.05em}v^{i-1},g\hspace{0.05em}v^i)\cap\kQuotD}\geq g \qquad\text{or}\qquad
  \card{(g\hspace{0.05em}v^{i-1},g\hspace{0.05em}v^i]\cap\kQuotD}\geq g,
\]
where  $g\in\NN_{\geq 1}$ denotes the smallest number such that the line 
connecting $gv^{i-1}$ and $gv^i$ contains lattice points.
In the first case, this implies that 
there are at least $(g+1)$ $\kc_1$-integral points along our path from
$gv^{i-1}$ to $gv^{k+1}$, hence, more than ever, from
$gv^{0}$ to $gv^{k+1}$. In the second case, we obtain the same, unless
$i=k$ and $v^k=v^{k+1}$ -- however this exactly means that we speak about
the half open interval $(v^{k-1},v^k]$ in the situation where
$\keta(\kc_1)\in\ZZ$, which was excluded before.
\\[1ex]
Anyway, we do always get
\[
\card{[g\hspace{0.05em}v^{0},g\hspace{0.05em}v^{k+1}]\cap\kQuotD}\geq g+1, 
\]
hence, we obtain
\[
g\cdot\kpi(S_1) = \braket{ gv^0,\kc_1} - \braket{ gv^{k+1},\kc_1}\geq g,
\]
a contradiction. 
This concludes the proof for $\ell=2$. The inductive step follows from Lemma~\ref{lem-independent} and  Remark~\ref{rem-positiveEtaCC}~\ref{item:inductiveStepKteta}.
\end{proof}

%%%%%%%%%%%%%%%%%%%%%%%%%%%%%%%%%%%%%%%%%%%%
\subsection{Liftings and relations of the Hilbert basis}
\label{ssec:relationsAndHB}
In the last part of this section we will prepare the proof of the finite generation of the semigroup of all relations $\ktetaZ(\kc_1,\dots,\kc_\ell)$ given in Proposition~\ref{prop-finGenT}.
To this aim, let
\(
  \big\{[\ahb_1,\ketaZ(\ahb_1)],\dots,[\ahb_k,\ketaZ(\ahb_k)],[{0},1]\big\}
\)
be the Hilbert basis of $\kcP\dual\cap(\kQuot\oplus\ZZ)$ (cf. Subsection~\ref{sigmaDual}, Remark~\ref{rem:firstPropertiesOfeta}).
Multisets supported on $\{\ahb_1,\ldots,\ahb_k\}$ correspond to elements of $\NN^k$ via the  multiplicities of occurrence of each element: $m_1,\ldots,m_k\in\NN$. We denote them by  $\{\ahb_1^{m_1},\dots,\ahb_k^{m_k}\}$. So we may speak of $\ketaZ$-dependent elements $\bfm=(m_1,\dots,m_k)\in\NN^k$ via this correspondence, and write % for  $\geta\in\{\keta,\ketaZ,\kteta,\ktetaZ\}$
\[
  \ketaZ(\bfm):= \ketaZ(\ahb_1^{m_1},\dots,\ahb_k^{m_k}),~~\forall~\bfm\in\NN.
\]

\begin{lemma}
\label{lem-finMinDep}
The number of minimally $\ketaZ$-dependent elements of $\NN^k$ finite.
\end{lemma}
\begin{proof} 
By Lemma~\ref{lem-independent},
the set of dependent sequences supported on 
$\ahb_1,\dots,\ahb_k$  is in order-preserving correspondence with a 
subset of $\NN^k$ representing a monomial ideal
$\,\dep_\ketaZ\subseteq\ZZ[t_1,\ldots,t_k]$.
So the above statement follows from Dickson's Lemma.
\end{proof}

%%%%%%%%%%%%%%%%%%%%%%%%%%%%%%%%%%%%%%%%%%%%%%%%%%%%%%%%%%%%%%%%%%%%%%%%%%%%%
\section{The universal \cocartesian~ extension}
\label{sec:mainPlayers}
The inspiration for the following definition comes from Proposition~\ref{prop:characterizeCoCartesian}, which defines \cocartesian~\aextension s by isomorphic relative boundaries,
and Remark~\ref{rem:firstPropertiesOfeta} which describes the relative boundary of our given object \givenpair, where $\cone_\ZZ(\kP)\dual:=\cone(\kP)\dual\cap(M\oplus \ZZ)$. Denote by
\[
  \kMTtZD(\kP):= \kQuot \oplus \kMTZD(\kP).
\]
%%%%%%%%%%%%%%%
%%%%%%%%%%%%%%%
\begin{definition}\label{def:upperSemigroups}
For any rational polyhedron $\kP\subseteq \kQuotD_\RR$, define the semigroups  $\ktT,\ktS \subset \kMTtZD(\kP)$ as
\begin{eqnarray*}
 \ktT &=& \spann{\NN}{[0,\ktetaZ(\kc_1,\kc_2)]\kst  \kc_1,\kc_2\in \tail(\kP)^\vee\cap \kQuot},\\
 \ktS &=& \ktT + \spann{ \NN }{[\kc,\ktetaZ(\kc)]\kst \kc\in \tail(\kP)^\vee\cap \kQuot }.
\end{eqnarray*}
\end{definition}
\noindent 
By Remark~\ref{rem-positiveEtaCC}~\ref{item:sameSemigroupWith2orMany} we have
\[
  \ktT = \spann{\NN}{[0,\ktetaZ(\kc_1,\dots,\kc_\ell)]\kst \forall~\ell\geq 2 \text{~and~}\forall~ \kc_1,\dots,\kc_\ell\in \tail(\kP)^\vee\cap \kQuot}.
\]
Further on, in Proposition~\ref{prop-finGenT}, we will see that we could 
have chosen, in the latter version of defining $\ktT$,
only those $\kc$ which appear in a Hilbert Basis of $\cone_\ZZ(\kP)\dual$.

\subsection{Belonging to the category}
\label{sec:UOisCocartExtension}
In this section we will check that the semigroups $\ktT\subset\ktS$ form a \cocartesian~\aextension\ of \givenpair.
\begin{proposition}\label{prop boundary pr}
It holds that
\[
\Rand{\ktT}{\ktS} = \{[\kc,\ktetaZ(\kc)]~|~ \kc\in\tail(\kP)\dual\cap\kQuot\}.
\]
\end{proposition}
\begin{proof}
The main consequence of Proposition~\ref{prop:equivalenceOfIndependence} is that $\ker \kpi_\kT=0$, from which it follows that $[\kc,\ktetaZ(c)]\in \Rand{\ktT}{\ktS}$. The other inclusion is obvious.
\end{proof}
%%%%%%%%%%%%%%%
%%%%%%%%%%%%%%%

\begin{remark}
\label{rem-noChoice}
Since $[\kc,\ktetaZ(c)]$ are natural
(but not the only) liftings of $[\kc,\ketaZ(\kc)]\in\Rand{\NN}{\cone_\ZZ(\kP)\dual}$,
it is quite natural to put these elements in $\ktS$.
Independently of the shape of $\ktT$, the required triviality of the kernel of $\kpi|_{\ktT}:\ktT\to\NN$
(cf.  Definition~\ref{def-extension})
implies that $[\kc,\ktetaZ(c)]\in \Rand{\ktT}{\ktS}$. 
By the defining property of the relative boundary, it follows that
\[
[0,\,\ktetaZ(\kc_1,\kc_2)] = 
[\kc_1,\,\ktetaZ(\kc_1)] + [\kc_2,\,\ktetaZ(\kc_2)] - 
[\kc_1+\kc_2,\,\ktetaZ(\kc_1+\kc_2)]
\]
has to be contained in $\ktT$. Thus,  the Definition~\ref{def:upperSemigroups} was quite inevitable. At least, it was the minimal choice.
\end{remark}

%%%%%%%%%%%%%%%
%%%%%%%%%%%%%%%
\begin{comment}
\begin{remark}
  \label{rem:etatildaLiftsTheBoundary}
  Every element in $\kts\in\ktS$ can be written as
  \[
    \kts = [\kc,\ktetaZ(\kc)] +\ktt,
  \]
  with $\kc\in\tail(\kp)\dual\cap\kQuot$ and $\ktt\in \ktT$. This is because every element in the span of the $[\kc,\ktetaZ(\kc)]$ is given by some sequence $\kc_1,\dots,\kc_\ell$, and it is then enough to set $\kc=\kc_1+\dots+\kc_\ell$, and add $\ktetaZ(\kc_1,\dots,\kc_\ell)$ to the part from $\ktT$. In particular, we have that
\[
\Rand{\ktT}{\ktS} = \{[\kc,\ktetaZ(\kc)]\kts  \kc\in\tail(\kP)\dual\cap\kQuot\}
\]
\matej{without $\kts$}
  and that the addition map $\ktsum:\Rand{\ktT}{\ktS}\times\ktT \too \ktS$ is surjective. 
\end{remark}
\end{comment}

Proposition~\ref{prop:equivalenceOfIndependence} is also crucial to prove the following.
\begin{proposition}
  \label{prop-coCart}
For every rational polyhedron $\kP$, the diagram  
  \[
 \xymatrix{
   \ktT~ \ar@{^{(}->}[r] \ar@{->}[d]_{\kpi_{\kT}}
   &
   \ktS \ar@{->}[d]^{\kpi_{\kS}}\\
   \givenpairXYmat
 }
\]
with vertical maps induced by $t_{\kbb,\kbb},s_v\mapsto 1$ for $v\not\in N$ and $s_v\mapsto 0$ for $v\in N$, is a \cocartesian~\aextension. This means  that 
the addition maps are surjective,  $\kpi_\kS$ induces a bijection on the boundaries, and $\ker\kpi_\kT=\ker\kpi_\kS=0$.
\end{proposition}

\begin{proof}

  The addition map downstairs is surjective because the pair is \straight.
  The addition map upstairs is by Proposition~\ref{prop boundary pr} surjective.
  The restriction of $\kpi_\kS$ to the boundary maps $[\kc,\ktetaZ(\kc)]\longmapsto[\kc,\ketaZ(\kc)]$, which is obviously bijective. 
  
  We have that $\kpi_\kT([0,\ktetaZ(\kc_1,\kc_2)])=0 \iff \ketaZ(\kc_1,\kc_2) =0$, which by  Proposition~\ref{prop:equivalenceOfIndependence} is equivalent to $\ktetaZ(\kc_1,\kc_2)=0$.
  Since every element $\kts\in\ktS$ can be written as $\kts=[\kc,\ktetaZ(\kc)]+\ktt$ for some elements $\kc\in\tail(\kP)^\vee\cap\kQuot$ and $\ktt\in\ktT$,
we have that $\kpi(\kts)=0$ implies  $\kc=0$ and $\kpi(\ktt)=0$, which implies $\kts=0$ because $\ker\kpi_\kT=0$.
\end{proof}

%%%%%%%%%%%%%%%
%%%%%%%%%%%%%%%
\subsection{The $s$ and multiples of  $t$ are in $\ktT$}
\label{ssec:sandtinktT}
The next result shows that the special elements $s_i$ and some multiples of the $t_{ij}$ are always in $\ktT$.
\begin{proposition}
\label{prop-siInT}
For every $v^i\in\VrtxnZ(\kP)$ there exist some  $c^i_1,c^i_2 \in \cone(P)^\vee\cap \kQuot$ such that $\ktetaZ(c^i_1,c^i_2) = s_i$, where $s_i=s_{v^{i}}$ is the corresponding coordinate.
Furthermore, we can also find for each $t_{ij}$ a positive integer
$a_{ij}\in\NN$ such that $a_{ij}t_{ij}\in \ktT$.
\end{proposition}

\begin{proof}
  Let $v^i\in\VrtxnZ(\kP)$. Clearly there exists a $c$ such that $v(c)=v^i$ and $\keta(c) \notin \ZZ$. We may assume  that
  \[
    \keta(c) = z+q \textup{~ with~} z\in\Z, q\in\Q \textup{~ and ~} 0<q \le\frac{1}{2}:
  \]
  otherwise we replace $c$ by $kc$ with $k$ being a  positive integer such that $\keta(kc)+1-\roundup{\keta(kc)}\le \frac{1}{2}$. This brings us to
\begin{eqnarray*}
  \ktetaZ(c,c) &=& \ktetaZ(c)s_i + \ktetaZ(c)s_i - \ktetaZ(2c)s_i\\
               &=& (z + 1)s_i +(z+1)s_i - \roundup{2z + 2q}s_i \\
               &=& s_i.
\end{eqnarray*}
For the second part, we look at one edge $[v^i,v^j]$. We can choose $c_1,c_2$ such that $v(c_1)=v^i$, $v(c_2)=v^j$ and 
furthermore such that $\braket{v^j,c_2}< \braket{v^i,c_2}$ and that
$v(c_1+c_2) = v(c_1)$. Finally, we can assume that all the brackets are integers.
By Definition~\ref{d:etatilde(c)} we then have
\begin{eqnarray*}
  \ktetaZ(c_1,c_2)&= \left(\braket{v^i,c_2} - \braket{v^j,c_2}\right)t_{ij}
\end{eqnarray*}
and by our assumptions the coefficient of $t_{ij}$ is a positive integer.
\end{proof}

%%%%%%%%%%%%%%%%%%%%%%%%%%%%%%%%%%%%%%%%%%%%%%%%%%%%%%%%%%%%%%%%%%%%%%%%%%%
%%%%%%%%%%%%%%%%%%%%%%%%%%%%%%%%%%%%%%%%%%%%%%%%%%%%%%%%%%%%%%%%%%%%%%%%%%%

%%%%%%%%%%%%%%%
%%%%%%%%%%%%%%%
\subsection{Finite generation}
\label{subs:finiteGenerationInTheory}
A  consequence of Proposition~\ref{prop:equivalenceOfIndependence} is that  lifting  the Hilbert basis elements
\[
[\ahb_1,\ketaZ(\ahb_1)],\dots,[\ahb_k,\ketaZ(\ahb_k)]
\]
we obtain generators of  $\ktS$ as a ``$\ktT$-module'':

%%%%%%%%%%%%%%%
%%%%%%%%%%%%%%%
\begin{corollary}
\label{cor-HilbBasisSuff} The following equality holds:
$\,\ktS=\ktT + \spann{\NN}{[\ahb_1,\ktetaZ(\ahb_1)],\dots,[\ahb_k,\ktetaZ(\ahb_k)]}.$
\end{corollary}

\begin{proof}
 Let $\kc\in \tail(\kP)^\vee\cap\kQuot$. Our goal is to prove that 
  \(
    [\kc,\ktetaZ(\kc)]\in  \spann{\NN}{[\ahb_i,\ktetaZ(\ahb_i)] : i=1\dots k}.
  \)
  By Remark~\ref{rem:firstPropertiesOfeta} $[\kc,\ketaZ(\kc)]\in \Rand{\NN}{\cone_\ZZ(\kP)\dual}$, that is $[\kc,\ketaZ(\kc)]\in\spann{\NN}{[\ahb_i,\ketaZ(\ahb_i)] : i=1\dots k}$.
  So there exists an $\ketaZ$-independent sequence consisting of $\ahb_i$s which adds up to $\kc$, and we conclude by Proposition~\ref{prop:equivalenceOfIndependence}.
\end{proof}

%%%%%%%%%%%%%%%
%%%%%%%%%%%%%%%
\begin{proposition}
\label{prop-finGenT}
The semigroup $\ktT$ is finitely generated. A finite set of generators is 
given by the minimally dependent sequences supported on
$\ahb_1,\dots,\ahb_k$, yielding
$\,\ktetaZ(\bfm)$ for certain $\bfm\in\NN^k$.
\end{proposition}

\begin{proof}
  We start by claiming that for any sequence $\kc_1,\dots,\kc_\ell$ of elements from $\tail(\kP)^\vee\cap\kQuot$, there exists an $\bfm\in\NN^k$ such that, with the notation introduced in~\ref{ssec:relationsAndHB}, we have
  \begin{equation}
    \label{eq:veryObviousClaim}
    \ktetaZ(\kc_1,\dots,\kc_\ell) = \ktetaZ(\bfm).
  \end{equation}
  Indeed, for every $\kc_i$ we have
\[
\textstyle
[\kc_i,\,\ketaZ(\kc_i)]=\sum_{j=1}^k m_{ij}\,[\ahb_j,\ketaZ(\ahb_j)] = [\sum_{j=1}^km_{ij}\ahb_j,\sum_{j=1}^km_{ij}\ketaZ(\ahb_j)],
\]
so we can choose $\bfm=\bfm_1+\dots+\bfm_\ell\in\NN^k$:
\begin{eqnarray*}
  \ktetaZ(\kc_1,\dots,\kc_\ell) &=& \ktetaZ(\kc_1) +\dots +\ktetaZ(\kc_\ell) - \ktetaZ(\kc_1+\dots+\kc_\ell) \\
                     &=& \textstyle  \ktetaZ(\sum_{j=1}^k m_{1j}\ahb_j) + \dots + \ktetaZ(\sum_{j=1}^k m_{\ell,j}\ahb_j) - \ktetaZ(\sum_{i=1}^\ell\sum_{j=1}^k m_{ij}\ahb_j)\\
                                &=& \textstyle \sum_{j=1}^k m_{1j}\ktetaZ(\ahb_j) + \dots + \sum_{j=1}^k m_{\ell,j}\ktetaZ(\ahb_j) - \ktetaZ(\sum_{i=1}^\ell\sum_{j=1}^k m_{ij}\ahb_j)\\
                                &=& \ktetaZ(\bfm),
\end{eqnarray*}
so (\ref{eq:veryObviousClaim}) holds. Furthermore, all the $\bfm_i$ are independent, but their sum $\bfm$ is independent if and only if $\kc_1,\dots,\kc_\ell$ are independent.

Next we claim  that for every sequence $\kc_1,\dots,\kc_\ell \in \tail(\kP)^\vee\cap\kQuot$
we can even express  $\ktetaZ(\kc_1,\dots,\kc_\ell)$ 
using a combination of $\ktetaZ(\kbb)$ with minimally dependent arguments from $\NN^k$.
From this second claim, we can immediately conclude. To prove this claim we use double induction:
first with respect to $\ketaZ(\kc_1,\dots,\kc_\ell)\in\NN$, and, inside each induction step we use induction on $\deg(\bfm)=\sum m_{ij}$.
The key of the proof is Remark~\ref{rem-positiveEtaCC}~\ref{item:inductiveStepKteta}
adapted to $\NN^k$: if $\bfm'\le \bfm$ component-wise, then
\begin{equation}
  \label{eq:inductionStepM}
  \textstyle
  \ktetaZ(\bfm) = \ktetaZ(\bfm') + \ktetaZ(\sum_{i=1}^km'_i\ahb_i,\ahb_1^{m_1-m'_1},\dots,\ahb_k^{m_k-m'_k}).
\end{equation}
\indent\emph{The case $\ketaZ(\kc_1,\dots,\kc_\ell)=1$}. If $\deg(\bfm)=2$ we are trivially done. Otherwise, assume $\bfm$ is not minimally dependent and choose $\bfm'<\bfm$ which is also dependent.
By (\ref{eq:inductionStepM}) we have $\ketaZ(\bfm) = \ketaZ(\bfm')=1$ so $\ketaZ(\sum_{i=1}^km'_i\ahb_i,\ahb_1^{m_1-m'_1},\dots,\ahb_k^{m_k-m'_k})=0$. By Proposition~\ref{prop:equivalenceOfIndependence} and (\ref{eq:inductionStepM}) it follows that $\ktetaZ(\bfm)=\ktetaZ(\bfm')$, with $\deg(\bfm')<\deg(\bfm)$, so we conclude by induction on $\deg(\bfm)$.

\emph{The inductive step}  follows very similarly from (\ref{eq:inductionStepM}).

\end{proof}

\begin{question}
Now, where we know that $\ktT$ is a finitely generated semigroup, it might be
interesting to ask for the polyhedral cone generated
by $\ktT$. What are its fundamental rays, and how do its facets look like?
This will be answered partially in \cite{m}. In that paper we will use  different techniques  which will provide a new description of the generators of $\ktT$ as well as  another proof of Proposition~\ref{prop-finGenT}.
\end{question}

\section{The initial object property}
\label{ssec:initialObject}
  
In Proposition~\ref{prop-coCart} we showed that $(\ktT,\ktS)$ belongs to
the category of \cocartesian~\aextension s of the pair $\NN\hookrightarrow \cone_\Z(\kP)\dual$.
For this, we have utilized the fact that $(\ktT,\ktS)$ is ``small enough'', 
i.e.\ that the elements of the space $\kMTZD(\kP)$ satisfy sufficiently many relations, e.g.\ the \se~edge relation, 
which lead us to Proposition~\ref{prop:equivalenceOfIndependence}. This was then crucial to 
Propositions~\ref{prop-coCart} and~\ref{prop-finGenT}.
We are now going to show that  $(\ktT,\ktS)$ it is an initial object in this category, so, in a sense,  we care about the opposite:
We have to show that all these relations within $(\ktT,\ktS)$ (or $\kMTZD(\kP)$) are not arbitrarily
but implicitly part of the structure of  every other \cocartesian~\aextension.
This will allow us to construct a unique map from $(\ktT,\ktS)$ to any other \cocartesian~\aextension.

\begin{notation}
  In contrast to  Section~\ref{discreteSetupStart}  we no longer use $(\kT,\kS)$
  to denote the starting pair $(\NN,\cone_\Z(\kP)\dual)$. 
  Instead, we assume that $(\kT,\kS)$ is an arbitrary \cocartesian~\aextension\ of $(\NN,\cone_\Z(\kP)\dual)$.   
\end{notation}
\noindent Our goal in this section is to define compatible maps
$\kell_\kT:\ktT\to\kT$ and $\kell_\kS:\ktS\to\kS$ and prove the following theorem.
\[ \begin{tikzcd}
     \ktT
     \arrow[hook]{r}
     \arrow{rd}
     \arrow[bend left=30, dotted]{rr}{\kell_\kT}
     & \ktS
     \arrow[bend left=30, dotted, crossing over]{rr}{\kell_\kS}
     & \kT
     \arrow{ld}
     \arrow[hook]{r}
     & \kS
     \arrow{ld}
  \\
     & \NN
     \arrow[hook]{r}
     & \cone_\ZZ(\kP)\dual
     \arrow[crossing over,leftarrow]{lu}
     &    
   \end{tikzcd}
\]

\begin{theorem}
\label{th-initialObject}
The pair $(\ktT,\ktS)$ is an initial object in the category of \cocartesian~\aextension s of the pair $(\NN, \cone_\Z(\kP)\dual)$. 
Down to earth, this means that, for any given $(\kT,\kS)$ inducing
a diagram as above, there exists a unique pair $(\kell_T,\kell_S)$
of compatible maps $\kell_T:\ktT\to\kT$ and $\kell_S:\ktS\to\kS$.
\end{theorem}

\noindent The proof of this theorem will be done in the following subsections
filling the rest of Section~\ref{ssec:initialObject}.

%%%%%%%%%%%%%%%%%%%%%%%%%%%%%%%%%%%%%%%%%%%%%%%%%%%%%%%%%%%%%%%%%%%%%%%%%%%
\subsection{Uniqueness}
\label{initialUniqueness}
The fact that both $(\kpi_\ktT,\kpi_\ktS)$ and $(\kpi_\kT,\kpi_\kS)$
are \cocartesian~\aextension s implies that we have vertical isomorphisms
$\kpi_{\wt{\partial}}:=\kpi_{\ktS}|_{\Rand{\ktT}{(\ktS)}}$
and $\kpi_{{\partial}}:=\kpi_{\kS}|_{\Rand{\kT}{(\kS)}}$.
\[
\xymatrix@C=3em@R=3ex{
\Rand{\ktT}{(\ktS)} \ar@{.>}[rr]^(0.50){\kellp}_{(\sim)}
\ar@{->}[rd]_(0.4){\kpi_{\wt{\partial}}}^\sim &&
\Rand{\kT}{(\kS)} \ar@{->}[ld]^(0.4){\kpi_{\partial}}_\sim
\\
& \Rand{\NN}{(\cone_\Z(\kP)\dual)}
}
\]
In particular, we are allowed and forced to set
\[
\kellp = \kpi_{\partial}^{-1}\circ \kpi_{\wt{\partial}}
\]
which will become the unique restriction to $\Rand{\ktT}{(\ktS)}$
of any possible $\kell_S:\ktS\to\kS$. 
Moreover, it follows that, like $\kpi_{\wt{\partial}}$ and
$\kpi_{{\partial}}$, the map $\kellp$ is bijective too.
\begin{notation}
  \label{not:always_kc}
  In this section $\kc, \kc_i$ will always denote elements from the semigroup $\tail(\kP)\dual\cap\kQuot$. For every $\kc$ we write
  \begin{eqnarray*}
    \kellp(\kc)&:=&\kellp\big([\kc,\ktetaZ(\kc)]\big) =\kpi_{\partial}^{-1}\big([\kc,\ketaZ(\kc)]\big),\\
    \kell(\kc_1,\kc_2)&:=&\kellp(\kc_1)+\kellp(\kc_2)-\kellp(\kc_1+\kc_2) \in \kT-\kT.
  \end{eqnarray*}  
 So we can also regard $\kellp$ as a map $\kellp:\tail(\kP)\dual\cap\kQuot\to\Rand{\kT}{(\kS)}$.
\end{notation}
\noindent Recall that from Remark~\ref{rem:firstPropertiesOfeta} and as a consequence of Proposition~\ref{prop-coCart} we have
\begin{eqnarray*}
  \Rand{\NN}{(\cone_\Z(\kP)\dual)}&=&\{[\kc,\ketaZ(\kc)]\kst \kc\in \tail(\kP)\dual\cap\kQuot\}\quad\text{and}\\
  \Rand{\ktT}{\ktS}&=&\{[\kc,\ktetaZ(\kc)]\kst \kc\in \tail(\kP)\dual\cap\kQuot\},
\end{eqnarray*}
with  $\kpi_{\wt{\partial}}:[\kc,\ktetaZ(\kc)]\mapsto [\kc,\ketaZ(\kc)]$.
So $\ktT$ is generated by a combination of elements from the boundary:
\[
[0,\,\ktetaZ(\kc_1,\kc_2)]=
[\kc_1,\,\ktetaZ(\kc_1)] + [\kc_2,\,\ktetaZ(\kc_2)] 
- [\kc_1+\kc_2,\,\ktetaZ(\kc_1+\kc_2)].
\]
Thus, since $\kell_S$ is supposed to equal the unique $\kellp$ on the summands on the right hand side, it is uniquely determined too.

\subsection{Defining the maps}
\label{subsec:definingTheMapsEll}
There is thus not much choice in defining the maps $\kell$: on the boundary it has to be $\kellp= \kpi_{\partial}^{-1}\circ \kpi_{\wt{\partial}}$ and $\kell_\ktT$ it has to satisfy
\[
  \kell_\ktT([0,\ktetaZ(\kc_1,\kc_2)]) = \kellp\big([\kc_1,\,\ktetaZ(\kc_1)]\big) + \kellp\big([\kc_2,\,\ktetaZ(\kc_2)]\big) -
  \kellp\big([\kc_1+\kc_2,\,\ktetaZ(\kc_1+\kc_2)]\big).
\]
We can then define
\[
  \kell_\ktS(\awt{s}) := \kellp(\awt{\partial}(\awt{s})) + \kell_\ktT(\awt{\lambda}(\awt{s})).
\]
We first have to check that the maps land where they are supposed to. For $\kellp$ this holds by definition. For $\kell_\ktT$ this follows directly from Proposition~\ref{prop:boundaryIndepEquivalence}. Furthermore, by the same proposition, the map $\kellp$ is as linear as it may be:
%%%%%%%%%%%%%%%
%%%%%%%%%%%%%%%
\begin{proposition}
\label{prop-kellpLinear}
For all $\kc_1,\kc_2$  we have $\kell_\ktT([0,\ktetaZ(\kc_1,\kc_2)])\in\kT$.
Moreover, $\kc_1,\dots,\kc_r$ are $\ketaZ$-independent if and only if $\kellp(\kc_1),\dots,\kellp(\kc_r)$ are boundary independent.
\end{proposition}
\begin{remark}
  \label{rem:ellTisEnough}
  If $\kell_\ktT$ is a well-defined semigroup homomorphism, then so is $\kell_\ktS$. Being well-defined follows from the uniqueness of the decomposition $\awt{s}=\awt{\partial}(\awt{s}) + \awt{\lambda}(\awt{s})$. The fact that $\kell_\ktS(\awt{s}_1+\awt{s}_2)=\kell_\ktS(\awt{s}_1)+\kell_\ktS(\awt{s}_2)$ is a consequence of Proposition~\ref{prop:boundaryIndepEquivalence} combined with the easy remark that for any \straight  pair $(\ktT,\ktS)$ and any $s_1,s_2\in\ktS$ we have
  \begin{eqnarray*}
    \partial(s_1+s_2) &=&\partial(\partial(s_1)+\partial(s_2)),\quad\text{and}\\
    \lambda(s_1+s_2) &=& \lambda(s_1)+\lambda(s_2)+\lambda(\partial(s_1)+\partial(s_2)).
  \end{eqnarray*}
\end{remark}

The hard part is to show that $\kell_\ktT$ is   well-defined  i.e.\ that it depends only on the element
$\ktetaZ(\kc_1,\kc_2)$ but not on the individual $\kc_1,\kc_2$.  This will be a consequence of Lemma~\ref{lem-varphEll}.
So, for most of the remainder of this section we will work towards this goal. We will use the $s$ and $t$ coordinates introduced in Section~\ref{sec:theAbmientSpace} and prove that there are corresponding elements in $\kT$ as well, and then show that these corresponding elements satisfy the relations from Definitions~\ref{def:Clin}~and~\ref{def:TP}.
The idea is to recover $\kell_\kT$ from a linear map
$\kMTD(\kP)\to (\kT-\kT)\otimes_\ZZ\RR$.
So the elements $s_v$ and $t_{ij}$ are important because they generate
$\kMTD(\kP)$, and because the relations (such as those arising from the \se~edges) are
formulated in terms of the elements $s_v$ and $t_{ij}$.
This is the rough idea of the next sections.

%%%%%%%%%%%%%%%%%%%%%%%%%%%%%%%%%%%%%%%%%%%%%%%%%%%%%%%%%%%%%%%%%%%%%%%%%%%
\subsection{Recovering the $s$-parameters}
\label{recoverS}
Recall from Definition~\ref{d:etatilde(c)} that the elements $\kteta(\kc)$
depend linearly on $\kc$ whenever the vertex $v(\kc)$ is not changing.
That means that, fixing some vertex $v$ of $\kP$, the map
\[
\kteta(\kbb):\normal(v,\kP)\subseteq \tail(\kP)\dual \to \kMTD(\kP)
\] 
is a linear on the normal cone
$\normal(v,\kP)\subseteq \tail(\kP)\dual\subseteq \kQuotR$; it defines some
element $\kteta_v\in\kQuotRD\otimes\kMTD(\kP)$.
Let us denote simplify further the notation introduced in~(\ref{eq:etaFrac}) by setting
\[
\etaFrac{\kc}:= \etaFrac{\keta(\kc)} = \ketaZ(\kc)-\keta(\kc)\in [0,1)\subset\RR.
\]
Then Definition~\ref{d:etatilde(c)} turns into
$\,\ktetaZ(\kc)=\kteta(\kc)+ \etaFrac{\kc}\cdot s_{v(\kc)}\in\kMTZD(\kP)$,
and, via $\kpi$, this element maps to 
$\,\ketaZ(\kc)=\keta(\kc)+ \etaFrac{\kc}\in\ZZ$.
In Proposition~\ref{prop-siInT} we have used elements
$\kc$ with $\etaFrac{\kc}\in[\frac{1}{2},1)$ to represent
$s_{v(\kc)}=\ktetaZ(\kc,\kc)$. This generalizes to the fact that
\[
\renewcommand{\arraystretch}{1.3}
\ktetaZ(\kc_1,\kc_2)=\left\{\begin{array}{ll}
s_v & \mbox{\rm if $\etaFrac{\kc_1}+\etaFrac{\kc_2}\geq 1$}\\
0 & \mbox{\rm if $\etaFrac{\kc_1}+\etaFrac{\kc_2}< 1$}.
\end{array}\right.
\]
whenever $\kc_i\in\normal(v,\kP)\cap\kQuot$, i.e.\ 
whenever $v$ can be chosen as $v(\kc_i)$ ($i=1,2$).
Now, the first step into the direction of establishing the map
$\kell$ is that this independence on the special choice of elements
$\kc_i\in\normal(v,\kP)$ remains true in $\kS$.

%%%%%%%%%%%%%%%
%%%%%%%%%%%%%%%
\begin{proposition}
\label{prop-representingC}
\begin{enumerate}[label={(\roman*)}]
\item \label{item:repC1}%{\rm 1)}
Assume that $v\in\kP$ is a vertex. Then there is an element $\kells(v)\in\kS$
such that for all $\kc_1,\kc_2\in\normal(v,\kP)\cap\kQuot$ we have
\[
\renewcommand{\arraystretch}{1.3}
\kell(\kc_1,\kc_2)=\left\{\begin{array}{ll}
\kells(v) & \mbox{\rm if $\etaFrac{\kc_1}+\etaFrac{\kc_2}\geq 1$}\\
0 & \mbox{\rm if $\etaFrac{\kc_1}+\etaFrac{\kc_2}< 1$}.
\end{array}\right.
\]
\item \label{item:repC2}%{\rm 2)}
If $\kc\in\normal(v,\kP)\dual\cap\kQuot$ with 
$n\in\NN$ being the smallest positive integer such that
$n\cdot\etaFrac{\kc}\geq 1$, 
e.g.\ if $\etaFrac{\kc}=1/n$, 
% with $n\geq 2$, 
then we obtain 
$\kells(v)= n\cdot \kellp(\kc) -\kellp(n\kc)$.
\end{enumerate}
\end{proposition}

\begin{proof}
\ref{item:repC1}
{\;\em Step 1.}
We check first that $\kell(\kc_1,\kc_2)=0$ whenever
$\etaFrac{\kc_1}+\etaFrac{\kc_2}<1$. This inequality is equivalent to the
equality
\[
\etaFrac{\kc_1}+\etaFrac{\kc_2}=\etaFrac{\kc_1+\kc_2},
\]
i.e.\ it yields
\[
\ketaZ(\kc_1)+\ketaZ(\kc_2)=\ketaZ(\kc_1+\kc_2).
\]
Hence, $\kell(\kc_1,\kc_2)=0$ follows from Proposition~\ref{prop-kellpLinear}.
Note that the assumption of the just proven claim is trivially
fulfilled if,  $\etaFrac{\kc_1}=0$, i.e.\ if
$\keta(\kc_1)$ is an integer. We will use this in the next step.
\\[1ex]
{\;\em Step 2.}
Assume that $\kc,\kc'\in\normal(v,\kP)\cap\kQuot$ such that
$\etaFrac{\kc}=1/n$ and that $\etaFrac{\kc'}=(n-k)/n$ for some 
not necessarily coprime natural numbers $n\in\NN$ and
$k\in\{1,\ldots,n-1\}$. Then
\[
\kell(\kc',\,k\cdot\kc) = 
n\cdot\kellp(\kc)-\kellp(n\kc)=
\kell(a\cdot\kc,\, b\cdot\kc)
\hspace{0.6em}
\mbox{for all $a,b\in\ZZ_{\geq 1}$ with $a+b=n$}.
\]
Step 1 immediately implies the second equality.
To check the first one, we have to show that
\[
\kellp(\kc')+\kellp(k\kc)-\kellp(\kc'+k\kc)
=
n\cdot\kellp(\kc)-\kellp(n\kc).
\]
Since Step 1 yields $\kellp(k\kc)=k\cdot\kellp(\kc)$,
this reduces to the claim
\[
\kellp(\kc') + \kellp(n\kc) =
(n-k)\cdot \kellp(\kc) + \kellp(\kc'+k\kc).
\]
However, since $\etaFrac{n\kc}=\etaFrac{\kc'+k\kc}=0$, the expression
$\kellp$ is linear on both sides, i.e.\ both sides are equal
to $\kellp(\kc'+n\kc)$.
\\[1ex]
{\;\em Step 3.}
Assume that  $\etaFrac{\kc_1}+\etaFrac{\kc_2}\geq 1$;
in particular, that both summands are positive.
For the present Step 3 we suppose that we have found an element 
$\kc\in\normal(v,\kP)\cap\kQuot$ such that
it satisfies the assumption made in Step 2 with respect to
both $\kc':=\kc_1,\kc_2$. That is,
$\etaFrac{\kc}=1/n$ and $\etaFrac{\kc_i}=(n-k_i)/n$
with $k_i\in\{1,\ldots,n-1\}$ for $i=1,2$.
This leads to the equalities
\[
\kellp(\kc_i) + k_i\cdot\kellp(\kc)
-\kellp(\kc_i+k_i\cdot\kc)
= \kell\big(\kc_1, \, k_i\cdot\kc\big)
= n\cdot\kellp(\kc)-\kellp(n\kc),
\]
hence
\[
\kellp(\kc_i) = \kellp(\kc_i+k_i\cdot\kc)
+ (n-k_i)\cdot \kellp(\kc)-\kellp(n\kc).
\]
Alternatively, we could also take $\kc':=\kc_1+\kc_2$ instead of the single
$\kc_i$. Since 
$k_1+k_2<n$,
we have to replace the coefficients
$k_i$ by $(k_1+k_2)$. 
This leads to
\[
\kellp(\kc_1+\kc_2) = \kellp(\kc_1+\kc_2+(k_1+k_2)\cdot\kc)
+ (n-k_1-k_2)\cdot \kellp(\kc)-\kellp(n\kc).
\]
Using these equations, we obtain
\[
\begin{array}{rcl}
\kell(\kc_1,\kc_2)
&=& \kellp(\kc_1) + \kellp(\kc_2) - \kellp(\kc_1+\kc_2)
\\
&=& 
\begin{array}[t]{@{}l}
\kellp(\kc_1+k_1\cdot\kc) +
\kellp(\kc_2+k_2\cdot\kc) 
-\kellp(\kc_1+\kc_2+(k_1+k_2)\cdot\kc)
\\
\hspace{20em}
+ n\cdot \kellp(\kc) -\kellp(n\kc).
\end{array}
\end{array}
\]
The arguments of the first two summands have the property that
$\etaFrac{\kbb}=0$, i.e.\ $\ketaZ(\kbb)=\keta(\kbb)$. In particular,
since $\kellp$ is linear in this case, their sum cancels
with the third summand. Altogether this yields
\[
\kell(\kc_1,\kc_2)= n\cdot \kellp(\kc) -\kellp(n\kc).
\]
{\em Step 4.}
Since $v$ is a rational vertex of $\kP$, we know that the denominators
of all $\keta(\kc)$ and hence that of all fractional parts $\etaFrac{\kc}$
with $\kc\in\normal(v,\kP)\cap\kQuot$ are bounded. If $n$ is the maximal
denominator among them, 
then we can find a special $\kc\in\normal(v,\kP)\cap\kQuot$
with $\etaFrac{\kc}=1/n$. We will fix this element and set
\[
\kells(v):= n\cdot \kellp(\kc) -\kellp(n\kc).
\]
And now we can apply Step 3 for any given $\kc_1,\kc_2\in
\normal(v,\kP)\cap\kQuot$ and our fixed $\kc$.
\\[1ex]
\ref{item:repC2} This is a direct consequence of the first part of the proposition
and of Proposition~\ref{prop-kellpLinear}.
\end{proof}

\begin{remark}
\label{rem-whyEllV}
The meaning of the elements $\kells(v)$ is that $\kell_\kT$ will map $s_v$ onto $\kells(v)$.
Hence the existence of  $\kells(v)$ is, on the one hand, a necessary condition for the existence of
$\kell_\kT$, but, on the other, it will also help to prove it.
\end{remark}

%%%%%%%%%%%%%%%%%%%%%%%%%%%%%%%%%%%%%%%%%%%%%%%%%%%%%%%%%%%%%%%%%%%%%%%%%%%
\subsection{Recovering the $s$-equations (\ref{eq:dualEquationsT(Q)3}) for lattice-disjoint edges}
\label{recoverNoLatticePt}
In Definition~\ref{def:TP}
we had imposed the equations {$s_i=s_j$} 
on the vector space $\kMT(\kP)$ for 
compact edges $\aed=[v^i,v^j]$ with $[v^i,v^j]\cap\kQuotD=\emptyset$.
These impose the equality $s_i=s_j\in\kMTD(\kP)$.
Hence, for the well-definition of the map $\kell_\kT:\ktT\to\kT$,
we have to check that this leads to the equality
$\kells(v^i)=\kells(v^j)$ inside $\kT$ too.
\\[1ex]
Recall from Definition~\ref{def-shortEdge} that   $g_\aed\in\ZZ_{\geq 1}$ is minimal 
such that the affine line $\ko{g_\aed\cdot \aed}$ spanned by $g_\aed\cdot \aed$ contains 
lattice points.
If $g_\aed=1$, then we may choose some $w\in\ko{\aed}\cap\kQuotD$,
and for any integral
\[
\kc\in\normal(\aed,\kP)=\normal(v^i,\kP)\cap\normal(v^j,\kP)
\]
we obtain that
\[
\keta(\kc) = -\langle v^i,\kc\rangle = -\langle v^j,\kc\rangle
=-\langle w,\kc\rangle\in\ZZ,
\]
i.e.\ that $\etaFrac{\kc}=0$. That means that those $\kc$ do not qualify
to determine neither $\kells(v^i)$, nor $\kells(v^j)$ via
Proposition~\ref{prop-representingC}\,\ref{item:repC2}. While this is bad news,
the point is that the reverse implication works as well:
Assume that $g_\aed\geq 2$. Considering the projection
\[
\kQuotRD\surj \kQuotRD/\RR(v^j-v^i)=:\ko{\kQuotRD}
\]
the polyhedron $\kP$ maps
to a polyhedron $\ko{\kP}$, and the edge $\aed$ becomes a vertex
$\ko{\aed}$ of $\ko{\kP}$. Within the dual setup, the injection
$\ko{\kQuotR}\hookrightarrow \kQuotR$ sends
$\normal(\ko{\aed},\ko{P})$ isomorphically to $\normal(\aed,\kP)$.
The assumption $g_\aed\geq 2$ means that $\ko{\aed}$ is not a lattice point
in $\ko{\kQuotRD}$. In particular, there are integral
$\kc\in\normal(\ko{\aed},\ko{P})\stackrel{\sim}{\to}\normal(\aed,\kP)$
such that $\langle\ko{\aed},\kc\rangle\notin\ZZ$. However, this number equals
$\langle v^i,\kc\rangle = \langle v^j,\kc\rangle=-\keta(\kc)$.
As a direct consequence, we obtain the following.

%%%%%%%%%%%%%%%
%%%%%%%%%%%%%%%
\begin{proposition}
\label{prop-gGreaterThanTwo}
If $\aed=[v^i,v^j]$ is an edge with $g_\aed\geq 2$, then
$\kells(v^i)=\kells(v^j)$ inside $\kT$.
\end{proposition}

\begin{proof}
Using the element $\kc\in\normal(\aed,\kP)\cap\kQuot$ with $\etaFrac{\kc}\neq 0$
constructed right before the proposition, we denote by $n\geq 2$ the smallest
positive integer such that $n\cdot\etaFrac{\kc}\geq 1$. Then, \ref{item:repC2}~of
Proposition~\ref{prop-representingC} implies that
$\kells(v^i)= n\cdot \kellp(\kc) -\kellp(n\kc)=\kells(v^j)$.
\end{proof}

The task mentioned at the beginning of the present subsection is not 
fulfilled yet -- it remains to show that
$\kells(v^i)=\kells(v^j)$ for the lattice-disjoint edges $\aed=[v^i,v^j]$ 
with $g_\aed=1$. While we have already indicated that the method of the proof of
Proposition~\ref{prop-gGreaterThanTwo} does not work here, we are saved by the
fact that, supposed that $g_\aed=1$ and $v^i,v^j\notin\kQuotD$, the property
$\aed\cap\kQuotD=\emptyset$ is equivalent to $\aed$ being a \se~edge (see Definition~\ref{def-shortEdge}).
Thus, we can and will postpone this case until we have studied the elements $\kell(t_{ij})$ where $t_{ij}$
is the dilation parameter.

%%%%%%%%%%%%%%%%%%%%%%%%%%%%%%%%%%%%%%%%%%%%%%%%%%%%%%%%%%%%%%%%%%%%%%%%%%%
\subsection{Recovering the $t$-parameters}
\label{recoverT}
In Subsection~\ref{recoverS} we have utilized the fact that
$\kteta(\kbb)$ is linear on the normal cones $\normal(v,\kP)$ for vertices
$v\in\kP$. In the present subsection, however, we start with an edge
$\aed=[v^1,v^2]$ leading to the normal cones
\[
\normal(\aed,\kP)=\normal(v^1,\kP)\cap \normal(v^2,\kP).
\]
Here, we have to pay attention that
the function $\kteta(\kbb)$ is linear on each individual $\normal(v^i,\kP)$,
but not on their union. In particular, the function
$\kc\mapsto\etaFrac{\kc}$ ceases to be linear (even mod $\ZZ$) when crossing
the boundaries of normal cones.

%%%%%%%%%%%%%%%
%%%%%%%%%%%%%%%
\begin{definition}
\label{def-superintegral}
We call an element $\kc\in\tail(\kP)\dual$ super integral if it 
belongs to $\kQuot$ and has integral
values on all (rational, but not necessarily integral) vertices of $\kP$.
In particular, super integral elements $\kc$ satisfy 
$\keta(\kc)\in\ZZ$, hence $\etaFrac{\kc}=0$.
This notion is additive, i.e.\ the set of super integral elements
form a sublattice $\kQuotSup\subseteq\kQuot$.
\end{definition}

Now, assume that $\kc_i\in\normal(v^i,\kP)$ ($i=1,2$) are super integral
such that $\kc_1+\kc_2\in\normal(v^1,\kP)\cup\normal(v^2,\kP)$.
Note that the latter condition is automatic if 
$\normal(v^1,\kP)\cup\normal(v^2,\kP)$ is convex.
Denoting $\vect{\aed}:=v^2-v^1$, this implies that
$\braket{ \vect{\aed},\kc_1},\braket{ -\vect{\aed},\kc_2}\geq 0$.
In Proposition~\ref{prop-siInT}, we have related
the element $\ktetaZ(\kc_1,\kc_2)$ to the edge parameter $t=t_{12}$.
The exact statement generalizes to the fact that
\[
\ktetaZ(\kc_1,\kc_2)=
\min\{\braket{ \vect{\aed},\kc_1},\;\braket{ -\vect{\aed},\kc_2}\}\cdot t.
\]
Now, the next towards establishing the map
$\kell$ is that this special dependence on the choice of elements
$\kc_i\in\normal(v^i,\kP)$ remains true in $\kT$.
%%%%%%%%%%%%%%%
%%%%%%%%%%%%%%%
\begin{proposition}
\label{prop-representingD}
There is an element $\kellt(\aed)\in\QQ_{>0}\cdot\kT$
such that for all super integral $\kc_i\in\normal(v^i,\kP)$ ($i=1,2$) with 
$\kc_1+\kc_2\in\normal(v^1,\kP)\cup\normal(v^2,\kP)$ we have
$
\kell(\kc_1,\kc_2)=
\min\{\braket{ \vect{\aed},\kc_1},\;\braket{ -\vect{\aed},\kc_2}\}\cdot \kellt(\aed)
$.
\end{proposition}

\begin{proof}
We may assume that $\braket{ \vect{\aed},\kc_1}\geq \braket{ -\vect{\aed},\kc_2} \;(\geq 0)$.
In this case, the claim turns into the equation
$\,\kell(\kc_1,\kc_2)=\braket{ -\vect{\aed},\kc_2}\cdot \kellt(\aed)$.
\\[0.5ex]
{\em Step 1:} Show that $\kell(\kc_1,\kc_2)$ does indeed not depend on $\kc_1$, 
provided that it does not leave the range 
$\normal(v^1,\kP)\cap \big[\braket{ \vect{\aed},\kbb}\geq \braket{ -\vect{\aed},\kc_2}\big]$.
If $\kc_1'$ is another candidate, then we obtain
\[
\kell(\kc_1,\kc_2)-\kell(\kc_1',\kc_2)=
\kellp(\kc_1) - \kellp(\kc_1+\kc_2)
- \kellp(\kc_1') + \kellp(\kc_1'+\kc_2).
\]
Hence, as our goal is $\kell(\kc_1,\kc_2) = \kell(\kc_1',\kc_2)$, we have to show that
\[
\kellp(\kc_1) + \kellp(\kc_1'+\kc_2) =
\kellp(\kc_1') + \kellp(\kc_1+\kc_2).
\]
The inequalities 
$\braket{ \vect{\aed},\kc_1}, \braket{ \vect{\aed},\kc_1'} \geq \braket{ -\vect{\aed},\kc_2}$
imply that both $\kc_1+\kc_2$ and $\kc_1'+\kc_2$ belong to
$\normal(v^1,\kP)$, 
i.e.\ to the same normal cone which already contains $\kc_1$ and $\kc_1'$.
In particular, since $\kc_1,\kc_1'$ are super integral,
Proposition~\ref{prop-kellpLinear} shows that
the map $\kellp$ behaves linearly on both sums, 
i.e.\ adding up to $\kellp(\kc_1+\kc_1'+\kc_2)$ in both cases.
\\[1ex]
{\em Step 2:} Fix a super integral $\kc_1\in\normal(v^1,\kP)$ and show that
$\kell(\kc_1,\kbb)$ is an additive function on
\[
B(\kc_1):=\{\kc_2\in\normal(v^2,\kP)\kst
\kc_1+\kc_2\in \normal(v^1,\kP)\mbox{ and }
\braket{ v^j,\kc_2}\in\ZZ\mbox{ for } j=1,2\}.
\]
Note that $B(\kc_1)$ is not a cone.
However, if $\kc_2,\kc_2'\in B(\kc_1)$ with $\kc_2+\kc_2'\in B(\kc_1)$, 
then we obtain
\[
\renewcommand{\arraystretch}{1.5}
\begin{array}[t]{rcl}
\kell(\kc_1,\kc_2+\kc_2')-\kell(\kc_1,\kc_2)-\kell(\kc_1,\kc_2')
&=&
\renewcommand{\arraystretch}{1.0}
\begin{array}[t]{@{}l}
\hspace{-1em}\phantom{ - }\kellp(\kc_1) + \kellp(\kc_2+\kc_2') 
- \kellp(\kc_1+\kc_2+\kc_2')
- \kellp(\kc_1)- \\
\hspace{-1em}- \kellp(\kc_2) + \kellp(\kc_1+\kc_2)
- \kellp(\kc_1) - \kellp(\kc_2') + \kellp(\kc_1+\kc_2')
\end{array}
\\
&=&
\renewcommand{\arraystretch}{1.0}
\begin{array}[t]{@{}l}
\hspace{-1em}\phantom{ - }\kellp(\kc_1+\kc_2) + \kellp(\kc_1+\kc_2')
+ \kellp(\kc_2+\kc_2') -\\
\hspace{-1em}
- \kellp(\kc_1) - \kellp(\kc_2) - \kellp(\kc_2')
- \kellp(\kc_1+\kc_2+\kc_2').
\end{array}
\end{array}
\]
Since $\kc_2,\kc_2'\in\normal(v^2,\kP)$ are super integral, we know that
$
\,\kellp(\kc_2) + \kellp(\kc_2') =
\kellp(\kc_2+\kc_2'). 
$
This transforms the previous expression into
\[
\kell(\kc_1,\kc_2+\kc_2')-\kell(\kc_1,\kc_2)-\kell(\kc_1,\kc_2')
\;= \;
\kellp(\kc_1+\kc_2) + \kellp(\kc_1+\kc_2')
- \kellp(\kc_1) - \kellp(\kc_1+\kc_2+\kc_2'),
\]
and the additivity claim that we want to prove for $\kell(\kc_1,\kbb)$  is equivalent to the equality
\[
\kellp(\kc_1+\kc_2) + \kellp(\kc_1+\kc_2')
\;=\;
\kellp(\kc_1) + \kellp(\kc_1+\kc_2+\kc_2').
\]
The integrality of $\braket{ v^1,\kc_1+\kc_2}$ (and similarly for
$\kc_2'$) implies that the left hand side equals the compact expression
$\kellp(2\kc_1+\kc_2+\kc_2')$. The same argument applies for the right hand
side, yielding the same value.
\\[1ex]
{\em Step 3:} Define the map
\(
\psi:\{\kc_2\in\normal(v^2,\kP)\cap\kQuot\kst 
\braket{ v^i,\kc_2}\in\ZZ,\; i=1,2\}
\too\kT
\)
as
\[
  \psi(\kc_2):=\kell(\kc_1,\kc_2)
\]
under use of any $\kc_1\in \normal(v^1,\kP)\cap\kQuot$ with
$\braket{ v^i,\kc_1}\in\ZZ$ ($i=1,2$) and
$\kc_1+\kc_2\in \normal(v^1,\kP)$, 
i.e.\ such that $\kc_2\in B(\kc_1)$ from Step 2.\\
While it is obvious that those elements $\kc_1$ exist, it is a consequence of
Step 1 that the definition of $\psi(\kc_2)$ does not depend on their choice.
Moreover, for any $\kc_2,\kc_2'\in \normal(v^2,\kP)\cap\kQuot$
with $\braket{ v^i,\kc_2}, \braket{ v^i,\kc_2'}\in\ZZ$
($i=1,2$), we can find an element $\kc_1$ 
such that $\kc_2,\kc_2',\kc_2+\kc_2'\in B(\kc_1)$. Hence, it follows from
Step 2, that $\psi$ is an additive function.
\\[0.5ex]
By definition, it is clear that $\vect{\aed}\leq 0$ on $\normal(v^2,\kP)$,
and we may restrict $\psi$ to the $\vect{\aed}$-face
\[
\normal(v^2,\kP)\cap (\vect{\aed})^\bot \;=\; \normal(\vect{\aed},\kP)
\;\subseteq\;\normal(v^1,\kP).
\]
That means that both arguments from $\psi(\kc_2)=\kell(\kc_1,\kc_2)$
become super integral elements of $\normal(v^1,\kP)$, i.e.\
they satisfy the linearity relation
$\ketaZ(\kc_1) + \ketaZ(\kc_2) = \ketaZ(\kc_1+\kc_2)$.
Thus, Proposition~\ref{prop-kellpLinear} implies 
that $\kell(\kc_1,\kc_2)=0$ on $(\vect{\aed})^\bot$.
It follows that $\psi$ extends to a linear map
\[
\normal(v^2,\kP)/(\vect{\aed})^\bot \to \Q_{\geq 0}\cdot\kT,
\]
i.e.\ it is of the form $\psi(\kc_2)=\braket{ -\vect{\aed},\kc_2}\cdot\kellt(\vect{\aed})$
for some element $\kellt(\aed)\in \Q_{\geq 0}\cdot\kT$.
\end{proof}

%%%%%%%%%%%%%%%%%%%%%%%%%%%%%%%%%%%%%%%%%%%%%%%%%%%%%%%%%%%%%%%%%%%%%%%%%%%
\subsection{Recovering the $s/t$-equations (\ref{eq:dualEquationsT(Q)4}) for \se~edges}
\label{recoverShortE}
Note that the vector $\vect{\aed} = v^2-v^1$ spans the 1-dimensional 
$\Q$-vector space associated to the edge $\aed=[v^1,v^2]$. While the affine line
spanned by $\aed$ might lack lattice points (i.e.\ $g_\aed\ge 2$ from Definition~\ref{def-shortEdge}), the intersection $(\Q\cdot \vect{\aed})\cap\kQuotD$ 
can be identified with $\Z$. It is dual to $\kQuot/(\vect{\aed})^\bot=\Z$.
Choose a representative $\kc_+\in\kQuot$ lifting $1$.
Note that, in the case of $g_\aed\geq 2$, the choice might indeed matter, cf.\
Subsubsection~\ref{charactShortEdges}.
\\[1ex]
We fix an element $w\in \innt\normal(\aed,\kP)$, meaning that $\braket{ w, v^1} = \braket{ w, v^2}$ is strictly
less than the value of $w$ on all other vertices of $\kP$. 
We will, additionally, assume that it is super integral $w\in\kQuotSup$,
meaning that it has integral values on all, even on the non-integral, vertices of $\kP$.
This allows us to choose and fix an $A\gg 0$ leading to elements
\[
\kc_1:= \kc_+ +A\cdot w \in\normal(v^1,\kP)
\hspace{1em}\mbox{and}\hspace{1em}
\kc_2:= -\kc_+ +A\cdot w \in\normal(v^2,\kP).
\]
In particular, 
\[
\kc_0:=\kc_1+\kc_2 =2A\cdot w \in \normal(\aed,\kP).
\] 
In contrast to $w$, the values of $\kc_+$ on $v^1,v^2$ might be non-integral.
Let $n\in g_\aed\cdot\Z_{\geq 1}$ such that
\[
n\cdot \braket{ \kQuot, v^i} \in \Z
\hspace{1em}(i=1,2).
\]
The idea is now to frequently make use of Proposition~\ref{prop-kellpLinear}
stating that any linearity from $\ketaZ$ transfers directly to $\kellp$.
On the other hand, when looking for linearity instances of $\ketaZ$,
having $w$ (or an integral multiple) as one argument does always help:
First, since $w$ is contained
in $\normal(\aed,\kP)$, the function $\keta$ acts linear into both regions
$\normal(v^1,\kP)$ and $\normal(v^2,\kP)$. Second, the integrality assumption
for $w$ implies $\ketaZ(w)=\keta(w)$.

%%%%%%%%%%%%%%%%%%%%%%%%%%%%%%%%%%%%%%%%%%%%%%%%%%%%%%%%%%%%%%%%%%%%%%%%%%%
\subsubsection{The first recursion formula}
\label{firstRec}
For any $h\in\Z_{\geq 0}$ we consider the differences 
\[
\ketaZ(h\kc_1,\kc_1)=
\ketaZ(h\kc_1)+\ketaZ(\kc_1)-\ketaZ\big((h+1)\cdot\kc_1\big)\in\NN.
\]
Since both arguments sit in the same normal cone,
i.e.\ $\keta$ behaves linear, we know $\ketaZ(h\kc_1,\kc_1)\in\{0,1\}$.
\\[1ex]
{\em Case 1: $\;\ketaZ(h\kc_1,\kc_1)=0$. }
Then Proposition~\ref{prop-kellpLinear} implies 
$\kellp(h\kc_1)+\kellp(\kc_1)=\kellp\big((h+1)\cdot\kc_1\big)$.
\\[1ex]
{\em Case 2: $\;\ketaZ(h\kc_1,\kc_1)=1$. }
Now, Proposition~\ref{prop-representingC} says that
\[
\kell(h\kc_1,\kc_1)=
\kellp(h\kc_1)+\kellp(\kc_1)-\kellp\big((h+1)\cdot\kc_1\big)
=\kells(v^1).
\]
Hence, we can express 
\[
\kellp\big((h+1)\cdot\kc_1\big)
=
\kellp(h\kc_1)+\kellp(\kc_1) -
\left\{\begin{array}{ll}
0 & \mbox{in Case 1}\\
\kells(v^1) & \mbox{in Case 2.}
\end{array}\right.
\]
Assume that, for $h=0,\ldots,n-1$, the Cases 1 and 2 occur
$(n-k)$ and $k$ times, respectively. Then, 
since $\kellp(0\cdot\kc_1)=0$, these recursion formulae add
up to
$
\kellp(n\kc_1)=n\cdot\kellp(\kc_1) - k\cdot\kells(v^1).
$
Analogously, utilizing the corresponding $k_2$ replacing $k_1:=k$, we obtain 
the same formula for $\kells(v^2)$. Hence,
\[
n\cdot\kellp(\kc_i)-\kellp(n\kc_i)=k_i\cdot\kells(v^i)
\hspace{1em}(i=1,2).
\]
Finally, we use Proposition~\ref{prop-representingD}. Note that
$\kc_1,\kc_2$ do not meet the assumptions, but
$n\kc_1,n\kc_2$ do. Hence
\[
\kell(n\kc_1,n\kc_2)=\kellp(n\kc_1)+\kellp(n\kc_2)-\kellp(n\kc_0)=\braket{\vect{\aed},n\kc_1}\cdot\kellt(\aed).
\]

%%%%%%%%%%%%%%%%%%%%%%%%%%%%%%%%%%%%%%%%%%%%%%%%%%%%%%%%%%%%%%%%%%%%%%%%%%%
\subsubsection{The relation between $\kellp$- and $\ketaZ$-equations}
\label{kellpVSketaZ}
Recall from Subsection~\ref{initialUniqueness} 
that we have an additive map $\kpi_\partial$ 
sending $\kellp(\kc)\mapsto [\kc,\ketaZ(\kc)]$.
Followed by the projection to $\Z$, this becomes
$\kellp(\kc)\mapsto \ketaZ(\kc)$. That means that all equations 
among the $\kellp(\kc)\in\kS$ we obtained so far (or in the upcoming text)
induce the same equations among the integers $\ketaZ(\kc)$.
Actually it is the point of claims as treated in Subsubsection~\ref{firstRec}
to deal with the reverse direction, i.e.\ lifting certain $\ketaZ$-relations to
$\kellp$-relations.
\\[1ex]
Nevertheless, when transferring relations, via $\pr_\Z\circ\kpi_\partial$,
from the elements $\kellp(\kc)$ to the integers $\ketaZ(\kc)$,
then by Proposition~\ref{prop-representingC} and 
Proposition~\ref{prop-representingD}
we obtain that $\kells(v),\kellt(\aed)\mapsto 1$, provided that $v\notin\kQuotD$.
This fits well with the facts $s_v,t_e\mapsto 1$ under $\kpi$ (cf. Section~\ref{sec:theAbmientSpace}).
In particular, the equations of Subsubsection~\ref{firstRec}
imply that
\[
k_i= n\cdot\ketaZ(\kc_i) - \ketaZ(n\kc_i) 
= n\cdot\big(\ketaZ(\kc_i) - \keta(\kc_i)\big)
\]
and
\[
\ketaZ(n\kc_1)+\ketaZ(n\kc_2)-\ketaZ(n\kc_0)=
\braket{ e,n\kc_1}=\braket{ -e,n\kc_2}.
\]
Note that $k_i\geq 1$ if and only if 
$\keta(\kc_i)=-\braket{ v^i, \kc_i}\notin\Z$, i.e.\
exactly when $v^i\notin\kQuotD$. In the case of $v^i\in\kQuotD$,
i.e.\ if the parameter $s_i$  is set to $0$, then
we proceed with $\kells(v^i)$ in the very same way.

%%%%%%%%%%%%%%%%%%%%%%%%%%%%%%%%%%%%%%%%%%%%%%%%%%%%%%%%%%%%%%%%%%%%%%%%%%%
\subsubsection{A property of \se~edges}
\label{charactShortEdges}
Here we will show that, whenever $\aed$ is a \se~edge, 
then there is a choice of $\kc_+$ (lifting $1\in\kQuot/e^\bot$)
such that the associated special elements $\kc_i$ lead to
$\ketaZ(\kc_1,\kc_2)=1$, i.e.\ we obtain
$\ketaZ(\kc_1)+\ketaZ(\kc_2)=\ketaZ(\kc_0)+1$.
The special choice of $\kc_+$ does only matter for $g=g_\aed\geq 2$.
\\[1ex]
Fix an element $p\in\sfrac{1}{g}\cdot \kQuotD$ of the affine line 
$\ko{\aed}$ containing
the edge $\aed$. That means that $\aed-p\subseteq\Q\cdot \vect{\aed}$,
inducing a lattice structure on $\ko{\aed}$
with $p$ becoming the origin.
Now, the striking point is that we may and will choose $p$ such that
$\braket{\kc_+,p}\in\Z$. Note that a different choice
of the lifting $\kc_+\in\kQuot$ of $1\in\kQuot/(\vect{\aed})^\bot$ 
at the beginning of the present Subsection~\ref{recoverShortE}
leads to a different $p$. 
\\[1ex]
Now we can write $v^i=p+v^i_0$ with $v^i_0\in \Q\cdot \vect{\aed}$ for $i=1,2$. 
Then, we obtain
$\keta(\kc_i)=-\braket{ \kc_i,\,p+v^i_0}$
and $\keta(\kc_1+\kc_2)=-\braket{ \kc_1+\kc_2,\,p+v^{1\wedge 2}_0}$.
Since $\braket{\kc_+,p}\in\Z$, we obtain that
$\braket{\kc_i,p}\in\Z$ as well, hence
\[
\renewcommand{\arraystretch}{1.3}
\begin{array}{rcl}
\ketaZ(\kc_1,\kc_2)&=&
\roundup{-\braket{ \kc_1,\,p+v^1_0}} +
\roundup{-\braket{ \kc_2,\,p+v^2_0}} -
\roundup{-\braket{ \kc_1+\kc_2,\,p+v^{1\wedge 2}_0}}
\\
&=&
\roundup{-\braket{ \kc_1,\,v^1_0}} +
\roundup{-\braket{ \kc_2,\,v^2_0}} -
\roundup{-\braket{ \kc_1+\kc_2,\,v^{1\wedge 2}_0}}
\\
&=&
\roundup{-\braket{ \kc_+,\,v^1_0}} +
\roundup{-\braket{ -\kc_+,\,v^2_0}}
\\
&=&
\roundup{\braket{ \kc_+,\,v^2_0}} -
\rounddown{\braket{ \kc_+,\,v^1_0}}.
\end{array}
\]
If we identify $(\Q\cdot \vect{\aed})\cap\kQuotD$ with $\Z$, thus also
$\Q\cdot \vect{\aed}$ with $\Q$, then $\kc_+$ becomes $1$ again, so
\[
\ketaZ(\kc_1,\kc_2)=\roundup{v^2_0} - \rounddown{v^1_0}.
\]
Hence, our claim $\ketaZ(\kc_1,\kc_2)=1$ is equivalent to the lack of interior
lattice points in $\aed$ (which we identified with $\aed-p$). 
This is clearly satisfied for \se~edges with $g_\aed=1$, but
we have to take a closer look at the case of $g_\aed\geq 2$.
\\[1ex]
Assume that $g_\aed\geq 2$. Then, the \se ness  still implies 
the lack of interior lattice points on $\aed-p$, provided that
$p\in(\sfrac{1}{g}\cdot\kQuotD)\cap\ko{\aed}$ is chosen as close as possible to $\aed$.
Thus, it remains to check that any desirable choice of $p$ can be realized
by a suitable choice of $\kc_+$. For this, we start by choosing coordinates
\begin{tikzcd}[cramped,sep=small]%Column = 1.5em,row sep = .5em, /tikz/baseline = -10em]
\kQuotD\arrow[r]{}{\sim}&\Z^d 
\end{tikzcd}
such that $~\vect{\aed} \cdot\Q$ becomes the first coordinate axis
$\Q\times \ku{0}^{d-1}=\{(\bullet,\ku{0})\}$. The dual picture is
\begin{tikzcd}[cramped, sep=small]%[column sep = 1.5em, /tikz/baseline = -10em] 
  \Z^d \arrow[r]{}{\sim}&\kQuot
\end{tikzcd}
with $(\vect{\aed})^\bot=0\times\Q^{d-1} = \{(0,\ku{\bullet})\}$.
So, the affine line $\ko{\aed}$ equals 
$\Q\times\{(\frac{k_2}{g},\ldots,\frac{k_d}{g})\} = \{(\bullet,\frac{k_2}{g},\ldots,\frac{k_d}{g})\}$
with $k_2,\ldots,k_d\in\ZZ$ and $\gcd(k_2,\ldots,k_d,g)=1$.
Thus, if $p=(\frac{p_1}{g},\frac{k_2}{g},\ldots,\frac{k_d}{g})$,
there are coefficients $\lambda_i\in\Z$ such that
\[
\textstyle
\frac{p_1}{g} \equiv \sum_{i=2}^d \lambda_i\cdot \frac{k_i}{g}
\hspace{1em}(\hspace{-0.9em}\mod\Z).
\]
Then $\kc_+:=(1,-\lambda_2,\ldots,-\lambda_d)\in\Z^d=\kQuot$ is a suitable 
initial choice allowing to take this special point $p$ as an origin 
afterwards.

%%%%%%%%%%%%%%%%%%%%%%%%%%%%%%%%%%%%%%%%%%%%%%%%%%%%%%%%%%%%%%%%%%%%%%%%%%%
\subsubsection{The second recursion formula}
\label{secondRec}
Here we assume that $\aed$ is a \se~edge and use
Subsection~\ref{charactShortEdges}. We will first show that
\[
\ketaZ(h\kc_1,\kc_1)= 1 - \ketaZ\big((h+1)\kc_1,\kc_2\big)
\]
for all $h\in\Z$ (both positive and negative).  This can be seen as follows:
\[
\begin{array}{rcl}
\ketaZ(h\kc_1)+\ketaZ(\kc_1)-\ketaZ\big((h+1)\cdot\kc_1\big)
&=&
\ketaZ(h\kc_1)+ \Big(\ketaZ(\kc_0) +1 -\ketaZ(\kc_2)\Big)
-\ketaZ\big((h+1)\cdot\kc_1\big)
\\
&=& 1 + \Big(\ketaZ(h\kc_1)+ \ketaZ(\kc_0)\Big) -\ketaZ(\kc_2)
-\ketaZ\big((h+1)\cdot\kc_1\big)
\\
&=& 1 + \ketaZ(h\kc_1+\kc_0) -\ketaZ(\kc_2)
-\ketaZ\big((h+1)\cdot\kc_1\big).
\end{array}
\]
Let us recall Case 2 from Subsection~\ref{firstRec}.
We had assumed $\;\ketaZ(h\kc_1,\kc_1)=1$,
occurs for  $(k_1=k)$ values of $h\in \{0,\dots, n-1\}$.
Using our new relation, this implies
$\ketaZ\big((h+1)\kc_1,\kc_2\big)=0$.
Hence, Proposition~\ref{prop-kellpLinear} implies that
$\kellp\big((h+1)\kc_1\big)+\kellp(\kc_2)=\kellp(h\kc_1+\kc_0)$.
So, if $k_1\geq 1$, i.e.\ if Case 2 occurs, then we can express
\[
\renewcommand{\arraystretch}{1.3}
\begin{array}{rcl}
\kells(v^1)&=&
\kellp(h\kc_1)+\kellp(\kc_1)-\kellp\big((h+1)\cdot\kc_1\big)\\
&=&
\kellp(h\kc_1)+\kellp(\kc_1)+\kellp(\kc_2) - \kellp(h\kc_1+\kc_0)\\
&=&
\kellp(h\kc_1)+\kellp(\kc_1)+\kellp(\kc_2) - \kellp(h\kc_1) - \kellp(\kc_0)\\
&=&
\kellp(\kc_1)+\kellp(\kc_2)- \kellp(\kc_0).
\end{array}
\]
In Subsubsection~\ref{kellpVSketaZ} we have seen that
$k_i\geq 1$ if and only if $v^i\notin\kQuotD$.
In particular, we obtain for these cases 
\[
\kells:=\kells(v^i)=\kellp(\kc_1)+\kellp(\kc_2)- \kellp(\kc_0).
\]
If both $v^1,v^2\notin\kQuotD$, then this already shows that
$\kells(v^1)=\kells(v^2)$. 
Anyway, it remains to compare $\kells$ with $\kellt(\aed)$.
At the end of Subsubsection~\ref{firstRec} we already got
\[
\braket{\vect{\aed},n\kc_1}\cdot\kellt(\aed)=
\kellp(n\kc_1)+\kellp(n\kc_2)-\kellp(n\kc_0),
\]
which is in the same spirit as the formula before.
Applying $\kpi$ as explained in Subsubsection~\ref{kellpVSketaZ}
and $\ketaZ(\kc_1,\kc_2)=1$ from Subsubsection~\ref{charactShortEdges},
this yields 
\[
\renewcommand{\arraystretch}{1.3}
\begin{array}{rcl}
\braket{ e,n\kc_1} &= &
\ketaZ(n\kc_1)+\ketaZ(n\kc_2)-\ketaZ(n\kc_0) \\
&= &
n\cdot\ketaZ(\kc_1) -k_1 + n\cdot\ketaZ(\kc_2) -k_2 -n\cdot\ketaZ(\kc_0)
\\ &= & n-(k_1+k_2).
\end{array}
\]
Adding up the two equations from Subsubsection~\ref{firstRec}:
$n\cdot\kellp(\kc_i)-\kellp(n\kc_i)=k_i\cdot\kells(v^i)$,
for $i=1,2$, we obtain
\[
n\cdot\big(\kellp(\kc_1)+\kellp(\kc_2)\big)
-\big(\kellp(n\kc_1)+\kellp(n\kc_2)\big)=(k_1+k_2)\cdot\kells.
\]
Note that this is even correct if one of the vertices $v^i$ belongs to
$\kQuotD$, i.e.\ if $k_i=0$.
Now, we replace 
$\kellp(\kc_1)+\kellp(\kc_2)$ by $\kellp(\kc_0)+\kells$
and $\kellp(n\kc_1)+\kellp(n\kc_2)$ by 
$\kellp(n\kc_0)+(n-k_1-k_2)\cdot\kellt(\aed)$. We obtain
\[
n\cdot\big(\kellp(\kc_0)+\kells\big)
-\big(\kellp(n\kc_0)+(n-k_1-k_2)\cdot\kellt(\aed)\big)=(k_1+k_2)\cdot\kells.
\]
Reordering, this yields
\[
(n-k_1-k_2)\cdot\kellt(\aed)=(n-k_1-k_2)\cdot\kells,
\]
and it remains to check that $n-(k_1+k_2)\neq 0$. 
However, since we have seen before that
\[
n-(k_1+k_2) = \braket{ e,n\kc_1} = \braket{ -e,n\kc_2},
\]
the vanishing of $n-(k_1+k_2)$ would imply
$
v^2-v^1=\braket{ e,1}=\braket{ e,\kc_+}=\braket{ e,\kc_1}=0,
$
which leads to a contradiction.

%%%%%%%%%%%%%%%%%%%%%%%%%%%%%%%%%%%%%%%%%%%%%%%%%%%%%%%%%%%%%%%%%%%%%%%%%%%
\subsection{Recovering the closing conditions along 2-faces}
\label{recoverClosingCond}
In Subsection~\ref{recoverT} we have looked at adjacent vertices
$v,v'\in\kP$. If their oriented connecting edge is $\vect{\aed}=v'-v$,
then we may choose sufficiently integral 
$\kc\in\normal(v,\kP)$, $\kc'\in\normal(v',\kP)$ in $\kQuot$ with
$\braket{ \vect{\aed}, \kc+\kc'}=0$, i.e.\ with
$\braket{ \vect{\aed}, \kc} = -\braket{ \vect{\aed}, \kc'}>0$,
and $\kc+\kc'\in \normal(\aed,\kP)= \normal(v,\kP) \cap\normal(v',\kP)$
leading to
$\kell(\kc,\kc')=\braket{ \vect{\aed}, \kc}\cdot \kellt(\aed)$
by Proposition~\ref{prop-representingD}. For the whole subsection we 
could keep the assumption of being ``sufficiently integral'' for all 
relevant elements from $\kQuot$ -- namely, we could entirely work within the
super integral sublattice $\kQuotSup\subseteq\kQuot$.
Instead, we replace $\kellp(\kc)$ by the following
stabilized version:
\[
\textstyle
\kellz(\kc):=\frac{1}{A}\cdot \kellp(A\cdot\kc)
\hspace{1em}\mbox{for } A\in\NN \mbox{ with }A\gg 0.
\]
In accordance to this, we replace
$\kell(\kc,\kc')=\kellp(\kc)+\kellp(\kc')-\kellp(\kc+\kc')$ by the stabilized
version too:
$\kellz(\kc,\kc'):=\kellz(\kc)+\kellz(\kc')-\kellz(\kc+\kc')$.
It extends the validity of the above formula for $\kellt(\aed)$ to non-integral
arguments.
\\[1ex]
Now we consider a compact 2-dimensional face $\kF\leq \kP$. Assume that its
vertices and oriented edges are $v^i\in\kQuotRD$ 
and $\vect{\aed_i}=v^{i+1}-v^i$ ($i\in\Z/n\Z$), respectively.
Then, the cones
$\normal(\kF,\kP)\subseteq\normal(\aed_i,\kP)\subseteq\normal(v^i,\kP)$
are part of the inner normal fan $\normal(\kP)$. Projecting them down to
the 2-dimensional vector space $\kQuotR/\kF^\bot=:\kF^*$
(dual to the vector space $\kF-\kF$ accompanying the affine space spanned by
$\kF$) yields a 2-dimensional complete fan $\ko{\normal}_{\kF}$
within $\kF^*$. We denote the image cones by 
\[
0\;=\;\ko{\normal}(\kF,\kP)\;\subseteq\;\ko{\normal}(\aed_i,\kP)
\;\subseteq\;\ko{\normal}(v^i,\kP)\;\subseteq\;\kF^*.
\]
The cones $\ko{\normal}(\aed_i,\kP)$ form the rays, and their linear
hull is $(\vect{\aed_i})^\bot/F^\bot$. The two-dimensional cones
$\ko{\normal}(v^i,\kP)$ are spanned by the rays 
$\ko{\normal}(d^{i-1},\kP)$ and $\ko{\normal}(\vect{\aed_i},\kP)$.
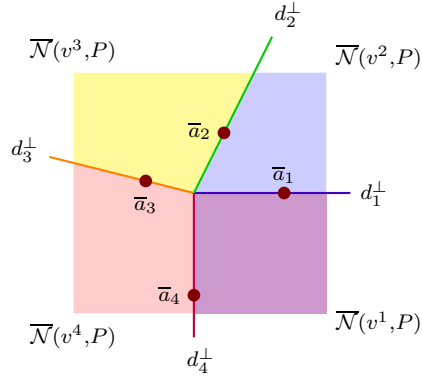
\begin{figure}[h!]
\begin{tikzpicture}[scale=.8]
    \newcommand{\xa}{2} %width
    \newcommand{\ya}{5} %height
    \node (o) at (0,0) {};
    \node (a) at (\xa.2,0) {};
    \node (A) at (\xa.6,0) {};
    \node (ab) at (\xa.2,\xa) {}; 
    \node (b) at (\xa/2,\xa) {};
    \node (B) at (\xa.6/2,\xa.6) {};
    \node (bc) at (-\xa,\xa) {};
    \node (c) at (-\xa,\xa/4) {};
    \node (C) at (-\xa.4,\xa.4/4) {};
    \node (cd) at (-\xa,-\xa){};
    \node (d) at (0,-\xa) {};
    \node (D) at (0,-\xa.4) {};
    \node (da) at (\xa.2,-\xa) {};
    %%%%%%%%%%%%%%
    %%%%% klaus:
    %%%%%%%%%%%%%
\fill[color=blue,%
      fill opacity=0.2] (o.center) -- (a.center) -- (ab.center)  -- (b.center)  --cycle;
    \fill[fill = yellow ,fill opacity=0.4] (o.center) -- (b.center) -- (bc.center)  -- (c.center)  --cycle;
    \fill[fill= red,fill opacity=0.2] (o.center) -- (c.center) -- (cd.center)  -- (d.center)  --cycle;
    \fill[fill=violet,fill opacity=0.4] (o.center) -- (d.center) -- (da.center)  -- (a.center)  --cycle;

    \draw[thick, color=blue!50!violet] (0,0) -- (A.center);
    \draw[thick, color=green!80!black] (0,0) -- (B.center);
    \draw[thick, color=red!50!yellow] (0,0) -- (C.center);
    \draw[thick, color=red!50!violet] (0,0) -- (D.center);

    \fill[thick, color=red!50!black] (1.5,0) circle (3pt);
    \fill[thick, color=red!50!black] (0.5,1) circle (3pt);
    \fill[thick, color=red!50!black] (-0.8,0.2) circle (3pt);
    \fill[thick, color=red!50!black] (0,-1.7) circle (3pt);
    
   \draw[thick,  color=black] (1.5,0.3) node{$\ks \ko{a}_1$};
    \draw[thick,  color=black] (0.1,1.03) node{$\ks \ko{a}_2$};
    \draw[thick,  color=black] (-0.8,-0.2) node{$\ks \ko{a}_3$};
    \draw[thick,  color=black] (-0.4,-1.7) node{$\ks \ko{a}_4$};  
    
    \node[anchor=west] at (A.center)  {$\ks \vect{\aed_{1}}^\bot$};
    \node[anchor=south west] at (B.center)  {\hspace{-.1cm}$\ks \vect{\aed_2}^\bot$};
    \node[anchor=east]  at (C.center) {\raisebox{.3cm}{$\ks \vect{\aed_{3}}^\bot$}};
    \node[anchor=north] at (D.center)   {~~$\ks \vect{\aed_{4}}^\bot$};
    \node[anchor=  west] at (ab.center) {\raisebox{.5cm}{$\ks \ko{\normal}(v^2,\kP)$}};
    \node[anchor = south] at (bc.center) {\raisebox{0cm}{$\ks \ko{\normal}(v^3,\kP)$}};
    \node[anchor = north] at (cd.center) {$\ks \ko{\normal}(v^4,\kP)$};
    \node[anchor = west] at (da.center) {\raisebox{-.5cm}{$\ks \ko{\normal}(v^1,\kP)$}}; 
  \end{tikzpicture}
  \label{fig:7
  }
  \caption{The normal fan of a 2-face with four edges.}
\end{figure}

\begin{proposition}
\label{prop-ClosingCond}
Corresponding to~(\ref{eq:dualEquationsT(Q)1}), in  $\kT\otimes_\Z\kQuotRD$ we have the equation
$\sum_{i\in\Z/n\Z}\kellt(\vect{\aed_i})\otimes \vect{\aed_i}=0$.
\end{proposition}

\begin{proof}
We choose elements $a_i\in \innt\normal(\aed_i,\kP)$ mapping
to points $\ko{a}_i$ on the rays $\ko{\normal}(\aed_i,\kP)$.
Since $\vect{\aed_i}\subseteq \kF-\kF$, we know that
$\braket{ \vect{\aed_i},\kF^\bot} =0$, i.e.\ we may write
$\braket{ \vect{\aed_i}, a_j} = \braket{ \vect{\aed_i}, \ko{a}_j}$.
This yields 
\[
\braket{ \vect{\aed_i}, \ko{a}_{i-1}} >0,\hspace{1em}
\braket{ \vect{\aed_i}, \ko{a}_{i}} =0,\hspace{1em}
\braket{ \vect{\aed_i}, \ko{a}_{i+1}} <0.
\]
For a fixed $i\in\Z/n\Z$ and some $A\gg 0$ we will use 
\[
\textstyle
\kc:=\frac{1}{\braket{ \vect{\aed_i},\, \ko{a}_{i-1}}} \cdot a_{i-1} + A\cdot a_i
\hspace{0.8em}\mbox{and}\hspace{0.8em}
\kc':=\frac{-1}{\braket{ \vect{\aed_i},\, \ko{a}_{i+1}}} \cdot a_{i+1} + A\cdot a_i
\]
for which we have $\braket{ \vect{\aed_i}, \kc'} =-\braket{ \vect{\aed_i}, \kc}$.
Hence, we obtain that $\kc+\kc'\in \vect{\aed_{i}}^\bot$ and, if $A$ is large enough,
even $\kc+\kc'\in \normal(\aed_i,\kP)$. Thus, with
$v:=v^i$, $v':=v^{i+1}$, and $\vect{\aed}=\vect{\aed_i}$ we are exactly in the situation of the
begin of this subsection. That is,
\[
\kellt(\aed_i)=
\kellz(\kc)+\kellz(\kc')-\kellz(\kc+\kc').
\]
This equation remains valid if we alter $\kellz(\kc)$ by a 
function that is linear in $\kc$. Thus, we may and will assume that
$\kellz$ vanishes on $\normal(\kF,\kP)$. This implies that 
$\kellz$ descends to a well-defined function
$\kellz:\kF^*\to (\kS-\kS)\otimes_\Z\Q$. It is linear
on the cones of $\ko{\normal}_{\kF}$, and we still have
\[
\textstyle
\renewcommand{\arraystretch}{1.3}
\begin{array}{rcl}
\kellt(\vect{\aed_i}) &=&
\kellz(\ko{\kc})+\kellz(\ko{\kc}')-\kellz(\ko{\kc}+\ko{\kc}')
\\
&=&
\begin{array}[t]{@{}l}
\kellz\big(\frac{1}{\braket{ \vect{\aed_i},\, \ko{a}_{i-1}}} \cdot \ko{a}_{i-1} 
+ A\cdot \ko{a}_i\big)
+ \kellz\big(\frac{-1}{\braket{ \vect{\aed_i},\, \ko{a}_{i+1}}} \cdot \ko{a}_{i+1} 
+ A\cdot \ko{a}_i\big)
\\
\hspace{10em}
-\kellz\big(\frac{1}{\braket{ \vect{\aed_i},\, \ko{a}_{i-1}}} \cdot \ko{a}_{i-1} 
- \frac{1}{\braket{ \vect{\aed_i},\, \ko{a}_{i+1}}} \cdot \ko{a}_{i+1}
+ 2A\cdot \ko{a}_i\big)
\end{array}
\\
&=&
\frac{1}{\braket{ \vect{\aed_i},\, \ko{a}_{i-1}}} \,\kellz(\ko{a}_{i-1})
+A\,\kellz(\ko{a}_{i})
-\frac{1}{\braket{ \vect{\aed_i},\, \ko{a}_{i+1}}} \, \kellz(\ko{a}_{i+1})
+A\,\kellz(\ko{a}_{i})
-(\beta_i+2A)\,\kellz(\ko{a}_{i})
\\
&=&
\frac{1}{\braket{ \vect{\aed_i},\, \ko{a}_{i-1}}} \,\kellz(\ko{a}_{i-1})
-\frac{1}{\braket{ \vect{\aed_i},\, \ko{a}_{i+1}}} \, \kellz(\ko{a}_{i+1})
-\beta_i\,\kellz(\ko{a}_{i})
\end{array}
\]
where $\beta_i\in\RR$ is defined by the equality
\[
\textstyle
\frac{1}{\braket{ \vect{\aed_i},\, \ko{a}_{i-1}}} \cdot \ko{a}_{i-1} 
- \frac{1}{\braket{ \vect{\aed_i},\, \ko{a}_{i+1}}} \cdot \ko{a}_{i+1}
-\beta_i\cdot \ko{a}_{i} =0.
\]
Now, we consider 
\[
\textstyle
\renewcommand{\arraystretch}{1.3}
\begin{array}{rcl}
\sum_{i}\kellt(\aed_i)\otimes \vect{\aed_i}
&=&
\sum_{i\in\Z/n\Z} \big(\frac{1}{\braket{ \vect{\aed_i},\, \ko{a}_{i-1}}}
\,\kellz(\ko{a}_{i-1})
-\frac{1}{\braket{ \vect{\aed_i},\, \ko{a}_{i+1}}} \, \kellz(\ko{a}_{i+1})
-\beta_i\,\kellz(\ko{a}_{i})\big)
\otimes \vect{\aed_i}
\\
&=&
\sum_{i\in\Z/n\Z} \;\kellz(\ko{a}_{i}) \otimes \big(
\frac{1}{\braket{ \vect{\aed}_{i+1},\, \ko{a}_{i}}}\,\vect{\aed}_{i+1}
- \frac{1}{\braket{ \vect{\aed}_{i-1},\, \ko{a}_{i}}}\vect{\aed}_{i-1}
-\beta_i\,\vect{\aed_i} \big)
\end{array}
\]
and check that all of the second factors vanish. This will be done in the
following quick and dirty way via choosing coordinates, i.e.\ 
fixing some isomorphism
 \begin{tikzcd}[column sep = 1.5em]%, /tikz/baseline = -10em] 
   (\kF-\kF) \arrow[r]{}{\sim}&\RR^2
 \end{tikzcd}
 which determines a dual isomorphism
 \begin{tikzcd}[column sep = 1.5em]%, /tikz/baseline = -10em] 
   \RR^2 \arrow[r]{}{\sim}&\kF^*
 \end{tikzcd}
 too.
While the vectors $\vect{\aed_i}\in(\kF-\kF)=\RR^2$ are given by the choice of
$\kF\leq\kP$, we have some freedom in choosing the 
$\ko{a}_i\in\kF^*=\RR^2$ -- one has just to ensure that $\ko{a}_i\bot \vect{\aed_i}$
and that they have the right orientation. This can be obtained by
\[
\ko{a}_i:=
\left(\begin{array}{rr} 0 & -1 \\ 1 & 0\end{array}\right) \cdot \vect{\aed_i}.
\] 
Doing so, the $\ko{a}_i$ satisfy the same linear relations as the $\vect{\aed_i}$
do. i.e.\ we obtain that
\[
\textstyle
\frac{1}{\braket{ \vect{\aed_i},\, \ko{a}_{i-1}}} \cdot \vect{\aed}_{i-1} 
- \frac{1}{\braket{ \vect{\aed_i},\, \ko{a}_{i+1}}} \cdot \vect{\aed}_{i+1}
-\beta_i\cdot \vect{\aed_{i}} =0.
\]
Thus, the claim follows from the equalities
$\braket{\vect{\aed}_{i-1},\, \ko{a}_{i}}=
-\braket{ \vect{\aed_i},\, \ko{a}_{i-1}}$ for all indices
$i\in\Z/n\Z$ and our special choices of $\ko{a}_i\in\RR^2$.
\end{proof}

%%%%%%%%%%%%%%%%%%%%%%%%%%%%%%%%%%%%%%%%%%%%%%%%%%%%%%%%%%%%%%%%%%%%%%%%%%%
\subsection{Concluding the proof of the existence of $\kell$}
\label{concludeProofExKell}
One might think that we are already done with the construction of 
the map $\kell$ -- but it requires the following, seemingly paranoid 
conclusion of the proof.
What do we have so far? First, we have well-defined elements
$\kellp(\kc)\in\Rand{\kT}{(\kS)}$ which have to become the images of
$[\kc,\ktetaZ(\kc)]\in\Rand{\ktT}{(\ktS)}$.
Second, we have constructed the following elements:
\begin{enumerate}[label=(\roman*)]
\item
If $v\in\kP$ is a vertex, then there is a well-defined
$\kells(v)\in\kT$  planned to become the image $\kell(s_v)$,
cf.\ Proposition~\ref{prop-representingC} in Subsection~\ref{recoverS}.
\item
For each compact edge $\aed\leq\kP$ there is a well-defined 
$\kellt(\aed)\in\Q_{>0}\cdot\kT$ planned to become the image $\kell(t_\aed)$,
cf.\ Proposition~\ref{prop-representingD} in Subsection~\ref{recoverT}.
\end{enumerate}
In the Subsections~\ref{recoverNoLatticePt},~\ref{recoverShortE},~and~\ref{recoverClosingCond}
we have shown that the new elements
$\kells(v)$ and $\kellt(\aed)$ satisfy the same linear relations as the
original elements $s_v$ and $t_\aed$. This gives rise to a
well-defined linear map
\[
\varphi:\; \kMTD(\kP)=(\ktT-\ktT)\otimes_\Z\Q \;\to\; (\kT-\kT)\otimes_\Z\Q
\]
with $\varphi(s_v)=\kells(v)$ and $\varphi(t_\aed)=\kellt(\aed)$.

\begin{lemma}
\label{lem-varphEll}
For $\kc_1,\kc_2\in\tail(\kP)\dual\cap\kQuot$ we have
$\varphi\big(\ktetaZ(\kc_1,\kc_2)\big)=\kell(\kc_1,\kc_2)$.
\end{lemma}

\begin{proof}
{\em Step 1. }
First, by Proposition~\ref{prop-representingC} and the preceding remarks in
Subsection~\ref{recoverS}, the claim of the lemma follows for those
pairs $(\kc_1,\kc_2)$ where $\kc_1,\kc_2$ are contained in a common
normal cone $\normal(v,\kP)$ for some vertex $v\in\kP$.
Second, by Proposition~\ref{prop-representingD} and the preceding remarks in
Subsection~\ref{recoverT} the claim of the lemma does also follow
for super integral $\kc_1,\kc_2\in\kQuotSup$ being contained in two
adjacent normal cones $\normal(v^1,\kP)$ and $\normal(v^2,\kP)$, respectively.
(That is, $v^1$ and $v^2$ have to be connected by an edge, 
and one has to suppose that $\kc_1+\kc_2$ belongs to the 
union of these normal cones).
\\[1.0ex]
{\em Step 2. }
If $\kc\in\kQuot$ and $n\in\NN$, then we know from
Subsection~\ref{recoverS} that
$n\cdot\ktetaZ(\kc)-\ktetaZ(n\kc)= \big(n\,\ketaZ(\kc)-\ketaZ(n\kc)\big)
\cdot s_{v(\kc)}$ and 
$n\cdot\kellp(\kc)-\kellp(n\kc)=
\big(n\,\ketaZ(\kc)-\ketaZ(n\kc)\big)\cdot\kells(v(\kc))$. Hence,
for $\kc_1,\kc_2\in\kQuot$ we obtain that
$n\cdot \ktetaZ(\kc_1,\kc_2)-\ktetaZ(n\kc_1,n\kc_2)$ maps, via $\varphi$,
to $n\cdot \kell(\kc_1,\kc_2)-\kell(n\kc_1,n\kc_2)$. Consequently, the
fact that $\varphi\big(\ktetaZ(\kc_1,\kc_2)\big)=\kell(\kc_1,\kc_2)$
is equivalent to 
$\varphi\big(\ktetaZ(n\kc_1,n\kc_2)\big)=\kell(n\kc_1,n\kc_2)$.
This means that it remains to show the claim for super integral (but not necessarily from adjacent cones)
$\kc_1,\kc_2\in\kQuotSup$.
\\[1.0ex]
{\em Step 3. }
We are going to use the inhomogeneous description of group cohomology,
cf.\ \cite[VII.3]{LocalFields}.
Both $\ktetaZ(\kbb,\kbb)$ and $\kell(\kbb,\kbb)$ are 2-coboundaries.
Hence, the map $b:\kQuotSup\times\kQuotSup\to(\kT-\kT)\cdot\Q$
defined as
\[
b(\kbb,\kbb):=\varphi\big(\ktetaZ(\kbb,\kbb)\big)-\kell(\kbb,\kbb)
\]
is still a 2-cocyle for $\gH^\kbb\!\big(\kQuotSup,\,(\kT-\kT)\cdot\Q\big)$.
Since $(\kT-\kT)\cdot\Q$ is a divisible group, hence an injective 
$\Z$-module, we know that $\gH^2\!\big(\kQuotSup,\,(\kT-\kT)\cdot\Q\big)=0$.
Thus, $b$ is a 2-coboundary, i.e.\ there is a map
\[
\kbp:\kQuotSup\to(\kT-\kT)\cdot\Q
\hspace{1em}\mbox{with}\hspace{1em}
b(\kc_1,\kc_2)=\kbp(\kc_1)+\kbp(\kc_2)-\kbp(\kc_1+\kc_2).
\]
From Step~1 we know that $\kbp$ is linear on the full-dimensional
normal cones $\normal(v,\kP)$ or even on the union of
adjacent ones $\normal(v^1,\kP)$ and $\normal(v^2,\kP)$ --
provided that $\kc_1+\kc_2$ belongs to this union.
But this means that $\kbp$ is globally linear, i.e.\ $b=0$.
\end{proof}

So $\kell_\kT:\ktT\too\kT$ is a well-defined linear map, and we conclude the existence part of Theorem~\ref{th-initialObject} by Remark~\ref{rem:ellTisEnough}.

%%%%%%%%%%%%%%%%%%%%%%%%%%%%%%%%%%%%%%%%%%%%%%%%%%%%%%%%%%%%%%%%%%%%%%%%%%%
%%%%%%%%%%%%%%%%%%%%%%%%%%%%%%%%%%%%%%%%%%%%%%%%%%%%%%%%%%%%%%%%%%%%%%%%%%%
%%%%%%%%
\section{Minkowski decompositions revisited}
\label{sec:MinkowskiDec}
\subsection{Review of the case of lattice polytopes with primitive edges}
\label{ssec:lattPrimEdge}
In \cite{versalG} 
we had treated a special case of the scenario described in Subsection~\ref{buildSg}.
There it was assumed that $\kP$ is a lattice polytope with primitive edges, 
i.e.\ the edges did not contain any lattice points other than the vertices. 
In algebro-geometric terms this means that
$X=\toric(\cone(\kP))=\Spec\CC[\kS]$ is Gorenstein and it is smooth in
codimension two. Let us summarize the main results from \cite{versalG} for this special case.
\begin{enumerate}[label=(\roman*)]
\item\label{item:rlpe1}
If $\kP=\kP_0+\ldots+\kP_\minks$ is Minkowski decomposition, 
then this corresponds
to a decomposition $\one=\xi_0+\ldots+\xi_\minks$ within the cone
$\kMC(\kP)$. The summands $\kP_\nu$ are lattice polytopes if and only if
the corresponding $\xi_\nu$ belong to the lattice 
$\kMVZ(\kP)$ within the vector space
$\kMV(\kP):=\kMC(\kP)-\kMC(\kP)$. This lattice is defined by the integrality
of all coordinates $t_{ij}(\xi)$.
\vspace{1ex}
\item\label{item:rlpe2}
Since the coordinates $t_{ij}$ are supposed to be non-negative on 
$\kMC(\kP)$, they become elements of the dual cone $t_{ij}\in\kMC(\kP)\dual$.
The subsemigroup generated by these elements provides the base of the
initial object from  Theorem~\ref{th-initialObject}.
That is, the present special
case of a discrete setup shows tight parallels to the cone setup
 displayed in Proposition~\ref{prop-versalCP}.
\vspace{1ex}
\item\label{item:rlpe3}
Translated to the framework of algebraic geometry, lattice decompositions
of $\kP$ as in (i) correspond to components of the versal
deformation of $X=\toric(\cone(\kP))$ in degree $-\kR$.
The complexification of the vector space $\kMV(\kP)/\one\cdot \RR$ 
equals the space of
infinitesimal deformations $T^1_X(-\kR)$ of $X$ in degree $-R$.
Finally, the subsemigroup $\spann{\NN}{ t_{ij}}\subseteq\kMC(\kP)\dual$ 
of (ii) encodes the versal deformation itself, 
which is a much finer information than 
just its linear ambient space $T^1_X(-\kR)$.
\end{enumerate}

%%%%%%%%%%%%%%%
%%%%%%%%%%%%%%%
\begin{remark}
\label{rem-smoothCodimTwo}
In the Gorenstein case, i.e.\ when $\kP$ was a lattice polytope, then the
smoothness of $X$ in codimension two could be easily expressed
by the primitivity of its lattice edges. In \cite{AK13} we already got
rid of the Gorenstein assumption, but we heavily depended on the assumption
of smoothness in codimension two. In the non-Gorenstein case, this condition
can be still be expressed in the combinatorial language. It says that,
for each bounded edge $[v^i,v^j]$ of $\kP\subseteq\kQuotRD$,
the polyhedral cone generated
from $(v^i,1),(v^j,1)\in\kQuotRD\oplus\RR_{\geqslant 0}$ is
$\Z$-linearly (!) isomorphic to the ordinary upper orthant $\RR^2_{\geqslant 0}$.
\end{remark}

%%%%%%%%%%%%%%%%%%%%%%%%%%%%%%%%%%%%%%%%%%%%%%%%%%%%%%%%%%%%%%%%%%%%%%%%%%%
Now, returning to the general discrete setup established in Subsection~(\ref{buildSg})
and taking the points \ref{item:rlpe1}-\ref{item:rlpe3} above as a guideline,
we  no longer assume that $\kP$ is bounded, i.e.\ $\kP$ may have a non-trivial
tail cone. More important, however, is that we do not require $\kP$ to be a lattice 
polyhedron anymore, nor do we ask for any further restrictions (on the edges or anything else).

%%%%%%%%%%%%%%%
%%%%%%%%%%%%%%%
\begin{example}
\label{ex-discreteB 2}
Let us return to Example~\ref{ex-discreteB}.
In this case we have $\kMV(\kP)=\RR$, which shows that
(\ref{ssec:lattPrimEdge})\,\ref{item:rlpe3} is no longer valid in this case.
Moreover, neither the lattice $\kMVZ(\kP)$, nor lattice decompositions
$\kP=\kP_0+\ldots+\kP_\minks$ from Subsection~\ref{ssec:lattPrimEdge}\,\ref{item:rlpe1} make any sense here. So these notions need to be replaced: $\kMVZ(\kP)$ by $\kMTPZ$ and lattice decompositions by lattice \emph{friendly} decompositions (cf. Section~\ref{latticeFriends}).
\end{example}

%%%%%%%%%%%%%%%%%%%%%%%%%%%%%%%%%%%%%%%%%%%%%%%%%%%%%%%%%%%%%%%%%%%%%%%%%%%
%%%%%%%%%%%%%%%
%%%%%%%%%%%%%%%%%%%%%%%%%%%%%%%%%%%%%%%%%%%%%%%%%%%%%%%%%%%%%%%%%%%%%%%%%%%
%%%%%%%%%%%%%%%%%%%%%%%%%%%%%%%%%%%%%%%%%%%%%%%%%%%%%%%%%%%%%%%%%%%%%%%%%%%
\subsection{The universal Minkowski summand}
\label{univMinkSum}

We start by fixing a \kbox{reference vertex $\vast\in\Vrtx(\kP)$}, and recall from
Subsection~\ref{sec:theAbmientSpace} that
\[
  \kMTP(\kP):=\kMT(\kP)\cap(\RR^\kPr_{\geqslant 0}\oplus\RR_{\geqslant 0}^\kPm).
\]
For $\,\xi =(t,s)\in \kMTP(\kP)$
we construct $\kP_\xi$ by defining a map  
$\,\psi_{\vast}(\xi): \Vrtx(\kP)\longrightarrow\kQuotRD$.
For the reference vertex we set
\[
\psi_{\vast}(\xi,\vast):= s_\vast(\xi)\cdot \vast.
\]
Note that this definition implies that
$\psi_{\vast}(\xi,\vast)=0$ if $\vast\in \kQuotD$.
For every other $v\in\Vrtx(\kP)$, choose a path
$\vast=v^0,v^1,\ldots,v^k=v$ along the compact edges of $\kP$.
Denoting $\vect{\aed_i}:=v^i-v^{i-1}$, we define
\[ 
\textstyle
\psi_{\vast}(\xi,v) = \psi_{\vast}(\xi,\vast) + \sum_{i=1}^k
t_i(\xi)\cdot \vect{\aed_i}.
\]
It is a direct consequence of the closing conditions in the definition of
$\kMV(\kP)$ and hence of $\kMT(\kP)$ in
Subsection~\ref{sec:theAbmientSpace} that $\psi_{\vast}(\xi,v)$ does not depend
on the special choice of the path, cf.\ (\ref{indep-edgePath}).
Now, we obtain the Minkowski summand associated to $\,\xi\in\kMTP(\kP)$ as
\[
\kP_\xi:=
\psi_{\vast}(\xi,\kP) :=
\conv\{\psi_{\vast}(\xi,v) : v\in\Vrtx(\kP)\} + \tail(\kP).
\]
Note that one can avoid the usage of the $s$-variables
when there exists a lattice vertex in $\kP$ to be chosen as $\vast$. 
However,  even then the $s$-coordinates will
play an important role in (\ref{indep-refVertex}).
Similarly to Subsection~\ref{spaceMinkSummands} we proceed with the following.

%%%%%%%%%%%%%%%
%%%%%%%%%%%%%%%
\begin{definition}\label{d:tautologicalCone}
We define the \kbox{universal Minkowski summand} or the \kbox{tautological
cone} as
\[
\,\kMTtP^{\vast}(\kP) := 
\{ (\xi,w) : \xi\in \kMTP(\kP) ,\;w\in\psi_{\vast}(\xi,\kP)\}
\subseteq \kMTP(\kP)\times\kQuotRD.
\]
It comes with the natural projection 
$\,\kpr_+:\kMTtP^{\vast}(\kP)\to \kMTP(\kP)$ onto the first factor.
\end{definition}

Now we check that, up to consistent lattice-translations in $\kQuotD$,  
the previous definitions are independent of all choices,
i.e.\ that we may indeed call our Minkowski summands $\kP_\xi$
and denote the tautological cone by $\kMTtP(\kP)$.
In more detail, if $\xi \in \kMTP(\kP)\cap \kMTZ(\kP)$, then
the following 
(\ref{indep-edgePath}) and (\ref{indep-refVertex})
will imply
that, for all $\vast,\vastP$,
the polyhedron $\psi_{\vast}(\xi,\kP)$ is obtained from 
$\psi_{\vastP}(\xi,\kP)$ via a lattice isomorphism linearly depending
on $\xi$.

%%%%%%%%%%%%%%%%%%%%%%%%%%%%%%%%%%%%%%%%
\subsubsection{Independence on the path along the edges}
\label{indep-edgePath}
This is a direct consequence of the closure conditions along the compact
2-faces $F\leq\kP$ which define $\kMV(\kP)$
or $\kMT(\kP)$, cf.\ Definition~\ref{def:TP}. 
Here we are literally in the same situation as in \cite{versalG}.

%%%%%%%%%%%%%%%%%%%%%%%%%%%%%%%%%%%%%%%%
\subsubsection{Independence on the reference vertex}
\label{indep-refVertex}
Assume that
$\vast$ and $\vastP$ are two different vertices of $\kP$
which are connected by an (oriented) edge $\vect{\aed}=\vastP-\vast$.
Recall from Definition~\ref{def:latticeInT}  that this situation gives rise to
an element $\kL_\aed\in \kMTZ(\kP)^*\otimes_{\ZZ}\kQuotD$.
Again, we have to compare the Minkowski summands with respect to the same
vertex, say $\vast$:
\[
\psi_{\vastP}(\xi,\vast)-\psi_{\vast}(\xi,\vast) \;=\;
s_{\vastP}(\xi)\cdot\vastP \;-\;
t_d(\xi)\cdot (\vastP-\vast)
\;-\; s_{\vast}(\xi)\cdot\vast
\]
hence
\begin{eqnarray*}
\psi_{\vastP}(\vast)-\psi_{\vast}(\vast) &=&
  s_{\vastP}\otimes\vastP - t_\aed\otimes(\vastP-\vast) -
  s_{\vast}\otimes \vast\\
&=&
(t_\aed-s_{\vast})\otimes\vast - (t_\aed-s_{\vastP})\otimes \vastP
=
\kL_d\in \kMTZD(\kP)\otimes_{\ZZ}\kQuotD.
\end{eqnarray*}
That is the two tautological cones differ  via translation by an integral, linear section of $\kpr$.

%%%%%%%%%%%%%%%
%%%%%%%%%%%%%%%
\begin{convention}
\label{conv-refPtsA}
Unless $\kP$ has at least one lattice vertex, we cannot assume that the reference vertex $\vast$ is 0.
Nevertheless, we will write $\kMTtP(\kP)$ for $\kMTtP^{\vast }(\kP)$
and keep in mind the dependence on $\vast$ via\footnotemark~ the shift by 
the $\kpr$-section 
$\,\psi_{\vastP}-\psi_{\vast}=\kL_{\vast\vastP}$
and via
$\,\kMTtP^{\vastP 0}(\kP)=\kMTtP^{\vast 0}(\kP)+ \kL_{\vast\vastP}$.
\footnotetext{This is not true literally. It maps $(\xi,w)\mapsto (\xi, w+\kL_{\vast\vastP})$.}
\end{convention}

%%%%%%%%%%%%%%%
%%%%%%%%%%%%%%%
\begin{theorem}
\label{thm-tautCone}
The universal Minkowski summand $\kMTtP(\kP)$ is a convex, polyhedral cone.
\end{theorem}

%%%%%%%%%%%%%%%
%%%%%%%%%%%%%%%
\begin{proof}
This follows because $\xi(\vast):=\psi(\xi,\vast)$ and hence 
$v_\xi=\xi(v):=\psi(\xi,v)$ depend,
for every vertex $v\in\kP$, linearly on $\xi$.
\end{proof}

%%%%%%%%%%%%%%%%%%%%%%%%%%%%%%%%%%%%%%%%%%%%%%%%%%%%%%%%%%%%%%%%%%%%%%%%%%%%%
%%%%%%%%%%%%%%%%%%%%%%%%%%%%%%%%%%%%%%%%%%%%%%%%%%%%%%%%%%%%%%%%%%%%%%%%%%%%%
%%%%%%%%%%%%%
\subsection{Lattice friendly Minkowski decompositions}
\label{latticeFriends}
%%%%%%%%%%%%%
%%%%%%%%%%%%%%%%%%%%%%%%%%%%%%%%%%%%%%%%%%%%%%%%%%%%%%%%%%%%%%%%%%%%%%%%%%%%%
%%%%%%%%%%%%%%%%%%%%%%%%%%%%%%%%%%%%%%%%%%%%%%%%%%%%%%%%%%%%%%%%%%%%%%%%%%%%%

%%%%%%%%%%%%%%%%%%%%%%%%%%%%%%%%%%%%%%%%%%%%%%%%%%%%%%%%%%%%%%%%%%%%%%%%%%%
In the situation of Subsection~\ref{ssec:lattPrimEdge}\,\ref{item:rlpe1}, 
each decomposition of a lattice polytope $\kP$ into a sum of lattice polytopes
$\kP=\kP_0+\ldots + \kP_\minks$ was encoding a component of the versal deformation.
Independently on this interpretation, the lattice condition for the summands
$\kP_i$ was a discrete requirement reducing the number of 
admissible decompositions drastically; in particular, it becomes finite.
In the general setup, however, 
i.e.\ when $\kP$ is no longer a lattice polytope, 
then lattice decompositions cannot exist at all. 
Inspired by \cite[(3.2)]{ka-flip}, we nevertheless save this concept by
defining the following weaker version.

%%%%%%%%%%%%%%%
%%%%%%%%%%%%%%%
\begin{definition}
\label{def-latticeFriends}
A Minkowski decomposition $\kP=\kP_0+\ldots + \kP_\minks$ is called 
{\em lattice friendly} if all summands share the same tail cone,
and if, for every $\kc\in\tail(\kP)\dual\cap\kQuot$,
there is an index $\mu=\mu(\kc)$ such that all 
$\face(\kP_i,\kc)\leq\kP_i$ with $i\in\{0,\ldots,\minks\}\setminus\{\mu\}$
contain lattice points.
\end{definition}

Recall that $\face(\kP_i,\kc):=\{a\in\kP_i\kst \braket{ \kc,a} =
\min\braket{ \kc,\kP_i}\}$ is the face of $\kP_i$ where $\kc$ attains
its minimum. It suffices to check the condition of the previous definition for
generic $\kc\in\tail(\kP)\dual\cap\kQuot$, i.e.\ for those
where $\face(\kP,\kc)$ is a vertex of $\kP$. Since
\[
\face(\kP,\kc)= \face(\kP_0,\kc)+\ldots+\face(\kP_\minks,\kc),
\]
this implies that $\face(\kP_i,\kc)\leq\kP_i$ are vertices, too.
Hence, in this generic case, the above definition asks for 
\[
\face(\kP_0,\kc),\ldots,\face(\kP_\minks,\kc) \in\kQuotD
\] 
to be lattice vertices -- with at most one exception, 
namely for $\face(\kP_\mu,\kc)$. 
This means that 
\begin{enumerate}[label=(\roman*)]
\item
any failure $\face(\kP,\kc)\notin\kQuotD$ stems from 
one single summand $\face(\kP_\mu,\kc)\notin\kQuotD$ where $\mu$ depends on
the choice of the generic $\kc$, i.e.\ on the choice of the vertex 
$\face(\kP,\kc)$, and
\item
if $\face(\kP,\kc)\in\kQuotD$,
then all summands $\face(\kP_i,\kc)$ are lattice vertices, 
without any exception.
\end{enumerate}
In particular, if $\kP$ were a lattice polyhedron
as in Subsection~\ref{ssec:lattPrimEdge}\,\ref{item:rlpe1}, 
then  being lattice friendly just means being a lattice
decomposition, i.e.\  all
summands $\kP_i$ must be lattice polyhedra with $\tail(\kP_i)=\tail(\kP)$.

%%%%%%%%%%%%%%%%%%%%%%%%%%%%%%%%%%%%%%%%%%%%%%%%%%%%%%%%%%%%%%%%%%%%%%%%%%%
\subsection{The Kodaira-Spencer map}
\label{KodairaSpencer}
Assume that $\kP=\kP_0+\ldots + \kP_\minks$ is any Minkowski decomposition with
$\tail(\kP_i)=\tail(\kP)$. For each vertex $w=\face(\kP,\kc)$ of $\kP$
we will denote the corresponding vertex $\face(\kP_i,\kc)$
by $w(\kP_i)=w_i\in\kP_i$. 
Note that it depends on $w$ alone, i.e.\ not on the
special choice of $\kc\in\tail(\kP)\dual$. Actually, the associated normal
cones,
% $\normal(w,\kP)$ and $\normal(w^i,\kP_i)$ 
i.e.\ the regions of those $\kc$ providing the desired vertex,
satisfy $\normal(w,\kP)\subseteq\normal(w_i,\kP_i)$ and
\[
\textstyle
\normal(w,\kP) = \normal(w_0,\kP_0)\cap\ldots\cap\normal(w_\minks,\kP_\minks).
\]
In accordance with Notation~\ref{not:verticesEdges} we write 
$\Vrtx(\kP)=\{v^1\dots,v^\kPm\}$ and $\cedges(\kP):=\{\aed_1,\ldots,\aed_\kPr\}$,
which gives rise to the $\RR$-vector space $\RR^\kPr\oplus\RR^\kPm$
with coordinates $(\bft,\bfs)$. We will define an evaluation
$\kodSp:\{0,\ldots,\minks\}\to
\RR_{\geqslant 0}^\kPr\oplus\RR^\kPm$ of the Minkowski summands.

%%%%%%%%%%%%%%%
%%%%%%%%%%%%%%%
\begin{definition}
\label{def-KodSpMap}
Let $Q$ with $\tail(Q)=\tail(\kP)$ be a Minkowski summand of $\kP$.
The \emph{Kodaira-Spencer} evaluation 
$\kodSp(Q)=\big(t(Q),s(Q)\big)\in
\RR_{\geqslant 0}^\kPr\oplus\RR^\kPm$ is defined by 
\[
\begin{array}{rcl}
t_d(Q) & := &
(\mbox{the dilation factor of the edge $d$ inside $Q$})\in [0,1]\subset\RR
\\[0.5ex]
s_v(Q) & := &
                      \begin{cases}
                        0 & \textup{~if~} v(Q)\in\kQuotD\\
                        1 & \textup{~if~} v(Q)\notin\kQuotD
                      \end{cases} 
\hspace{1em}\mbox{for any vertex $v\in\kP$}.
\end{array}
\]
\end{definition}

Recall from Subsection~\ref{spaceMinkSummands} that the dilation factor
means the non-negative scalar transforming an edge of $\kP$ into the
associated edge of $Q$, i.e.\ satisfying
$v^j(Q)-v^i(Q)=t_{ij}(Q)\cdot(v^j-v^i)$ for vertices $v^i,v^j\in\kP$.
Note that the values collected in $s(Q)\in\RR^\kPm$ 
do heavily depend
on the position of $Q$, i.e.\ in general, they do change after shifting $Q$
along a vector from $\kQuotRD\setminus\kQuotD$.
In particular, the Kodaira-Spencer map $\kodSp$ is, in general, neither
Minkowski-additive, nor is its image contained in the subspace 
\mbox{$\kMT(\kP)\subseteq\RR^\kPr\oplus\RR^\kPm$}
from Definition~\ref{def:TP}. 
Nevertheless, we have $\kodSp(\kP)=\oneone=[\kP]\in\kMT(\kP)$ and  $\kodSp(0)=\ku{0}$.

%%%%%%%%%%%%%%%
%%%%%%%%%%%%%%%
\begin{example}
\label{ex-kodSpNotInTP}
Take $\kP=[\frac{1}{2},\frac{3}{4}]$
from Example~\ref{ex-CQS}.\ref{item:CQS3} and decompose it as
\[
\textstyle
\kP_0+\kP_1=[0,\frac{1}{4}] + [\frac{1}{2},\frac{1}{2}].
\]
Using the coordinates $(t;s_1,s_2)$ of $\RR^3$,
the Kodaira-Spencer map yields 
\[
\kodSp(\kP)=(1;1,1),
\hspace{0.8em}
\kodSp(\kP_0)=(1;0,1),
\hspace{0.8em}\mbox{and}\hspace{0.8em}
\kodSp(\kP_1)=(0;1,1).
\]
While this is clearly not additive, both summands $\kodSp(\kP_i)$
do also miss $\kMT(\kP)$: Since both half open edges induced from
$\kP$ are \se, the equations for $\kodSp(\kP_i)$ involve $s_1=t=s_2$,
which is not satisfied. 
\end{example}

%%%%%%%%%%%%%%%%%%%%%%%%%%%%%%%%%%%%%%%%%%%%%%%%%%%%%%%%%%%%%%%%%%%%%%%%%%%
\subsection{The Kodaira-Spencer map for lattice friendly decompositions}
\label{KodairaSpencerLattFriend}
While the Kodaira-Spencer map $\kodSp=(\bft,\bfs)$ behaves rather wildly for 
general Minkowski decompositions, it turns out to be the right tool to 
reflect lattice friendly decompositions.

%%%%%%%%%%%%%%%
%%%%%%%%%%%%%%%
\begin{theorem}\label{thm-KodairaSpencerLattFriend}
Let $\kP=\kP_0+\ldots + \kP_\minks$ be a Minkowski decomposition with
$\tail(\kP_i)=\tail(\kP)$. Then this decomposition is lattice friendly
if and only if 
\[
\,\kodSp(\kP)=\oneone=\kodSp(\kP_0)+\ldots+\kodSp(\kP_\minks),
\]
and this is a decomposition inside $\kMTZ(\kP)$, i.e.\
for all summands we have $\kodSp(\kP_i)\in\kMTZ(\kP)
\subset\RR^\kPr\oplus\RR^\kPm$.
\end{theorem}

\begin{proof}
($\neht$)
For each vertex $w\in\kP$, we obtain a decomposition
$s_w(\kP)=s_w(\kP_0)+\ldots+s_w(\kP_\minks)$ inside $\NN$.
Since $s_w(\kP)\in\{0,1\}$, this means that there is at most one
index $\mu=\mu(w)$ such that $s_w(\kP_\mu)=1$. All remaining summands vanish,
and this translates directly into the condition of
Definition~\ref{def-latticeFriends}.
\\[1ex]
($\then$)
Assume that the decomposition $\kP=\kP_0+\ldots + \kP_\minks$ is lattice
friendly. 
\\[0.5ex]
{\em Step 1. }
Since there is never a problem with the dilation factors, 
let us focus on the $s$-parameters.
If $w\in\kP$ is a vertex, then there is at most one index
$\mu=\mu(w)$ such that $s_w(\kP_\mu)\neq 0$. Moreover, we know that
\[
s_w(\kP_\mu)=1 \iff s_w(\kP_\mu)\neq 0 ~\Longrightarrow~ w\notin\kQuotD
\iff s_w(\kP)\neq 0 \iff s_w(\kP)=1.
\]
This shows the formula $\,\kodSp(\kP)=\sum_{i=0}^\minks\kodSp(\kP_i)$.
Moreover, the integrality of the $s_w(\kP_i)$ is clear, too.
\\[0.8ex]
{\em Step 2. }
Assume that $[v,w]\leq\kP$ is a compact edge with
$[v,w]\cap\kQuotD=\emptyset$. We may choose an element
$\kc\in\tail(\kP)\dual\cap\kQuot$ such that 
$[v,w]=\face(\kP,\kc)$. Since
\[
\face(\kP,\kc) \;=\; \face(\kP_0,\kc)+\ldots +  \face(\kP_\minks,\kc),
\]
there is at least one summand $\face(\kP_\mu,\kc)$ lacking lattice points, too.
In particular, $v(\kP_\mu),w(\kP_\mu)\notin\kQuotD$, and
since the decomposition of $\kP$ is lattice friendly, $\kP_\mu$ is the only
summand with 
$v(\kP_\mu)\notin\kQuotD$ or $w(\kP_\mu)\notin\kQuotD$. Hence
\[
s_v(\kP_\mu)=1=s_w(\kP_\mu)
\hspace{0.9em}\mbox{and}\hspace{0.9em}
s_v(\kP_i)=0=s_w(\kP_i)
\hspace{0.5em}\mbox{for}\hspace{0.5em}
i\neq\mu.
\]
That is, the equation $s_v=s_w$ from the definition of $\kMT(\kP)$
is satisfied for all Minkowski summands.
\\[0.8ex]
{\em Step 3. }
Assume that $[v,w)$ is a \se\ half open edge of $\kP$. We are supposed to
check the equality $s_v=t:=t_{vw}$ for all Minkowski summands.
Since $v\notin\kQuotD$, we know that $s_v(\kP)=1$, i.e.\ there is exactly one
index $\mu=\mu(v)$ such that $s_v(\kP_\mu)=1$.
Since this means $s_v(\kP_i)=0$ for $i\neq\mu$,
it remains to show that $t(\kP_i)=0$ for these indices;
the equality $t(\kP_\mu)=1$ follows then automatically.
If we had $t(\kP_i)>0$, then the equality
\[
w(\kP_i) - v(\kP_i) = t(\kP_i)\cdot (w-v)\neq 0.
\]
would imply that $w(\kP_i)\neq v(\kP_i)$. On the other hand,
both vertices $w(\kP_i)$ and $v(\kP_i)$ are lattice points.
While this is clear for $v(\kP_i)$, we have to provide an extra argument
for $w(\kP_i)$: If $w\in\kQuotD$, then it is clear; if
$w\notin \kQuotD$, then it follows from the \se ness of $[v,w)$
that $[v,w]\cap\kQuotD=\emptyset$. Hence, we can use
the equation $s_v=s_w$ obtained in Step 2.
% \\
Now, since we know that $w(\kP_i),v(\kP_i)\in\kQuotD$ do not coincide, 
we get a lower bound for the lattice lengths
\[
\ell(w-v)\geq \ell\big(w(\kP_i)-v(\kP_i)\big)\geq 1,
\]
which, by Remark~\ref{rem:afterDefShortedge}, is
not possible for short half open edges.
\\[0.8ex]
{\em Step 4. }
So far, we have seen that $\kodSp(\kP_i)\in\kMT(\kP)$.
To show that $\kodSp(\kP_i)$ is integral,
i.e.\ that $\kodSp(\kP_i)\in\kMTZ(\kP)$, 
we are supposed to check, for all $i=0,\ldots,\minks$, that
\[
\kL_{vw}(\kP_i)\;=\;
t(\kP_i)\cdot(v-w) - s_v(\kP_i)\cdot v + s_w(\kP_i)\cdot w\in \kQuotD.
\]
For this, we rewrite
\[
\kL_{vw}(\kP_i) = \big(v(\kP_i) - s_v(\kP_i)\cdot v\big)
- \big(w(\kP_i) - s_w(\kP_i)\cdot w\big)
\]
and analyse the membership of $\kQuotD$ for both summands separately.
If $v\in\kQuotD$, then $v(\kP_i)\in \kQuotD$ and $s_v(\kP_i)=0$,
hence 
\[
v(\kP_i) - s_v(\kP_i)\cdot v\in\kQuotD.
\]
If $v\notin\kQuotD$, then we denote by $\mu=\mu(v)$ the unique index with 
$v(\kP_\mu)\notin \kQuotD$,
and the previous argument survives for $i\neq\mu$.
On the other hand, since $s_v(\kP_\mu)=1$,
\[
\textstyle
v(\kP_\mu) - s_v(\kP_\mu)\cdot v
\;=\; v(\kP_\mu) - v
\;=\;-\sum_{i\neq\mu} v(\kP_i) \in\kQuotD.
\]
The proof for the $w$-summand is the same, with a possibly different index $\mu=\mu(w)$.
\end{proof}

\begin{example}\label{ex-calComponents}
Let us continue our main example. The interval $\kP=[-\frac{1}{2},\frac{1}{2}]\subset\RR$ allows two
non-trivial, lattice friendly decompositions, namely
\[
\textstyle
[-\frac{1}{2},\frac{1}{2}] = [-\frac{1}{2},0] + [0,\frac{1}{2}] = [-\frac{1}{2},-\frac{1}{2}] + [0,1].
\]
Applying the Kodaira-Spencer map $\kodSp$, this decomposition looks like
\[
\textstyle
(1,1,1) = (\frac{1}{2},1,0) + (\frac{1}{2},0,1)=(0,1,1)+(1,0,0)
\hspace{0.8em}\mbox{inside }\,\kMTZ(\kP).
\]
According to Subsection~\ref{KodairaSpencerRevisited}, we can understand
these two decompositions, for  $i=1,2$,  as two linear maps $\kodSp_i:\ZZ^2\to\kMTZ(\kP)$;
the dual maps $\kodSp^*_i:\kMTZD(\kP)\to\ZZ^2$ are given by the matrices
\[
\textstyle
\left(\begin{array}{ccc}
\frac{1}{2} & 1 & 0\\
\frac{1}{2} & 0 & 1
\end{array}\right),~~~~
\textstyle
\left(\begin{array}{ccc}
0 & 1 & 1\\
1 & 0 & 0
\end{array}\right).
\]
The integrality of the target can be checked when $\kodSp^*_i$ are applied
to the generators of $\ktT$: 
\[
  s_1,~~~ s_2,~~~ A=t+\frac{1}{2}s_1-\frac{1}{2}s_2,~~~ B=t+\frac{1}{2}s_2-\frac{1}{2}s_1.
\]
 Then, $\kodSp^*_i$ yield two integral $(2\times 4)$-matrices mapping from 
$ \ZZ s_1 \oplus \ZZ s_2  \oplus \ZZ A \oplus \ZZ B$ to $\ZZ^2$:
\[
\textstyle
\kodSp^*_1=
\left(\begin{array}{cccc}
1 & 0 & 1 & 0\\
0 & 1 & 0 & 1
\end{array}\right),~~~~
\kodSp^*_2=
\left(\begin{array}{ccccc}
1 & 1 & 0 & 0\\
0 & 0 & 1 & 1
\end{array}\right).
\]
\end{example}

%%%%%%%%%%%%%%%%%%%%%%%%%%%%%%%%%%%%%%%%%%%%%%%%%%%%%%%%%%%%%%%%%%%%%%%%%%%%%
\subsection{Lattice friendly decompositions and the map $\psi$}
\label{latticeFriendlyPsi}
In Subsection~\ref{univMinkSum} we have defined
for every $\xi\in \kMTP(\kP)$ an associated Minkowski
summand $\kP_\xi=\psi(\xi,\kP)$. While this construction depends on the
choice of  a reference vertex $\vast$,  
we have seen in  (\ref{indep-refVertex})
that, for integral $\xi\in\kMTZ(\kP)$,  
this dependence involves only  lattice translations.
Furthermore, $\psi$  is linear in $\xi$, i.e.\ for $\xi,\xi'\in \kMTP(\kP)$
we have 
\[
\kP_\xi+\kP_{\xi'}=\kP_{\xi+\xi'}.
\]

%%%%%%%%%%%%%%%
%%%%%%%%%%%%%%%
\begin{remark}
\label{rem-oneone}
We do not always have the equality $\psi_\vast(\oneone,\kP)=\kP$, 
but we can be very precise about this:
\[
\kP_{\oneone}\neq \kP \iff \vast\in\kQuotD\setminus\{0\}.
\]
This issue can again be solved by a lattice translation of $\kP$.
Hence, in accordance with Convention~\ref{conv-refPtsA},
we can and will assume  that $\vast$ is chosen such that
$\kP_\oneone = \kP$.
\end{remark}

In Theorem~\ref{thm-KodairaSpencerLattFriend} we have seen how 
the Kodaira-Spencer map $\kodSp$ can detect whether a Minkowski decomposition
is lattice friendly or not.
The next result shows how the map $\psi(\kbb,\kP)=\kP_{\kbb}$ 
fits into this relation.

%%%%%%%%%%%%%%%
%%%%%%%%%%%%%%%
\begin{theorem}
\label{thm-latticeDecompPsi}
Let $\xi_0,\ldots,\xi_\minks\in \kMTP(\kP)$ with 
$\xi_0+\ldots+\xi_\minks=\oneone$. 
Then, the following three conditions are equivalent:
\begin{enumerate}[label=(\roman*)]
\item\label{item:tldp1}
$\,\xi_0,\ldots,\xi_\minks\in \kMTZ(\kP)$,
\vspace{0.5ex}
\item\label{item:tldp2} for each vertex $w\in\kP$ and index $i\in\{0,\ldots,\minks\}$
we have
\vspace{-0.2ex}
\[
w(\kP_i)\notin\kQuotD \iff w\notin\kQuotD 
\hspace{0.5em}\mbox{and}\hspace{0.5em}
s_w(\kP_{\xi_i})=1,
\hspace{0.8em}\mbox{and}
\vspace{0.0ex}
\]
\item \label{item:tldp3}
the decomposition $\kP_{\xi_0}+\ldots+\kP_{\xi_\minks}=\kP$ 
is lattice friendly with $\kodSp(\kP_{\xi_i}) = \xi_i\,$
for $i=0,\ldots,\minks$. 
\end{enumerate}
\end{theorem}

%%%%%%%%%%%%%%%
%%%%%%%%%%%%%%%
\begin{proof}
{\ref{item:tldp1}$\then$\ref{item:tldp2}: }
Let $v\in\kP$ be a vertex and choose a path 
$\vast=v^0,v^1,\ldots,v^k=v$ along the compact edges of $\kP$.
Denoting $\vect{\aed_i}:=v^i-v^{i-1}$, we know from Subsection~\ref{univMinkSum}
that for each $\xi\in\kMTP(\kP)$
\[ 
\textstyle
\psi_{\vast}(\xi,v) 
\;=\; s_\vast(\xi)\cdot \vast +\sum_{i=1}^k t_i(\xi)\cdot \vect{\aed_i}
\;=\; \sum_{i=1}^k \kL_{i,\,i-1}(\xi) + s_v(\xi)\cdot v.
\]
As $\xi\in\kMTZ(\kP)$ implies $\kL_{i,\,i-1}(\xi)\in\kQuotD$
for every $i$, the equivalence in (ii) becomes evident. 
\\[1ex]
{\ref{item:tldp2}$\then$\ref{item:tldp3}: }
Since $\kP_{\oneone}=\kP$,
we obtain a Minkowski decomposition $\kP_{\xi_0}+\ldots+\kP_{\xi_\minks}=\kP$.
To check that it is lattice friendly, it suffices to check that for
every vertex $w\in\kP$ we have at most one index
$\mu\in\{0,\ldots,\minks\}$ with $w(\xi_\mu)\notin\kQuotD$.
However, this follows directly from $s_w\big(\xi\in\kMTZ\big)\in\NN$,
from $\xi_0+\ldots+\xi_\minks=\oneone$, hence from
$s_w(\xi_0)+\ldots+s_w(\xi_\minks)=1$ and~\ref{item:tldp2}.
\\[0.5ex]
Finally, the $t$-coordinates of $\kodSp(\kP_\xi)$ and $\xi$ are equal by
definition. For the equality of the $s$-coordinates we use again
$\xi_0+\ldots+\xi_\minks=\oneone$ and (ii) in a straightforward manner.
\\[1ex]
{\ref{item:tldp3}$\then$\ref{item:tldp1}: }
This follows from the direction ($\then$) in Theorem~\ref{thm-KodairaSpencerLattFriend}.
\end{proof}

%%%%%%%%%%%%%%%
%%%%%%%%%%%%%%%
\begin{remark}
\label{rem-kodSpPsi}
\begin{enumerate}[label={(\roman*)}]
\item On the one hand, since there are many non-lattice choices for $\xi$ which produce the same
polyhedron~$\kP_\xi$, we cannot expect that $\xi$ can be recovered from
$\kP_\xi$. In particular, the equality $\kodSp(\kP_{\xi}) = \xi$
cannot be true in general. Hence, it is not possible to erase the phrase
``with $\kodSp(\kP_{\xi}) = \xi$'' from~\ref{item:tldp3} of the previous theorem.

\item On the other hand, every lattice shift of a Minkowski summand produces 
the same value of $\kodSp$.
So we cannot expect to obtain $Q=\kP_{\kodSp(Q)}$ in general, either. 
\item Despite the negative claims above, we can consider the following
two sets:
\[
\begin{array}{l}
A:=\{\mbox{polyhedra }Q\subseteq\kQuotRD \mbox{ with }
\tail(Q)=\tail(\kP) \mbox{ such that there is a polyhedron } Q' 
\\
\hspace*{11em}
\mbox{ providing a lattice friendly decomposition } \kP=Q+Q'\}
\hspace{1em}\mbox{and}
\\[1ex]
B:= \{\xi\in \kMTP(\kP)\cap\kMTZ(\kP)\kst \oneone-\xi\in\kMTP(\kP)\}.
\end{array}
\]
Then, dividing out integral translations, it follows from the theorems
\ref{thm-KodairaSpencerLattFriend} and~\ref{thm-latticeDecompPsi}
that the two maps
$\kodSp:A/\kQuotD\to B$ and $\psi:B\to A/\kQuotD$ are mutually inverse.
\end{enumerate}
\end{remark}

%%%%%%%%%%%%%%%%%
%%%%%%%%%%%%%%%%%
\begin{example} 
\label{ex-negativeS}
While the construction $\xi\mapsto \kP_\xi$ 
from Subsection~\ref{univMinkSum} does not make use of the
non-negativity of $s$ in $\xi=(t,s)$, the assumption
$\xi\in\kMTP(\kP)$ becomes really important for 
Theorem~\ref{thm-latticeDecompPsi}. To illustrate this, take
$P=[-\frac{1}{3},\frac{1}{4}]$ with $v_1=-\frac{1}{3}$ and $v_2=\frac{1}{4}$. 
So this is not a \se\ edge (none of the two half-open edges is), 
and the lattice conditions for $\xi=(t,s_1,s_2)$ are
\[
\textstyle
s_1, s_2 \in\ZZ
\hspace{1em}\mbox{and}\hspace{1em}
\frac{7}{12} t - \frac{1}{3}s_1-\frac{1}{4}s_2 \in\ZZ. 
\]
Choose $v_\ast=v_1$ and consider 
\[
\textstyle
\xi=(\frac{1}{7},1,-1),\hspace{1em}\xi'=(\frac{6}{7},0,2).
\]
We see that $\xi,\xi'\in\kMTZ(\kP)$ with $\xi+\xi'=(1,1,1)$, but that 
\[
\textstyle
\kP_\xi=[-\frac{1}{3},-\frac{1}{4}]
\hspace{1em}\mbox{and}\hspace{1em}
\kP_{\xi'}=[0,\frac{1}{2}],
\]
 provides a non-lattice-friendly Minkowski decomposition of $\kP$.
\end{example}

\subsection{The Kodaira-Spencer map revisited}
\label{KodairaSpencerRevisited}
In Definition~\ref{def-latticeFriends}
we have introduced the notion of lattice friendly decomposition $\kP=\kP_0+\ldots+\kP_\minks$.
In \cite[(3.2)]{ka-flip}, this notion was used to construct a
$\minks$-parameter family, i.e.\ a deformation 
$\tX\to \A_k^\minks$ of the associated toric singularity $X$.
Its total space was built from the Cayley product
mentioned in Remark~\ref{rem-propsCQ}.\ref{item:rem-propsCQ}.
Actually, similarly to $X=\toric(\sigma)$, one defines it as
$\tX=\toric(\tsigma)$ with $\tsigma:=\cone(\kP_0*\ldots*\kP_\minks)$.
\\[1ex]
In \cite[(3.3)+(3.4)]{ka-flip} it was shown that the Kodaira-Spencer 
map of this construction is exactly the map $\kodSp$ we have defined
in Subsection~\ref{KodairaSpencer} -- this is why we have called
it like this even in the purely discrete, i.e.\ non-algebraic setup.
In \cite[Proposition 4.3]{m} we will connect the notion of a free pair to flatness, from which it follows that  the corresponding
inclusion 
$\NN^{\minks+1}\hookrightarrow \tsigma\dual\cap(\kQuot\oplus\NN^{\minks+1})$
is a \cocartesian~\aextension\  of 
$\NN\hookrightarrow \sigma\dual\cap(\kQuot\oplus\NN)$.
In particular, denoting by $(\ktT,\ktS)$ the initial extension from
Theorem~\ref{th-initialObject}, then this is induced from 
a semigroup homomorphism $\kMTZD(\kP)\supseteq\ktT\to\NN^{\minks+1}$. 
The dual map $\ZZ^{\minks+1}\to\kMTZ(\kP)$ equals $\kodSp$, and the
fact that its target is $\kMTZ(\kP)\subset\kMT(\kP)$
illustrates Theorem~\ref{thm-KodairaSpencerLattFriend}.

\bibliographystyle{amsalpha}
\bibliography{GBUbib}

\end{document}